\documentclass[12pt]{amsart}
\usepackage[margin=1in]{geometry}
\usepackage{fancyhdr}
\usepackage{cite}
\usepackage{amsmath,amsthm,amsfonts,amssymb,bm}
\usepackage{xcolor}
\colorlet{RED}{red}
\usepackage{setspace}
\usepackage{hyperref}
\newtheorem{thm}{Theorem}[section]
\newtheorem{lem}[thm]{Lemma}
\newtheorem{prop}[thm]{Proposition}

\theoremstyle{definition}

\theoremstyle{remark}
\newtheorem{rem}[thm]{Remark}

\numberwithin{equation}{section}

\makeatletter
\def\tagform@#1{\maketag@@@{\ignorespaces#1\unskip\@@italiccorr}}
\let\orgtheequation\theequation
\def\theequation{(\orgtheequation)}
\makeatother

\allowdisplaybreaks
\begin{document}
	
	\title{Global Stability of 3D Compressible Non-Resistive MHD: Hidden Damping and Rational Background Fields}
	\author{Lin-An Li$^1$, Qian Li$^2$, Jiahong Wu$^3$, Xiaojing Xu$^1$}
	\address{$^1$School of Mathematical Sciences, Beijing Normal University and Laboratory of Mathematics and Complex Systems, Ministry of Education, Beijing 100875, People's Republic of China}
	\email{linanli@amss.ac.cn}
	
	\address{$^2$School of Mathematics, China University of Mining and Technology, Xuzhou 221116, China.}	
	\email{li-qian@mail.bnu.edu.cn}
	
	\address{$^3$Department of Mathematics, University of Notre Dame, Notre Dame, IN 46556, USA}	
	\email{jwu29@nd.edu}
	
	\address{$^1$School of Mathematical Sciences, Beijing Normal University and Laboratory of Mathematics and Complex Systems, Ministry of Education, Beijing 100875, P.R. China}	
	\email{xjxu@bnu.edu.cn}

	\begin{abstract}
		We prove the global well-posedness and nonlinear stability of classical solutions to the three-dimensional compressible viscous, non-resistive MHD system on $\mathbb{T}^3$ near the equilibrium $(1,\mathbf{0},\mathbf{e}_3)$, for small $x_3$-symmetric perturbations and without any Diophantine condition on the background magnetic field. The central obstruction is the $x_3$-independent sector, in which the density and the magnetic field possess no dissipation, no damping, and no decay. We overcome it by exhibiting a hidden wave structure for the exact total pressure $\mathcal{D}=P(1+a)+B_3+\frac12|\mathbf{B}|^2$, the averaged pair $(\overline{\mathbf{u}},\overline{\mathcal{D}})$ obeys a closed system, and $\overline{\mathcal{D}}$ satisfies a strongly damped wave equation whose principal wave part propagates at the fast magnetosonic speed and whose non-parabolic branch damps at a rate that involves no gain of derivatives. This hidden damping substitutes for the missing magnetic dissipation, and simultaneously absorbs the magnetic pressure $\frac12\nabla|\mathbf{B}|^2$, the main obstruction created by compressibility. Together with a damped wave structure for the oscillatory sector and space-time weighted energy functionals with shifted time weights, this yields global existence, uniform stability, and explicit polynomial decay rates.
	\end{abstract}
	\maketitle
	
	
	
	\section{Introduction}
	
	\subsection{The system and the equilibrium}
	The viscous magnetohydrodynamic (MHD) equations govern the motion of an electrically conducting fluid interacting with a magnetic field, and are central to the modeling of astrophysical plasmas, space physics, and controlled fusion devices. In many strongly collisional plasmas the viscosity is significant while the magnetic diffusivity is comparatively weak; neglecting the latter idealizes the fluid as a perfect conductor \cite{Cabannes1970,Landu60} and leads to the compressible viscous, non-resistive system
	\begin{equation}\label{equ cmhd0}
		\left\{\begin{aligned}
			&\partial_t\rho+\operatorname{div}(\rho\mathbf{u})=0,\qquad t>0,\ x\in \mathbb{T}^3,\\
			&\rho\partial_t\mathbf{u}+ \rho\mathbf{u}\cdot\nabla  \mathbf{u}+\nabla P -\mu\Delta \mathbf{u}-(\lambda+\mu)\nabla\operatorname{div}\mathbf{u}=(\nabla \times \mathbf{b})\times \mathbf{b},\\
			&\partial_t\mathbf{b}+\mathbf{u}\cdot \nabla \mathbf{b}-\mathbf{b}\cdot \nabla \mathbf{u}+\mathbf{b}\operatorname{div}\mathbf{u}=0,\\
			&\nabla \cdot\mathbf{b}=0,
		\end{aligned}\right.
	\end{equation}
	posed on the torus $\mathbb{T}^3=[-\frac12,\frac12]^3$. Here $\rho$, $\mathbf{u}$ and $\mathbf{b}$ denote the density, the velocity and the magnetic field, and the viscosity coefficients satisfy the physical constraints
	\begin{align*}
		\mu > 0,\qquad 2\mu + 3\lambda \geq 0.
	\end{align*}
	The pressure $P = P(\rho)$ is a smooth function of the density satisfying
	\begin{align*}
		P'(\rho) > 0\quad \text{and}\quad P'(\rho_a) =1,
	\end{align*}
	where $\rho_a=\int_{\mathbb{T}^3}\rho \,dx$ denotes the mean density. Since $\rho_a$ is conserved in time, we set $\rho_a=1$ without loss of generality.
	
	The mathematical interest of \eqref{equ cmhd0} lies in the absence of magnetic diffusion. The equation for $\mathbf{b}$ is purely transport-stretching. It has no regularizing mechanism of its own, and it is not known whether smooth solutions of the non-resistive system remain regular for general data. Stability near a nontrivial equilibrium is therefore governed by a subtle question: how much of the velocity's dissipation can be transferred to the magnetic field through the linear coupling with a background field, and in which directions.
	
	We study perturbations of the constant state $(\rho,\mathbf{u},\mathbf{b})=(1,\mathbf{0},\mathbf{e}_3)$ with $\mathbf{e}_3=(0,0,1)^{\top}$. Writing $(a,\mathbf{u},\mathbf{B})=(\rho-1,\mathbf{u},\mathbf{b}-\mathbf{e}_3)$, the system becomes
	\begin{equation}\label{equ pcmhd}
		\left\{\begin{aligned}
			&\partial_t a +\mathbf{u}\cdot \nabla a =-\operatorname{div} { \mathbf{u}}-a\operatorname{div}\mathbf{u}, \qquad t>0,\ x\in \mathbb{T}^3,\\
			&\rho \partial_t\mathbf{u}+\rho \mathbf{u}\cdot \nabla \mathbf{u}-\mu\Delta \mathbf{u}-(\lambda+\mu)\nabla \operatorname{div}\mathbf{u}+\nabla P\\
			&\hspace{4cm}=\partial_{x_3}\mathbf{B}-\nabla B_3+\mathbf{B}\cdot\nabla \mathbf{B}- \frac12\nabla|\mathbf{B}|^2,\\
			&\partial_t\mathbf{B}+\mathbf{u}\cdot \nabla \mathbf{B}=\partial_{x_3} \mathbf{u}-\mathbf{e}_3\operatorname{div} \mathbf{u}+\mathbf{B}\cdot \nabla \mathbf{u}-\mathbf{B}\operatorname{div}\mathbf{u},\\
			&\operatorname{div}\mathbf{B}=0,
		\end{aligned}\right.
	\end{equation}
	with initial data $(a,\mathbf{u},\mathbf{B})|_{t=0}=(a_0,\mathbf{u}_0,\mathbf{B}_0)$. The linear coupling between $\mathbf{u}$ and $\mathbf{B}$ is carried by the single derivative $\partial_{x_3}$, and this is the entire source of magnetic stabilization available to us.
	
	Following Pan, Zhou and Zhu \cite{PanZhouZhu18}, we impose an $x_3$-symmetry on the data. Denoting $x_h=(x_1,x_2)$, $\mathbf{u}_h=(u_1,u_2)^{\top}$ and $\mathbf{B}_h=(B_1,B_2)^{\top}$, we assume
	\begin{equation}\label{as symmetry}
		\begin{aligned}
			&  a_0(x_h, x_3),\ \mathbf{u}_{0,h}(x_h, x_3) \text{ and } B_{0,3}(x_h, x_3) \text{ are even and periodic in } x_3, \\
			&  u_{0,3}(x_h, x_3) \text{ and } \mathbf{B}_{0,h}(x_h, x_3) \text{ are odd and periodic in } x_3,
		\end{aligned}
	\end{equation}
	together with the zero-mean conditions
	\begin{equation}\label{as zero mean}
		\int_{\mathbb{T}^3} a_0 \, dx = 0, \quad \int_{\mathbb{T}^3} \rho_0\mathbf{u}_0 \, dx = 0, \quad \int_{\mathbb{T}^3} \mathbf{B}_0 \, dx = 0.
	\end{equation}
	Both \eqref{as symmetry} and \eqref{as zero mean} are propagated by the flow, by uniqueness of classical solutions. Their role is to force $\mathbf{B}_h(x_h,\cdot)$ and $u_3(x_h,\cdot)$ to have vanishing $x_3$-average, so that these two quantities coincide with their own oscillatory parts. The anisotropic Poincar\'e inequality $\|f\|_{L^2}\lesssim\|\partial_{x_3}f\|_{L^2}$ then applies to $\mathbf{B}_h$ and $u_3$ themselves, and not merely to their oscillations.
	
	\subsection{Main result}
	\begin{thm}\label{thm1}
		Let $s>7$ be an integer and $\sigma\in (0,1/2)$. Suppose that the initial data $(a_0,\mathbf{u}_0,\mathbf{B}_0)\in H^s(\mathbb{T}^3)$ satisfies $\nabla\cdot \mathbf{B}_0=0$ together with \eqref{as symmetry} and \eqref{as zero mean}. There exists $\varepsilon>0$ such that if
		\[
		\|(a_0,\mathbf{u}_0,\mathbf{B}_0) \|_{H^{s}}\leq \varepsilon,
		\]
		then \eqref{equ pcmhd} admits a unique global classical solution $(a,\mathbf{u},\mathbf{B})$. Moreover, there is $C>0$ such that for all $t>0$,
		\[
		\|(a,\mathbf{u},\mathbf{B}) (t)\|_{H^{s-1}}\leq C\varepsilon,
		\]
		and, for $i=1,2,3$,
		\[
		\|(\widetilde{a},\mathbf{u},\widetilde{\mathbf{B}})(t) \|_{H^{s-i}}\leq C\varepsilon (1+t)^{\frac{\sigma-i}{2}},
		\]
		where $\widetilde{f}=f-\int_{\mathbb{T}}f\,dx_3$ denotes the oscillatory part of $f$ in the $x_3$-variable.
	\end{thm}
	
	This result has three important features.
	
	First, the background field is the rational vector $\mathbf{e}_3$. Apart from the incompressible work of Pan, Zhou and Zhu \cite{PanZhouZhu18} discussed below, every global stability result for three-dimensional MHD on a periodic domain with only one dissipation, in either the incompressible or the compressible regime, requires the background field $\mathbf{n}$ to obey a Diophantine condition
	\begin{align}\label{Dio con}
		|\mathbf{n}\cdot \mathbf{k}|\geq  c|\mathbf{k}|^{-r},\qquad \forall\, \mathbf{k}\in\mathbb{Z}^3\setminus\{\mathbf{0}\},
	\end{align}
	which excludes $\mathbf{e}_3$ and, more generally, every rational direction. To the best of our knowledge, Theorem \ref{thm1} is the first stability result for the three-dimensional compressible non-resistive MHD system with a rational background field.
	
	Second, the decay is anisotropic and, in a precise sense, optimal for the structure at hand. The velocity decays in full, and so does the oscillatory part of $(a,\mathbf{B})$; the averages $(\overline{a},\overline{\mathbf{B}})$ are merely bounded uniformly in time. This is not a defect of the method: the $x_3$-independent sector of the linearized system supports steady states, and no decay for $(\overline{a},\overline{\mathbf{B}})$ can be expected.
	
	Third, the loss of one derivative in the uniform bound ($H^{s-1}$ rather than $H^{s}$) is genuine and is compensated by a slowly growing top-order bound $\|(a,\mathbf{u},\mathbf{B})(t)\|_{H^{s}}\leq C\varepsilon(1+t)^{\sigma/2}$; see Remark \ref{rem hs}.
	
	\subsection{Background}
	For the two-dimensional compressible non-resistive system in $\mathbb{R}^2$, Wu and Wu \cite{WuWu17} obtained global small solutions with explicit decay rates, under assumptions later relaxed by Zhu \cite{Zhu26}; in the periodic setting Wu and Zhu \cite{WuZhu22} exploited the weak smoothing induced by the background field $\mathbf{e}_2$, and Dong, Wu and Zhai \cite{DongWuZhai23} treated the $2\frac12$-dimensional case. None of these requires symmetry or Diophantine hypotheses. The two-dimensional geometry is genuinely more forgiving: a single background direction already controls a codimension-one set of frequencies.
	
	In three dimensions the situation changes. The stabilizing operator is $\mathbf{n}\cdot\nabla$, and unless $\mathbf{n}$ is Diophantine there are arbitrarily high frequencies $\mathbf{k}$ along which $\mathbf{n}\cdot\mathbf{k}$ is small or zero, so the magnetic field is not controlled by $\mathbf{n}\cdot\nabla\mathbf{B}$ in any Sobolev norm. Under \eqref{Dio con} one recovers
	\begin{align}\label{Dio poincare}
		\|\mathbf{B}\|_{H^k}\lesssim   \|\mathbf{n}\cdot \nabla \mathbf{B}\|_{H^{k+r}},
	\end{align}
	at the price of $r$ derivatives, and this has been the engine of the subject: Chen, Zhang and Zhou \cite{ChenZhangZhou22} for the incompressible system with either viscosity or resistivity alone, Wu and Zhai \cite{WuZhai23} for the compressible non-resistive system, Li and Qiao \cite{LiQiao25} for the inviscid resistive isentropic system, and Wu, Xu and Zhai \cite{WuXuZhai25} in the non-isentropic inviscid regime, where a stabilizing wave structure among the fluid, the temperature and the magnetic field is uncovered. The single exception is the incompressible work of Pan, Zhou and Zhu \cite{PanZhouZhu18}, which replaces \eqref{Dio con} by the symmetry \eqref{as symmetry}. The present paper carries that program into the compressible regime, where, as we now explain, the mechanism of \cite{PanZhouZhu18} breaks down and must be replaced.
	
		We refer the reader to \cite{HuWang10,LiXuZhang13,WuZhangZou21,ChenWang02,Kawashima84,HongGY117,TanZ18,JiangF19,JiangF19Parker,WangXin22,HuXP14}  for further results on compressible MHD, and to \cite{ZhouZhu18,WeiZhang20,WuZhu21,XiaoXinWu09,AbidiZhang17,LiTanYin17,Chemin16,Fefferman14,Fefferman17,LinXu15,RenWu14,Wang19,XuZhang15,CaiHan26,LinZhang14} for the incompressible system with partial or full dissipation.
	
	\subsection{The degeneracy of the rational background field}
	It is worth making precise what is lost when \eqref{Dio con} is abandoned.
	
	In \cite{WuZhai23}, with $\mathbf{n}$ Diophantine, the effective unknowns are
	\begin{align*}
		d=a+\mathbf{n}\cdot \mathbf{B}, \qquad \mathbf{G}=\mathbb{Q}\mathbf{u}-\frac{1}{\lambda+2\mu}\Delta^{-1}\nabla d,\qquad \mathbb{Q}=\nabla\Delta^{-1}\operatorname{div},
	\end{align*}
	and the triple $(d,\mathbf{G},\mathbf{n}\cdot\nabla\mathbf{B})$ carries a hidden damped wave structure. Combined with \eqref{Dio poincare}, this yields
	\begin{align*}
		\|a\|_{H^k}\lesssim \|d\|_{H^k}+\|\mathbf{n}\cdot \nabla\mathbf{B}\|_{H^{k+r}},
	\end{align*}
	so that damping for the two combined quantities is upgraded to damping for $a$ and $\mathbf{B}$ separately. The Diophantine condition is exactly what performs this upgrade.
	
	For $\mathbf{n}=\mathbf{e}_3$ the upgrade fails, and it fails on a specific sector. Decompose any function into its $x_3$-average and its oscillation,
	\begin{align*}
		\overline{f}(t,x_h)=\int_{\mathbb{T}}f(t,x_h,x_3)\,dx_3,\qquad \widetilde{f}=f-\overline{f},
	\end{align*}
	so that $\overline{f}$ collects the frequencies with $k_3=0$ and $\widetilde{f}$ the rest. On the oscillatory sector the anisotropic Poincar\'e inequality substitutes for \eqref{Dio poincare} at no loss of derivatives, since $\|\widetilde{\mathbf{B}}\|_{H^k}\lesssim\|\partial_{x_3}\mathbf{B}\|_{H^k}$. On the averaged sector, by contrast, $\partial_{x_3}$ annihilates everything, the linear coupling between $\overline{\mathbf{u}}$ and $\overline{\mathbf{B}}$ vanishes identically, and $\overline{a}$ and $\overline{\mathbf{B}}$ inherit no dissipation, no damping, and no decay whatsoever. The combined quantity $d$ is of no help here, because on this sector its damping cannot be decoupled into damping for $\overline{a}$ and $\overline{B_3}$ individually.
	
	Two further obstructions, both absent from \cite{PanZhouZhu18}, are created by compressibility. The magnetic pressure $\frac12\nabla|\mathbf{B}|^2$ can no longer be discarded by the incompressible cancellation, and the density-dependent nonlinearities $I(a)=\frac{a}{1+a}$ multiply the highest-order terms. Since $\overline{a}$ and $\overline{\mathbf{B}}$ do not decay, any estimate that treats these terms perturbatively is doomed: there is no smallness in time to spend.
	
	\subsection{The total pressure \texorpdfstring{$\mathcal{D}$}{D}: a hidden wave structure and a hidden damping}\label{subsec D}
	The central discovery of this paper is that the averaged sector, apparently devoid of any stabilizing mechanism, in fact carries one and that it becomes visible only in the variable
	\begin{align}\label{def D}
		\mathcal{D}=P(1+a)+B_3+\frac12|\mathbf{B}|^2.
	\end{align}
	Linearizing \eqref{def D} around the equilibrium and using $P'(1)=1$ gives $\mathcal{D}=P(1)+a+B_3+O(\text{quadratic})$, so $\mathcal{D}$ is a nonlinear completion of the combined quantity $d=a+\mathbf{n}\cdot\mathbf{B}$ of \cite{WuZhai23} in the case $\mathbf{n}=\mathbf{e}_3$. The point is that the completion is not arbitrary: among all such completions, \eqref{def D} is the one that is simultaneously the exact total pressure, the exact pressure force in the momentum equation, and the variable in which the averaged sector closes. We make these three assertions precise in turn.
	
	\medskip
	\noindent\textbf{(i) $\mathcal{D}$ is the exact total pressure.} Since $\mathbf{b}=\mathbf{e}_3+\mathbf{B}$,
	\begin{align*}
		\frac12|\mathbf{b}|^2=\frac12+B_3+\frac12|\mathbf{B}|^2,
		\qquad\text{hence}\qquad
		\mathcal{D}=P(\rho)+\frac12|\mathbf{b}|^2-\frac12 .
	\end{align*}
	Thus $\mathcal{D}$ is, up to an additive constant, the sum of the thermodynamic and magnetic pressures, the quantity whose gradient is the full pressure force in \eqref{equ pcmhd}$_2$:
	\begin{align*}
		\nabla P+\nabla B_3+\frac12\nabla|\mathbf{B}|^2=\nabla\mathcal{D}.
	\end{align*}
	We emphasize that $\mathcal{D}$ is defined by the exact nonlinear expression \eqref{def D}, not by its linearization. Consequently the magnetic pressure $\frac12\nabla|\mathbf{B}|^2$, the very term that defeats the incompressible argument of \cite{PanZhouZhu18}, is not estimated at all: it is absorbed into the structure. The obstruction becomes part of the mechanism.
	
	\medskip
	\noindent\textbf{(ii) The averaged pair $(\overline{\mathbf{u}},\overline{\mathcal{D}})$ closes.} Using the equations for $a$ and $\mathbf{B}$ in \eqref{equ pcmhd} and the normalization $P'(1)=1$, a direct computation gives
	\begin{equation}\label{eq mathcalD}
		\left\{\begin{aligned}
			\partial_t   \mathbf{u} =&
			\mu\Delta    \mathbf{u} +(\lambda+\mu)\nabla \operatorname{div}\mathbf{u} +\partial_{x_3}\mathbf{B}-  \nabla  \mathcal{D}+F_{\mathbf{u}}    ,\\
			\partial_t\mathcal{D}=& \partial_{x_3}u_3-2\operatorname{div}\mathbf{u}+F_{\mathcal{D}},
		\end{aligned}\right.
	\end{equation}
	where $F_{\mathbf{u}}$ and $F_{\mathcal{D}}$ are quadratic. The coefficient $2$ is not accidental: it is
	\begin{align*}
		2=P'(1)+|\mathbf{e}_3|^2=c_s^2+v_A^2,
	\end{align*}
	the sum of the squared sound and Alfv\'en speeds.
	
	The system \eqref{eq mathcalD} is not closed, since $\partial_{x_3}\mathbf{B}$ appears in \eqref{eq mathcalD}$_1$. The decisive point is what happens under averaging. Since $\overline{\partial_{x_3}f}=0$ for any periodic $f$, both linear terms involving $\partial_{x_3}$ disappear, and \eqref{eq mathcalD} reduces to a system in which the magnetic field is entirely absent at linear order:
	\begin{equation}\label{eq Dbar}
		\left\{\begin{aligned}
			\partial_t   \overline{\mathbf{u}} -\mu\Delta    \overline{\mathbf{u}} -(\lambda+\mu)\nabla \operatorname{div}\overline{\mathbf{u}} +  \nabla  \overline{\mathcal{D}} &=\overline{F_{\mathbf{u}}}    ,\\
			\partial_t\overline{\mathcal{D}}+2\operatorname{div}\overline{\mathbf{u}}&= \overline{F_{\mathcal{D}}}.
		\end{aligned}\right.
	\end{equation}
	In other words, in the variable $\mathcal{D}$ the intractable averaged sector of the MHD system becomes a barotropic compressible Navier--Stokes system with sound speed $\sqrt{2}$. This is the structural gain, and it is invisible in the unknowns $(\overline{a},\overline{\mathbf{B}})$, for which no such closed system exists.
	
	\medskip
	\noindent\textbf{(iii) The hidden wave structure and its damping.} Applying $\operatorname{div}$ to \eqref{eq mathcalD}$_1$, the magnetic term drops out exactly because $\operatorname{div}\mathbf{B}=0$; averaging the resulting identity and substituting \eqref{eq Dbar}$_2$, we obtain the strongly damped wave equation
	\begin{align}\label{eq Dwave}
		\partial_t^2\overline{\mathcal{D}}-2\Delta\overline{\mathcal{D}}-\nu\Delta\partial_t\overline{\mathcal{D}}=\partial_t\overline{F_{\mathcal{D}}}{-\nu\Delta\overline{F_{\mathcal{D}}}}-2\operatorname{div}\overline{F_{\mathbf{u}}},
		\qquad \nu=\lambda+2\mu>0.
	\end{align}
	This is a Kelvin--Voigt damped wave equation whose principal wave part carries the speed $\sqrt{2}=\sqrt{c_s^2+v_A^2}$, which is precisely the fast magnetosonic speed for propagation perpendicular to $\mathbf{e}_3$, the only direction of propagation available on the sector $k_3=0$, where the Alfv\'en and slow magnetosonic modes degenerate. The variable $\mathcal{D}$ diagonalizes the surviving mode.
	
	The dispersion relation of the linear part of \eqref{eq Dwave} is $\tau^2+\nu|\mathbf{k}|^2\tau+2|\mathbf{k}|^2=0$, with roots
	\begin{align*}
		\tau_{\pm}(\mathbf{k})=\frac{-\nu|\mathbf{k}|^2\pm\sqrt{\nu^2|\mathbf{k}|^4-8|\mathbf{k}|^2}}{2},
		\qquad
		\tau_{-}\sim-\nu|\mathbf{k}|^2,\quad \tau_{+}\to-\frac{2}{\nu}\ \ (|\mathbf{k}|\to\infty).
	\end{align*}
	The branch $\tau_{-}$ is the parabolic branch inherited from the viscosity. The branch $\tau_{+}$ is the interesting one: it converges to the finite negative limit $-2/\nu$, so it damps at a rate that does not degenerate as $|\mathbf{k}|\to\infty$ and involves no gain of derivatives. This is the hidden damping of $\mathcal{D}$, it is not dissipation, it produces no smoothing, and it is therefore compatible with a system in which the magnetic field is transported without diffusion. The limiting rate $2/\nu=(c_s^2+v_A^2)/(\lambda+2\mu)$ is the same as the damping rate of the combined quantity $d$ on the oscillatory sector (see \eqref{eq dG} below). This is not a coincidence but a consequence of the relation between $\mathcal{D}$ and $d$ recorded above: the two sectors relax at a common rate, through the linearly equivalent quantity, by two structurally different mechanisms.
	
	At the nonlinear level we realize \eqref{eq Dwave} not through the second-order equation but through the equivalent cross-term identity: differentiating $\int_{\mathbb{T}^3}\nabla^{s-2}\overline{\mathbf{u}}\cdot\nabla^{s-1}\overline{\mathcal{D}}\,dx$ in time and substituting \eqref{eq Dbar} produces the damping term $\|\overline{\mathcal{D}}\|^2_{\dot{H}^{s-1}}$ with a favorable sign, controlled by $\|\nabla\overline{\mathbf{u}}\|^2_{\dot{H}^{s-1}}$ and quadratic errors. This is carried out in Lemma \ref{est bE1}.
	
	\medskip
	\noindent\textbf{Why the average, and not $\mathcal{D}$ itself.} The term $\partial_{x_3}u_3$ in \eqref{eq mathcalD}$_2$ is a linear coupling of the same order as the damping it competes with; estimating it directly forces an additional constraint relating $\mu$ and $\lambda$. Averaging annihilates it identically. This is a second, independent reason for the average-oscillation decomposition. It is not merely a device for separating frequencies, but the operation that makes the total pressure obey a clean damped wave equation with no restriction on the viscosity coefficients. Accordingly, it is $\overline{\mathcal{D}}$, not $\mathcal{D}$, that enters our energy functional.
	
	\subsection{The oscillatory sector}
	On the complementary sector we follow the strategy of \cite{WuZhai23}, adapted to $\mathbf{n}=\mathbf{e}_3$. Setting
	\begin{align*}
		d=a+B_3,\qquad \mathbf{G}=\mathbb{Q}\mathbf{u}-\frac{1}{\nu}\Delta^{-1}\nabla d,\qquad \nu=\lambda+2\mu,
	\end{align*}
	the linearization of \eqref{equ pcmhd},
	\begin{equation}\label{equ pcmhd lin}
		\left\{\begin{aligned}
			&\partial_ta+\operatorname{div}\mathbf{u}=0,\\
			&\partial_t\mathbf{u}-\mu\Delta\mathbf{u}-(\lambda+\mu)\nabla\operatorname{div}\mathbf{u}+\nabla (a+ B_3)-\partial_{x_3}\mathbf{B}=0,\\
			&\partial_t\mathbf{B}-\partial_{x_3}\mathbf{u}+\mathbf{e}_3\operatorname{div}\mathbf{u} =0,
		\end{aligned}\right.
	\end{equation}
	yields
	\begin{equation}\label{eq dG}
		\left\{\begin{aligned}
			&\partial_td+\frac{2}{\nu}d+2\operatorname{div}\mathbf{G}=\partial_{x_3}u_3,\\
			&\partial_t\mathbf{G}-\nu\Delta\mathbf{G}=\frac{2}{\nu}\mathbb{Q}\mathbf{u}-\frac{1}{\nu}\Delta^{-1}\nabla (\partial_{x_3}u_3),
		\end{aligned}\right.
	\end{equation}
	exhibiting damping for $d$ at the rate $2/\nu$ announced above. Independently, since $\mathbb{P}=I-\mathbb{Q}$ commutes with $\Delta$ and $\partial_{x_3}$, the pair $(\mathbb{P}\mathbf{u},\partial_{x_3}\mathbf{B})$ satisfies
	\begin{equation}\label{equ Pud3B}
		\left\{\begin{aligned}
			&\partial_t\mathbb{P}\mathbf{u}-\mu\Delta\mathbb{P}\mathbf{u} =\partial_{x_3}\mathbf{B},\\
			&\partial_t(\partial_{x_3}\mathbf{B})-   \partial_{x_3}^2\mathbb{P}\mathbf{u} = \partial_{x_3}^2\mathbb{Q}\mathbf{u} -\partial_{x_3}(\mathbf{e}_3 \operatorname{div}\mathbf{u}) ,
		\end{aligned}\right.
	\end{equation}
	whence, applying $\partial_t-\mu\Delta$ to the second equation,
	\begin{align*}
		(\partial_t^2-\mu\Delta\partial_t)(\partial_{x_3} \mathbf{B})- \partial_{x_3}^2(\partial_{x_3} \mathbf{B})=(\partial_t-\mu\Delta)\big(\partial_{x_3}^2\mathbb{Q}\mathbf{u}-\partial_{x_3}(\mathbf{e}_3 \operatorname{div}\mathbf{u})\big),
	\end{align*}
	a damped wave equation for $\partial_{x_3}\mathbf{B}$. Since $\widetilde{\mathbf{B}}$ has vanishing $x_3$-average by definition, $\|\widetilde{\mathbf{B}}\|_{H^k}\lesssim\|\partial_{x_3}\mathbf{B}\|_{H^k}$ with no loss of derivatives, and therefore
	\begin{align*}
		\|\widetilde{a} \|_{H^{k}}\lesssim \|\widetilde{d} \|_{H^{k}}+\|\partial_{x_3}\mathbf{B} \|_{H^{k}} .
	\end{align*}
	Damping for $\widetilde{a}$ and $\widetilde{\mathbf{B}}$ separately follows. The symmetry \eqref{as symmetry} enters at the next stage. It gives $\overline{\mathbf{B}}_h=0$ and $\overline{u}_3=0$, so that the estimate just described applies to the full fields $\mathbf{B}_h$ and $u_3$ rather than only to their oscillations. This is what allows the Lorentz nonlinearity to be split as $\mathbf{B}\cdot\nabla\mathbf{B}=\widetilde{\mathbf{B}}_h\cdot\nabla_h\mathbf{B}+B_3\partial_{x_3}\mathbf{B}$ and controlled by the damping of the oscillatory sector; see the estimate of $F_{\overline{\mathbf{u}}}^{(1)}$ in Lemma \ref{est bE1}.
	
	On the averaged sector this decoupling is genuinely unavailable, and $\mathcal{D}$ does not restore it: $\overline{a}$ and $\overline{\mathbf{B}}$ remain undamped, as Theorem \ref{thm1} reflects. What $\overline{\mathcal{D}}$ supplies instead is damping for the one combination that appears as a force in the $\overline{\mathbf{u}}$ equation, and this is exactly what is needed to close the energy estimates without any damping for $\overline{a}$ and $\overline{\mathbf{B}}$ individually.
	
	\subsection{Weighted energy functionals}
	The two sectors relax on different time scales and must be weighted differently. With $s>7$, $\sigma\in(0,1/2)$, $i=1,2,3$ and a large parameter $\mathcal{T}>0$ to be fixed, we set
	\begin{align}\label{def energy}
		\begin{aligned}
			E_0(t)=&\sup_{\tau\in[0,t]}(\mathcal{T}+\tau)^{-\sigma}\| (a,\mathbf{B},\mathbf{u})(\tau)\|^2_{H^{s}}+\int_{0}^{t}(\mathcal{T}+\tau)^{-1-\sigma}\| (a,\mathbf{B})(\tau)\|^2_{H^{s}}d\tau\\
			&+\int_{0}^{t}(\mathcal{T}+\tau)^{-\sigma}\| \nabla \mathbf{u}(\tau)\|^2_{H^{s}}d\tau,\\
			\overline{E_i}(t)=&\sup_{\tau\in[0,t]} \| (\overline{a},\overline{\mathbf{B}})(\tau)\|^2_{H^{s-i}}+\sup_{\tau\in[0,t]}(\mathcal{T}+\tau)^{i-\sigma}\| \overline{\mathbf{u}}(\tau)\|^2_{H^{s-i}}+\int_{0}^{t}(\mathcal{T}+\tau)^{i-\sigma}\| \nabla\overline{\mathbf{u}}(\tau)\|^2_{H^{s-i}}d\tau\\
			&+\sup_{\tau\in[0,t]}(\mathcal{T}+\tau)^{i-\sigma}\| \overline{\mathcal{D}}(\tau)\|^2_{\dot{H}^{s-i}}+\int_{0}^{t}(\mathcal{T}+\tau)^{i-\sigma}\| \overline{\mathcal{D}}(\tau)\|^2_{\dot{H}^{s-i}}d\tau,\\
			\widetilde{E_i}(t)=&\sup_{\tau\in[0,t]}(\mathcal{T}+\tau)^{i-\sigma} \| (\widetilde{a},\widetilde{\mathbf{u}},\widetilde{\mathbf{B}})(\tau)\|^2_{H^{s-i}}+\int_{0}^{t}(\mathcal{T}+\tau)^{i-\sigma}\| \nabla\widetilde{\mathbf{u}}(\tau)\|^2_{H^{s-i}}d\tau,\\
			N_i(t)=&\sup_{\tau\in[0,t]}(\mathcal{T}+\tau)^{i-\sigma} \| (\widetilde{d},\widetilde{\mathbf{G}})(\tau)\|^2_{H^{s-1-i}}+\int_{0}^{t}(\mathcal{T}+\tau)^{i-\sigma}\| (\nabla \widetilde{\mathbf{G}},\widetilde{d},\partial_{x_3}\mathbf{B})(\tau)\|^2_{H^{s-1-i}}d\tau,
		\end{aligned}
	\end{align}
	and $E_i=\overline{E_i}+\widetilde{E_i}$, $E_{total}=E_0+\sum_{i=1}^3(E_i+N_i)$. Note that in $\overline{E_i}$ the damping term $\int_0^t(\mathcal{T}+\tau)^{i-\sigma}\|\overline{\mathcal{D}}\|^2_{\dot{H}^{s-i}}d\tau$ carries the same weight and the same number of derivatives as the corresponding supremum, the signature of damping without smoothing identified in \S\ref{subsec D}, whereas $(\overline{a},\overline{\mathbf{B}})$ appear with no time weight at all, since they do not decay.
	
	Our argument relies on two technical devices.
	
	The first is the shifted time weights. The mismatch between the weights carried by $(\overline{a},\overline{\mathbf{B}})$ and by $\overline{\mathbf{u}}$ prevents the linear coupling terms from cancelling when the individual estimates are combined. Using the shifted weight $(\mathcal{T}+t)^{i-\sigma}$ rather than $(1+t)^{i-\sigma}$ generates a small prefactor $\mathcal{T}^{\frac{2\sigma-1}{2}}$, and choosing $\mathcal{T}$ large absorbs these terms. This is where the restriction $\sigma<1/2$ is used.
	
	The second is the density-weighted estimates for $\widetilde{\mathbf{B}}$. The energy provides decay for $\|\widetilde{\mathbf{B}}\|_{H^{s-2}}$ but not for $\|\widetilde{\mathbf{B}}\|_{H^{s-1}}$, so the top-order term $\overline{\mathbf{B}}\cdot\nabla^{s}\widetilde{\mathbf{u}}\cdot\nabla^{s-1}\widetilde{\mathbf{B}}$ loses a derivative. Pairing it against the corresponding term in the $\widetilde{\mathbf{u}}$ equation does not suffice, because the density factor $I(a)$ leaves the residual
	\begin{align*}
		\int_{0}^t\!\!\int_{\mathbb{T}^3}(\mathcal{T}+\tau)^{1-\sigma}\frac{\overline{a}}{1+\overline{a}}\, \overline{\mathbf{B}}\cdot \nabla\nabla^{s-1} \widetilde{\mathbf{B}}\cdot \nabla^{s-1}\widetilde{\mathbf{u}}\,dx\,d\tau,
	\end{align*}
	which cannot be absorbed since $\overline{a}$ does not decay. Performing the $\widetilde{\mathbf{B}}$ estimate in the weighted norm with density $\frac{1}{1+\overline{a}}$ produces an exact cancellation. The same weight handles $\mathbf{B}\operatorname{div}\mathbf{u}$ and $\frac12\nabla|\mathbf{B}|^2$.
	
	\begin{rem}
		Compared with \cite{PanZhouZhu18}, our functionals do not contain the top-order quantities $\|\partial_{x_3}\mathbf{u}\|_{H^{s-1}}$ and $\|\partial_{x_3}\mathbf{B}\|_{H^{s-1}}$. In the absence of the incompressibility constraint, the highest-order advective terms no longer cancel under integration by parts.
	\end{rem}
	
	\begin{rem}\label{rem hs}
		The solution of Theorem \ref{thm1} in fact satisfies the top-order bound
		\begin{align*}
			\|(a,\mathbf{u},\mathbf{B})(t)\|_{H^{s}}\leq C\varepsilon(1+t)^{\sigma/2},
		\end{align*}
		with growth as slow as desired.
	\end{rem}
	
	\begin{rem}
		The requirement $s>7$ is not optimal; it is imposed for technical convenience.
	\end{rem}
	
	\begin{rem}
		Without the symmetry \eqref{as symmetry}, and for a rational background field, the global stability of the three-dimensional periodic MHD system with viscosity or resistivity alone remains open. We expect the total pressure $\mathcal{D}$ and its damped wave structure \eqref{eq Dwave} to be relevant beyond the symmetric class, since neither \eqref{def D} nor \eqref{eq mathcalD} uses \eqref{as symmetry}.
	\end{rem}
	
	\subsection{Organization}
	Section \ref{pre} collects the preliminary product, commutator, composition and Poincar\'e-type inequalities. Section \ref{apr} establishes the {\it a priori} estimates for $E_0$, $\overline{E}_1$, $\widetilde{E}_1$, $N_1$, $E_3$ and $N_3$; the damping of $\overline{\mathcal{D}}$ is extracted in the proof of Lemma \ref{est bE1}. Section \ref{sec:bootstrap-closure} closes the bootstrap and proves Theorem \ref{thm1}.
	
	Throughout, $A\lesssim B$ means $A\leq CB$ for a constant $C>0$ which may change from line to line and is independent of $t$ and $\mathcal{T}$. We write $P$ for the pressure and $\mathbb{P}=I-\mathbb{Q}$ for the Leray projector; the two are always distinguished by the font.

	\vskip .2in
	\section{Preliminaries} \label{pre}
	In this section, we provide several useful inequalities that will be frequently used in the subsequent sections. We begin with the standard Sobolev product and commutator estimates; see, for instance, \cite{Kato90,Kato98}.
	
	\begin{lem}\label{lem product est}
		Let $s\geq 0$. Then there exists a constant $C$ depending only on $s$, such that, for any $f,g\in {H^{s}}(\mathbb{T}^3)\cap {L^\infty}(\mathbb{T}^3)$, we have
		\begin{equation*}
			\|fg\|_{H^{s}}\leq C(\|f\|_{L^\infty}\|g\|_{H^{s}}+\|g\|_{L^\infty}\|f\|_{H^{s}}).
		\end{equation*}
	\end{lem}
	
	\begin{lem}\label{lem commu} 
		Let $s> 0$. Then there exists a constant $C$ depending only on $s$, such that, for any $f,g\in {H^{s}}(\mathbb{T}^3)\cap W^{1,\infty}(\mathbb{T}^3)$, there holds
		\begin{align*}
			\|[\Lambda^s,f\cdot\nabla ]g\|_{L^2}\leq C(\|\nabla f\|_{L^\infty}\|\Lambda^sg\|_{L^2}+\|\Lambda^s f\|_{L^2}\|\nabla g\|_{L^\infty}).
		\end{align*}
		\begin{equation*}
			\|[\Lambda^s,f]g\|_{L^2}
			\leq
			C(
			\|\nabla f\|_{L^\infty}
			\|\Lambda^{s-1}g\|_{L^2}
			+
			\|\Lambda^s f\|_{L^2}
			\|g\|_{L^\infty}
			),
		\end{equation*}
		where
		\[
		\Lambda^s=(-\Delta)^{s/2},
		\qquad
		[\Lambda^s,f]g
		:=\Lambda^s(fg)-f\Lambda^s g.
		\]
	\end{lem}
	
	In the following, we establish estimates for the average and oscillatory parts of a composite function.
	\begin{prop}\label{lem bftf}
		Let $f$ be a sufficiently smooth scalar function. For any $a\in H^k(\mathbb{T}^3)$ with $\| a\|_{H^{k}}\ll 1, \ k\geq 2$, there exists a constant $C>0$ depending only on $f,k$  such that
		\begin{align*}
			\|\overline{f(a)} \|_{H^{k}}\leq C,\qquad
			\|\widetilde{f(a)} \|_{H^{k}}\leq    C\|\widetilde{a} \|_{H^{k}}.
		\end{align*}
	\end{prop}
	\begin{proof}
		By   Taylor expansion, we have
		\begin{align*}
			f(a)=& f(\overline{a}+\widetilde{a})=f(\overline{a})+f^{\prime}(\overline{a})\widetilde{a}+\frac{1}{2}f^{\prime\prime}(\overline{a})\widetilde{a}^2+ \mathcal{R}_{f}(\overline{a},\widetilde{a})\widetilde{a}^3,
		\end{align*}
		where $\mathcal{R}_{f}(\cdot)$ denotes the remainder term of the Taylor expansion, which depends only on the function $f$. Taking the average with respect to $x_3$ of the above expansion and noting that $\overline{\widetilde{a}}=0$, we obtain
		\begin{align}\label{equ bf}
			\overline{f(a)}=& \overline{f(\overline{a})}+\overline{f^{\prime}(\overline{a})\widetilde{a}}+\frac{1}{2}\overline{f^{\prime\prime}(\overline{a})\widetilde{a}^2} +\overline{\mathcal{R}_{f}(\overline{a},\widetilde{a})\widetilde{a}^3}\\
			=&f(\overline{a})+\frac{1}{2}f^{\prime\prime}(\overline{a})\overline{\widetilde{a}^2} +\overline{\mathcal{R}_{f}(\overline{a},\widetilde{a})\widetilde{a}^3} .\nonumber
		\end{align}
		Since $H^k(\mathbb T^3)$ is an algebra for $k\geq2$, the standard Sobolev composition and product estimates, together with the boundedness of the averaging operator on $H^k$, give
		\begin{align*}
			\|\overline{f(a)}\|_{H^k}\leq C.
		\end{align*}
		
		For the oscillatory part, it follows that
		\begin{align}
			\widetilde{f(a)}=&f(a)-\overline{f(a)}=f^{\prime}(\overline{a})\widetilde{a}+\frac{1}{2}f^{\prime\prime}(\overline{a})\widetilde{\widetilde{a}^2}+\widetilde{\mathcal{R}_{f}(\overline{a},\widetilde{a})\widetilde{a}^3}.\label{equ tf}
		\end{align}
		Applying the same composition and product estimates to \eqref{equ tf}, we obtain
		\begin{align*} 
			\|\widetilde{f(a)}\|_{H^{k}}
			\leq C\big(1+\|a\|_{H^k}^2\big)\|\widetilde{a}\|_{H^k}
			\leq C\|\widetilde{a} \|_{H^{k}},
		\end{align*}
		which completes the proof.
	\end{proof}

	Under the assumption \eqref{as zero mean} on initial data $(a_0,\mathbf{u}_0,\mathbf{B}_0)$, we have 
	\begin{align}\label{zero mean}
		\int_{\mathbb{T}^3} a \, dx = 0, \quad \int_{\mathbb{T}^3} \rho \mathbf{u} \, dx = 0, \quad \int_{\mathbb{T}^3} \mathbf{B} \, dx = 0.
	\end{align}
	We shall also use the following Poincar\'e-type inequalities established by Wu and Zhai \cite{WuZhai23}.
	\begin{lem}\label{lem poin}
		Suppose that $(\rho, \mathbf{u}, \mathbf{B})$ is a smooth solution to the system \eqref{equ pcmhd} on $[0,\infty)\times\mathbb{T}^3$ satisfying the zero-mean condition \eqref{zero mean}. Assume further that there exist positive constants $c_0$ and $C_0$ such that
		\[
		0<c_0\leq \rho(t,x)\leq C_0 \quad \text{for } (t,x) \in [0,\infty) \times \mathbb{T}^3.
		\]
		Then, there exists a constant $C > 0$ such that for all $t \ge 0$,    the following inequalities hold
		\begin{align}\label{Poi1H}
			\|\mathbf{B}(t)\|_{L^2}^2\le C\|\nabla \mathbf{B}(t)\|_{L^2}^2,
		\end{align}
		\begin{align}\label{Poi1}
			\|(\sqrt{\rho } \mathbf{u})(t)\|_{L^2}^2\le C\|\nabla  \mathbf{u}(t)\|_{L^2}^2,
		\end{align}
		and
		\begin{align}\label{Poi1'}
			\| \mathbf{u}(t)\|_{L^2}^2\le C\|\nabla  \mathbf{u}(t)\|_{L^2}^2.
		\end{align}
	\end{lem}

	\vskip .3in 
	\section{A priori estimates}
	\label{apr}
	
	We begin with the standard local theory and the bootstrap assumption used throughout this section. 
	For initial data satisfying the assumptions of Theorem~\ref{thm1}, with
	$\rho_0$ bounded away from zero, the standard local well-posedness theory for the compressible viscous non-resistive MHD system  (see \cite{Kawa84})  yields a time $T_*>0$ and a unique classical solution $(a,\mathbf u,\mathbf B)$ to \eqref{equ pcmhd} on $[0,T_*]$ such that
	\[
	(a,\mathbf B)\in C([0,T_*];H^s(\mathbb T^3)),\qquad
	\mathbf u\in C([0,T_*];H^s(\mathbb T^3))
	\cap L^2(0,T_*;H^{s+1}(\mathbb T^3)).
	\]
	
	Before proceeding to the {\it a priori} estimates, we assume that the system \eqref{equ pcmhd} admits a smooth solution $(a,\mathbf u,\mathbf B)$ on the time interval $[0, T]$ for some $T>0$. Then, by uniqueness, the symmetry and zero-mean properties imposed on the initial data are preserved on the interval $[0, T]$. We make the following \textit{a priori} assumption
	\begin{equation}\label{ass}
		E_{total}(T)\leq \delta^2,
	\end{equation}
	where $\delta>0$ is a sufficiently small constant to be fixed independently of $T$. 
	By the definition of $E_{total}$ and the Sobolev embedding $H^{s-1}(\mathbb T^3)\hookrightarrow L^\infty(\mathbb T^3)$, we have, for every $t\in[0,T]$,
	\[
	\|a(t)\|_{L^\infty}
	\lesssim \|a(t)\|_{H^{s-1}}
	\lesssim E_{total}^{\frac{1}{2}}(T)
	\leq \delta.
	\]
	Thus, upon choosing $\delta$ sufficiently small,
	\begin{equation}\label{density positive bound}
		\frac12\leq \rho(t,x)=1+a(t,x)\leq \frac32,
		\qquad (t,x)\in[0,T]\times\mathbb T^3.
	\end{equation}
	In particular, all smooth functions of $\rho$ appearing below, including
	$\rho^{-1}$ and $P'(\rho)/\rho$, remain uniformly bounded together with the derivatives required in the subsequent estimates. We shall derive all estimates in this section under \eqref{ass}; the bootstrap assumption will be improved in the final argument, thereby extending the local solution globally.
	\subsection{Estimate of $E_0(t)$}
	We begin by establishing the highest-order energy bound for  $(a,\mathbf{u},\mathbf{B})$ without yet separating their average and oscillatory parts. The negative time weight supplies an integrable contribution for $(a,\mathbf{B})$, while the pressure and Lorentz-force couplings are paired with the velocity and magnetic equations to cancel the highest-order linear and nonlinear terms. 
	\begin{lem}\label{est E0}
		Under the condition of Theorem \ref{thm1} and assumption \eqref{ass}, we have
		\begin{align*}
			E_0(t)\lesssim \mathcal{T}^{-\sigma}\| (a_0,\mathbf{u}_0,\mathbf{B}_0)\|^2_{H^{s}}+ E_{total}^{\frac32}(t).
		\end{align*}
	\end{lem}
	\begin{proof}
		First, we establish the estimate for $(\mathcal{T}+\tau)^{-\sigma}\| a\|_{\dot{H}^{s}}^2$. Recall the equation for $a$:
		\begin{align*}
			\partial_t a +\mathbf{u}\cdot \nabla a =-\operatorname{div} { \mathbf{u}}-a\operatorname{div}\mathbf{u}.
		\end{align*}
		In order to eliminate the linear term, we perform the following weighted energy estimate:
		\begin{align*}
			\frac12\frac{d}{d\tau}&  \int_{\mathbb{T}^3}(\mathcal{T}+\tau)^{-\sigma}\frac{P^{\prime}(1+a)}{1+a}|\nabla^s a|^2dx\\
			=&(\mathcal{T}+\tau)^{-\sigma} \int_{\mathbb{T}^3}\frac{P^{\prime}(1+a)}{1+a}\nabla^sa\cdot\nabla^s\partial_{\tau}adx+\frac12(\mathcal{T}+\tau)^{-\sigma} \int_{\mathbb{T}^3}\partial_{\tau}\Big(\frac{P^{\prime}(1+a)}{1+a}\Big)|\nabla^sa|^2dx\\
			&-\frac{\sigma}{2}(\mathcal{T}+\tau)^{-1-\sigma} \int_{\mathbb{T}^3}\frac{P^{\prime}(1+a)}{1+a}|\nabla^sa|^2 dx\\
			=&:A_1+A_2+A_3. 
		\end{align*}
		Substituting the equation for $a$, we rewrite $A_1$ as follows:
		\begin{align*}
			A_1=&(\mathcal{T}+\tau)^{-\sigma} \int_{\mathbb{T}^3}\frac{P^{\prime}(1+a)}{1+a}\nabla^sa\cdot\Big(\nabla^s(-(1+a)(\operatorname{div}\mathbf{u}))\Big)dx \\
			&-(\mathcal{T}+\tau)^{-\sigma} \int_{\mathbb{T}^3}\frac{P^{\prime}(1+a)}{1+a}\nabla^sa\cdot\Big(\nabla^s(\mathbf{u}\cdot \nabla a)\Big)dx\\
			=&-(\mathcal{T}+\tau)^{-\sigma} \int_{\mathbb{T}^3}P^{\prime}(1+a)\nabla^sa\cdot\nabla^s(\operatorname{div}\mathbf{u})dx\\
			&- \sum_{l=1}^{s}\binom{s}{l}(\mathcal{T}+\tau)^{-\sigma}\int_{\mathbb{T}^3}\frac{P^{\prime}(1+a)}{1+a}\nabla^sa\cdot\Big(\nabla^{l}(1+a)\nabla^{s-l}(\operatorname{div}\mathbf{u})\Big)dx\\
			&- \sum_{l=1}^{s}\binom{s}{l}(\mathcal{T}+\tau)^{-\sigma}\int_{\mathbb{T}^3}\frac{P^{\prime}(1+a)}{1+a}\nabla^sa\cdot(\nabla^{l}\mathbf{u}\cdot\nabla\nabla^{s-l} a)dx\\
			&+\frac12(\mathcal{T}+\tau)^{-\sigma} \int_{\mathbb{T}^3}\nabla\cdot \Big(\frac{P^{\prime}(1+a)}{1+a}\mathbf{u}\Big)|\nabla^sa|^2dx\\
			=&:A_{1,1}+A_{1,2}+A_{1,3}+A_{1,4}.
		\end{align*}
		
		The first term $A_{1,1}$ will be eliminated by coupling it with the estimate for $\mathbf{u}$. By Proposition \ref{lem bftf} and the definition of $E_i(t)$, we have
		\begin{align*}
			\sup_{\tau\in[0,t]}\left\|\frac{P^{\prime}(1+a)}{1+a}\right\|_{H^{s-1}}\leq C(1+E_1^{1/2}(t)).
		\end{align*}
		Under the {\it a priori} bootstrap assumption $E_i(t) \ll 1$, the coefficient involving the pressure is strictly positive and uniformly bounded by a constant $C$. 
		
		Next, we estimate the term $A_{1,2}$. Utilizing Lemma \ref{lem product est} and the Sobolev embedding, we obtain
		\begin{align*}
			\int_{0}^tA_{1,2}d\tau= & - \sum_{l=1}^{s}\binom{s}{l}\int_{0}^t\int_{\mathbb{T}^3}(\mathcal{T}+\tau)^{-\sigma}\frac{P^{\prime}(1+a)}{1+a}\nabla^sa\cdot\Big(\nabla^{l}(1+a)\nabla^{s-l}(\operatorname{div}\mathbf{u})\Big)dxd\tau\\
			\lesssim & \int_{0}^t(\mathcal{T}+\tau)^{-\sigma}\| a \|_{H^{s}}(\| \nabla a\|_{L^{\infty}}\| \operatorname{div}\mathbf{u} \|_{H^{s-1}}+\| a \|_{H^{s}}\|\operatorname{div}\mathbf{u} \|_{L^{\infty}} )d\tau\\
			\lesssim & \sup_{\tau\in[0,t]}\|   a\|_{H^{3}}\int_{0}^t(\mathcal{T}+\tau)^{-\frac{1+\sigma}{2}}\|  a \|_{H^{s}}(\mathcal{T}+\tau)^{\frac{1-\sigma}{2}}\| \operatorname{div}\mathbf{u}\|_{H^{s-1}} d\tau\\
			&+\sup_{\tau\in[0,t]}(\mathcal{T}+\tau)^{-\frac{\sigma}{2}}\| a\|_{H^{s}}\int_{0}^t(\mathcal{T}+\tau)^{-\frac{1+\sigma}{2}}\| a \|_{H^{s}}(\mathcal{T}+\tau)^{\frac{1 }{2}}\|  \operatorname{div}\mathbf{u}\|_{H^{2}} d\tau\\
			\leq & E_0^{\frac12}(t)E_1(t)+E_0(t)E_2^{\frac{1}{2}}(t),
		\end{align*}
		where we have used the fact that $2-\sigma >1$  to control the time weights.
		
		Invoking standard product estimates to $A_{1,3}$, 
		\begin{align*}
			\int_{0}^tA_{1,3}d\tau =&- \sum_{l=1}^{s}\binom{s}{l}\int_{0}^t\int_{\mathbb{T}^3}(\mathcal{T}+\tau)^{-\sigma}\frac{P^{\prime}(1+a)}{1+a}\nabla^sa(\nabla^{l}\mathbf{u}\cdot\nabla\nabla^{s-l} a)dxd\tau\\
			\lesssim &\int_{0}^t (\mathcal{T}+\tau)^{-\sigma}\| a\|_{H^{s}}(\| \mathbf{u}\|_{H^{s}}\|\nabla a \|_{L^{\infty}}+\| \nabla \mathbf{u}\|_{L^{\infty}}\|a \|_{H^{s}} ) d\tau\\
			\lesssim& \sup_{\tau\in[0,t]}\|  a \|_{H^{3}}\int_{0}^t (\mathcal{T}+\tau)^{-\frac{1+\sigma}{2}}\| a\|_{H^{s}}(\mathcal{T}+\tau)^{\frac{1-\sigma}{2}}\| \nabla \mathbf{u}\|_{H^{s-1}} d\tau\\
			&+\sup_{\tau\in[0,t]}(\mathcal{T}+\tau)^{-\frac{\sigma}{2}}\| a \|_{H^{s}}\int_{0}^t (\mathcal{T}+\tau)^{-\frac{1+\sigma}{2}}\| a\|_{H^{s}}(\mathcal{T}+\tau)^{\frac{1}{2}}\| \nabla \mathbf{u}\|_{H^{2}} d\tau\\
			\lesssim& E_0^{\frac12}(t)E_1(t)+E_0(t)E_2^{\frac12}(t).
		\end{align*}
		
		Turning to $A_{1,4}$, integration by parts leads to
		\begin{align*}
			\int_{0}^tA_{1,4}d\tau =&\frac12\int_{0}^t\int_{\mathbb{T}^3}(\mathcal{T}+\tau)^{-\sigma}\nabla\cdot \Big(\frac{P^{\prime}(1+a)}{1+a}\mathbf{u}\Big)|\nabla^sa|^2dxd\tau\\
			\lesssim & \sup_{\tau\in[0,t]}(\mathcal{T}+\tau)^{-\frac{\sigma}{2}}\| a \|_{H^{s}}\int_{0}^t (1+\|a \|_{H^{3}})(\mathcal{T}+\tau)^{\frac12}\|\nabla \mathbf{u}\|_{H^{2}}(\mathcal{T}+\tau)^{-\frac{1+\sigma}{2}}\|a\|_{H^s}d\tau\\
			\lesssim & E_0(t)E_2^{\frac12}(t).
		\end{align*}
		
		To handle $A_2$, we utilize the continuity equation to replace $\partial_ta$. This allows us to deduce that
		\begin{align*}
			\int_{0}^tA_2d\tau&=\frac12\int_{0}^t\int_{\mathbb{T}^3}(\mathcal{T}+\tau)^{-\sigma} \partial_t\Big(\frac{P^{\prime}(1+a)}{1+a}\Big)|\nabla^sa|^2dxd\tau\\
			&= \frac12 \int_{0}^t\int_{\mathbb{T}^3}(\mathcal{T}+\tau)^{-\sigma} \Big(\frac{P^{\prime\prime}(1+a)}{1+a}-\frac{P^{\prime}(1+a)}{(1+a)^2}\Big)\\
			&\qquad\times\Big(-\operatorname{div} \mathbf{u}-\mathbf{u}\cdot \nabla a-a\operatorname{div}\mathbf{u} \Big)|\nabla^sa|^2dxd\tau\\
			&\lesssim  \sup_{\tau\in[0,t]}(\mathcal{T}+\tau)^{-\frac{\sigma}{2}}\| a \|_{H^{s}}\int_{0}^t(\mathcal{T}+\tau)^{\frac12}\| \nabla\mathbf{u}\|_{H^{2}}(1+\| a\|_{H^{3}})(\mathcal{T}+\tau)^{-\frac{1+\sigma}{2}}\|a\|_{H^s}d\tau\\
			&\lesssim  E_0(t)E_2^{\frac12}(t).
		\end{align*}
		Combining these estimates, we conclude that
		\begin{align*}
			\frac12 \int_{\mathbb{T}^3}&(\mathcal{T}+t)^{-\sigma}\frac{P^{\prime}(1+a)}{1+a}|\nabla^s a|^2dx+\frac{\sigma}{2} \int_{0}^t\int_{\mathbb{T}^3}(\mathcal{T}+\tau)^{-1-\sigma}\frac{P^{\prime}(1+a)}{1+a}|\nabla^sa|^2 dxd\tau-\int_{0}^tA_{1,1}d\tau\\
			\lesssim& \mathcal{T}^{-\sigma}\|a_0\|^2_{H^{s}}+E_0^{\frac32}(t)+E_1^{\frac32}(t)+E_2^{\frac32}(t).
		\end{align*}
		
		Recall that $\mathbf B$ satisfies
		\begin{align*}
			\partial_t\mathbf{B}+\mathbf{u}\cdot \nabla \mathbf{B}=\partial_{x_3} \mathbf{u}-\mathbf{e}_3\operatorname{div} \mathbf{u}+\mathbf{B}\cdot \nabla \mathbf{u}-\mathbf{B}\operatorname{div}\mathbf{u}.
		\end{align*}
		Invoking the Poincar\'e type inequality in Lemma \ref{lem poin}, it suffices to estimate the homogeneous norm $\|\mathbf{B} \|_{\dot{H}^s}$. Taking the $\dot{H}^{s}$ inner product of the above equation with $\mathbf{B}$ and multiplying by $(\mathcal T+\tau)^{-\sigma}$, we obtain
		\begin{align*}
			&\frac{1}{2}\frac{d}{d\tau}\Big((\mathcal{T}+\tau)^{-\sigma}\|  \mathbf{B} \|^2_{\dot{H}^{s}}\Big)+\frac{\sigma}{2}(\mathcal{T}+\tau)^{-1-\sigma}\| \mathbf{B} \|^2_{\dot{H}^{s}}\\
			&= \int_{\mathbb{T}^3}(\mathcal{T}+\tau)^{-\sigma}\nabla^s(\partial_{x_3}\mathbf{u}- \mathbf{e}_3\operatorname{div} \mathbf{u}) \cdot\nabla^s\mathbf{B} dx\\
			&\quad+ \int_{\mathbb{T}^3}(\mathcal{T}+\tau)^{-\sigma}\nabla^s(-\mathbf{u}\cdot \nabla \mathbf{B}+\mathbf{B}\cdot \nabla \mathbf{u}-\mathbf{B}\operatorname{div}\mathbf{u} )\cdot\nabla^s\mathbf{B} dx.
		\end{align*}
		The linear terms will be eliminated by coupling them with the estimate for $\mathbf{u}$ in the subsequent analysis.
		
		Applying the product estimate in Lemma \ref{lem product est} and integrating by parts, we can estimate the first of the nonlinear terms as follows:
		\begin{align*}
			\int_{0}^t&\int_{\mathbb{T}^3}(\mathcal{T}+\tau)^{-\sigma}\nabla^s(-\mathbf{u}\cdot \nabla \mathbf{B})\cdot\nabla^s\mathbf{B} dxd\tau\\
			=& \int_{0}^t \int_{\mathbb{T}^3}\frac12(\mathcal{T}+\tau)^{-\sigma}(\operatorname{div}\mathbf{u})  \nabla^s\mathbf{B} \cdot\nabla^s\mathbf{B} dxd\tau\\
			&- \sum_{l=1}^s\binom{s}{l}\int_{0}^t\int_{\mathbb{T}^3}(\mathcal{T}+\tau)^{-\sigma}(\nabla^l\mathbf{u}\cdot \nabla\nabla^{s-l} \mathbf{B})\cdot\nabla^s\mathbf{B} dxd\tau \\
			\lesssim& \int_{0}^t (\mathcal{T}+\tau)^{-\sigma}\|\operatorname{div}\mathbf{u} \|_{L^{\infty}}  \|\mathbf{B}  \|^2_{H^{s}} d\tau+\int_{0}^t(\mathcal{T}+\tau)^{-\sigma}(\|\mathbf{u} \|_{H^{s}}\|\mathbf{B} \|_{H^{3}}+\|  \mathbf{u} \|_{H^{3}}\| \mathbf{B}\|_{H^{s}}) \|\mathbf{B}\|_{H^{s}}d\tau\\
			\lesssim &  \sup_{\tau\in[0,t]}(\mathcal{T}+\tau)^{-\frac{\sigma}{2}}\|\mathbf{B} \|_{H^{s}}\int_{0}^t (\mathcal{T}+\tau)^{\frac{-1-\sigma}{2}}\|\mathbf{B} \|_{H^{s}} (\mathcal{T}+\tau)^{\frac{1}{2}} \|\nabla\mathbf{u} \|_{H^{2}} d\tau\\
			&+ \sup_{\tau\in[0,t]} \|\mathbf{B} \|_{H^{3}}\int_{0}^t (\mathcal{T}+\tau)^{\frac{-1-\sigma}{2}}\|\mathbf{B} \|_{H^{s}} (\mathcal{T}+\tau)^{\frac{1-\sigma}{2}}\|\mathbf{u} \|_{H^{s}} d\tau\\
			\lesssim & E_0(t)E_2^{\frac12}(t)+E_0^{\frac12}(t)E_1(t).
		\end{align*}
		For brevity, write $\mathbf{B}\nabla\mathbf{B}=\frac12\nabla|\mathbf{B}|^2$ for the magnetic pressure term. For the nonlinear term $\mathbf{B}\cdot \nabla \mathbf{u}-\mathbf{B}\operatorname{div}\mathbf{u}$, the highest-order (i.e., $(s+1)$-th) derivative cancels with those from $\mathbf{B}\cdot\nabla \mathbf{B}$ and $\mathbf{B}\nabla \mathbf{B}$ in the estimate for $\mathbf{u}$, which leads to the following identities:
		\begin{align*}
			\int_{0}^t&\int_{\mathbb{T}^3}(\mathcal{T}+\tau)^{-\sigma}\Big((\mathbf{B}\cdot \nabla\nabla^s \mathbf{u})\cdot\nabla^s\mathbf{B}+ (\mathbf{B}\cdot \nabla\nabla^s \mathbf{B})\cdot\nabla^s\mathbf{u}\Big) dxd\tau=0, 
		\end{align*}
		and
		\begin{align*}
			\int_{0}^t&\int_{\mathbb{T}^3}(\mathcal{T}+\tau)^{-\sigma}\Big((\mathbf{B}\nabla^s\operatorname{div}\mathbf{u})\cdot\nabla^s\mathbf{B}+(\mathbf{B}\nabla\nabla^s\mathbf{B})\cdot\nabla^s\mathbf{u}\Big) dxd\tau\\
			&=-\int_{0}^t\int_{\mathbb{T}^3}(\mathcal{T}+\tau)^{-\sigma}\nabla^s\mathbf{u}\cdot\nabla\mathbf{B}\cdot\nabla^s\mathbf{B} dxd\tau.
		\end{align*}
		Consequently, we deduce the following estimates:
		\begin{align}\label{est bnub hs}
			\int_{0}^t&\int_{\mathbb{T}^3}(\mathcal{T}+\tau)^{-\sigma}(\nabla^s(\mathbf{B}\cdot \nabla \mathbf{u})\cdot\nabla^s\mathbf{B}+ \nabla^s(\mathbf{B}\cdot \nabla \mathbf{B})\cdot\nabla^s\mathbf{u}) dxd\tau\\
			=& \sum_{l=1}^s\binom{s}{l}\int_{0}^t\int_{\mathbb{T}^3}(\mathcal{T}+\tau)^{-\sigma}\big[(\nabla^{l}\mathbf{B}\cdot \nabla \nabla^{s-l}\mathbf{u})\cdot\nabla^s\mathbf{B}+ (\nabla^{l}\mathbf{B}\cdot \nabla \nabla^{s-l}\mathbf{B})\cdot\nabla^s\mathbf{u}\big] dxd\tau\nonumber\\
			\lesssim& \int_{0}^t(\mathcal{T}+\tau)^{-\sigma}(\|\nabla\mathbf{u} \|_{H^{s-1}}\|\mathbf{B} \|_{H^{3}}+\|  \mathbf{u} \|_{H^{3}}\| \mathbf{B}\|_{H^{s}}) \| \mathbf{B}\|_{H^{s}}d\tau\nonumber\\
			&+\int_{0}^t(\mathcal{T}+\tau)^{-\sigma}\|  \mathbf{B}\|_{H^{3}}\|\mathbf{B} \|_{H^{s}} \|\mathbf{u}\|_{H^{s}}d\tau\nonumber\\
			\lesssim &   \sup_{\tau\in[0,t]}\|\mathbf{B} \|_{H^{3}}\int_{0}^t (\mathcal{T}+\tau)^{\frac{-1-\sigma}{2}}\|\mathbf{B} \|_{H^{s}}  (\mathcal{T}+\tau)^{\frac{1-\sigma}{2}}\|\nabla\mathbf{u} \|_{H^{s-1}} d\tau\nonumber\\
			&+ \sup_{\tau\in[0,t]}(\mathcal{T}+\tau)^{-\frac{\sigma}{2}}\|\mathbf{B} \|_{H^{s}}\int_{0}^t (\mathcal{T}+\tau)^{\frac{-1-\sigma}{2}}\|\mathbf{B} \|_{H^{s}} (\mathcal{T}+\tau)^{\frac{1}{2}} \|\nabla\mathbf{u} \|_{H^{2}} d\tau\nonumber
			\\
			\lesssim &E_0^{\frac12}(t)E_1(t)+ E_0(t)E_2^{\frac12}(t),\nonumber
		\end{align}
		and
		\begin{align}\label{est bdub hs}
			- \int_{0}^t&\int_{\mathbb{T}^3}(\mathcal{T}+\tau)^{-\sigma}(\nabla^s(\mathbf{B}\operatorname{div}\mathbf{u})\cdot\nabla^s\mathbf{B}+\nabla^s(\mathbf{B}\nabla\mathbf{B})\cdot\nabla^s\mathbf{u} )dxd\tau\\
			=& -\sum_{l=1}^s\binom{s}{l}\int_{0}^t\int_{\mathbb{T}^3}(\mathcal{T}+\tau)^{-\sigma}(\nabla^{l}\mathbf{B}\nabla^{s-l}\operatorname{div}\mathbf{u})\cdot\nabla^s\mathbf{B}dxd\tau \nonumber\\
			&- \sum_{l=1}^{s}\binom{s}{l} \int_{0}^t\int_{\mathbb{T}^3}(\mathcal{T}+\tau)^{-\sigma}(\nabla^l\mathbf{B}\nabla \nabla^{s-l}\mathbf{B})\cdot\nabla^s\mathbf{u} dxd\tau\nonumber\\
			&+\int_{0}^t\int_{\mathbb{T}^3}(\mathcal{T}+\tau)^{-\sigma} \nabla^s\mathbf{u}\cdot \nabla \mathbf{B}\cdot \nabla^s\mathbf{B}  dxd\tau\nonumber\\
			\lesssim& \int_{0}^t(\mathcal{T}+\tau)^{-\sigma}(\|\mathbf{u} \|_{H^{s}}\|\nabla\mathbf{B} \|_{L^{\infty}}+\|\operatorname{div} \mathbf{u} \|_{L^{\infty}}\| \mathbf{B}\|_{H^{s}}) \|\mathbf{B}\|_{H^{s}}d\tau\nonumber\\
			&+\int_{0}^t(\mathcal{T}+\tau)^{-\sigma}(\| \mathbf{B} \|_{H^{s}}\| \nabla\mathbf{B}\|_{L^{\infty}}+\| \mathbf{B} \|_{H^{s-1}}\| \nabla^2\mathbf{B}\|_{L^{\infty}}) \|\mathbf{u}\|_{H^{s}}d\tau\nonumber\\
			\lesssim &   \sup_{\tau\in[0,t]}\|\mathbf{B} \|_{H^{4}}\int_{0}^t (\mathcal{T}+\tau)^{\frac{-1-\sigma}{2}}\|\nabla^s\mathbf{B} \|_{L^{2}} (\mathcal{T}+\tau)^{\frac{1-\sigma}{2}} \|\nabla\mathbf{u} \|_{H^{s-1}} d\tau\nonumber\\
			&+  \sup_{\tau\in[0,t]}(\mathcal{T}+\tau)^{-\frac{\sigma}{2}}\|\mathbf{B} \|_{H^{s}}\int_{0}^t (\mathcal{T}+\tau)^{\frac{-1-\sigma}{2}}\|\nabla^s\mathbf{B} \|_{L^{2}} (\mathcal{T}+\tau)^{\frac{1}{2}} \|\nabla\mathbf{u} \|_{H^{2}} d\tau\nonumber\\
			\lesssim &E_0^{\frac12}(t)E_1(t)+ E_0(t)E_2^{\frac12}(t).\nonumber
		\end{align}
		Recall the equation for $\mathbf{u}$:
		\begin{align*}
			\rho \partial_t\mathbf{u}+\rho \mathbf{u}\cdot \nabla \mathbf{u}-\mu\Delta \mathbf{u}-(\lambda+\mu)\nabla \operatorname{div}\mathbf{u}+\nabla P(1+a)=\partial_{x_3}\mathbf{B}-\nabla B_3+\mathbf{B}\cdot\nabla \mathbf{B}- \frac12\nabla |\mathbf{B}|^2.
		\end{align*} 
		Taking the $\dot{H}^s$ inner product with $\mathbf{u}$ and integrating by parts, we obtain
		\begin{align*}
			\frac12\frac{d}{d\tau}& \int_{\mathbb{T}^3} \Big[(\mathcal{T}+\tau)^{-\sigma}\rho|\nabla^s\mathbf{u} |^2\Big]dx -\frac{1}{2} \int_{\mathbb{T}^3}\partial_{\tau}((\mathcal{T}+\tau)^{-\sigma}\rho)|\nabla^s\mathbf{u} |^2dx\\
			&+ \int_{\mathbb{T}^3}(\mathcal{T}+\tau)^{-\sigma}[\nabla^s,a]\partial_{\tau}\mathbf{u}\cdot \nabla^s\mathbf{u} dx +
			\mu(\mathcal{T}+\tau)^{-\sigma}\|\nabla \mathbf{u} \|^2_{\dot{H}^{s}}+
			(\lambda+\mu)(\mathcal{T}+\tau)^{-\sigma}\| \operatorname{div}  \mathbf{u} \|^2_{\dot{H}^{s}}
			\\
			=& \int_{\mathbb{T}^3}(\mathcal{T}+\tau)^{-\sigma}\Big[\nabla^s(\partial_{x_3}\mathbf{B}+\mathbf{B}\cdot\nabla \mathbf{B})-\nabla \nabla^s(P(1+a)+B_3+\frac12|\mathbf{B}|^2)\Big]\cdot\nabla^s\mathbf{u}dx\\
			&- \int_{\mathbb{T}^3}(\mathcal{T}+\tau)^{-\sigma}\nabla^s\big(\rho(\mathbf{u}\cdot\nabla \mathbf{u})\big)\cdot\nabla^s\mathbf{u}dx.
		\end{align*}
		
		To estimate the second term on the left-hand side, we utilize the continuity equation for $\rho$ to replace its time derivative. Integrating over time, it follows that
		\begin{align*}
			-\frac{1}{2} \int_{0}^t&\int_{\mathbb{T}^3}\partial_{\tau}((\mathcal{T}+\tau)^{-\sigma}\rho)|\nabla^s\mathbf{u} |^2dxd\tau 
			\\
			= & \frac{1}{2} \int_{0}^t\int_{\mathbb{T}^3}(\sigma(\mathcal{T}+\tau)^{-1-\sigma}\rho-(\mathcal{T}+\tau)^{-\sigma}\partial_{\tau}\rho)|\nabla^s\mathbf{u} |^2dxd\tau\\
			\lesssim  &\sup_{\tau\in[0,t]}(1+\| a\|_{L^{\infty}})\mathcal{T}^{-1} \int_{0}^t(\mathcal{T}+\tau)^{ -\sigma}\|\nabla\mathbf{u} \|^2_{H^{s }}d\tau\\
			&+\int_{0}^t(\mathcal{T}+\tau)^{-\sigma}\|\nabla\mathbf{u} \|^2_{H^{s-1}} \|\operatorname{div}\mathbf{u}+\mathbf{u}\cdot \nabla a+a\operatorname{div}\mathbf{u} \|_{L^{\infty}}d\tau\\
			\lesssim & \mathcal{T}^{-1} \int_{0}^t(\mathcal{T}+\tau)^{ -\sigma}\|\nabla\mathbf{u} \|^2_{H^{s }}d\tau\\
			&+\sup_{\tau\in[0,t]} (1+ \| a\|_{H^{3}})\| \mathbf{u}\|_{H^{3}} \int_{0}^t(\mathcal{T}+\tau)^{ -\sigma}\|\nabla\mathbf{u} \|^2_{H^{s-1}}d\tau.
		\end{align*}
		The first term on the right-hand side can be absorbed into the corresponding dissipation term, provided that  $\mathcal{T}\geq \frac{2C\sigma}{\mu}$, where $C$ is the constant in the preceding line.
		
		For the commutator term, the standard commutator estimate in Lemma \ref{lem commu} gives
		\begin{align*}
			\int_{0}^t\int_{\mathbb{T}^3}&(\mathcal{T}+\tau)^{-\sigma}[\nabla^s,a]\partial_{\tau}\mathbf{u}\cdot \nabla^s\mathbf{u} dxd\tau \\
			\lesssim & \int_{0}^t(\mathcal{T}+\tau)^{-\sigma}(\|a \|_{H^{s}}\|\partial_t\mathbf{u} \|_{L^{\infty}}+\|\nabla \rho \|_{L^{\infty}}\|\partial_{\tau} \mathbf{u} \|_{H^{s-1}} ) \| \mathbf{u} \|_{H^{s}}  d\tau.
		\end{align*}
		Utilizing the momentum equation, we can bound the time derivative $\partial_t\mathbf{u}$ as follows:
		\begin{align}\label{est dtu hk}
			\|\partial_{\tau} \mathbf{u}\|_{H^{k}}\lesssim &\| \mathbf{u}\cdot \nabla \mathbf{u} \|_{H^{k}}+ \Big\|\frac{1}{1+a} \Delta \mathbf{u}\Big\|_{H^{k}}+\Big\|\frac{1}{1+a} \nabla \operatorname{div}\mathbf{u}\Big\|_{H^{k}}\\
			&+\Big\|\frac{1}{1+a}(-\nabla P(1+a)+\partial_{x_3}\mathbf{B} -\nabla B_3 +\mathbf{B}\cdot\nabla \mathbf{B}-\frac12\nabla|\mathbf{B}|^2 )\Big\|_{H^{k}}\nonumber\\
			\lesssim & \|\mathbf{u} \|_{H^{3}}\| \nabla\mathbf{u}\|_{H^{k}}+(1+\|a \|_{H^{3}})\|\nabla \mathbf{u} \|_{H^{k+1}}+\|a \|_{H^{k}}\|\mathbf{u} \|_{H^{4}}\nonumber\\
			&+(1+\| a\|_{H^{3}})\|a \|_{H^{k+1}}+\|\mathbf{B} \|_{H^{k+1}}+\|\mathbf{B} \|_{H^{3}}\|\mathbf{B} \|_{H^{k+1}}.\nonumber
		\end{align}
		Invoking the above estimate, it follows that
		\begin{align*}
			\sum_{k=1}^s&\int_{0}^t\int_{\mathbb{T}^3}(\mathcal{T}+\tau)^{-\sigma}[\nabla^k,a]\partial_{\tau}\mathbf{u}\cdot \nabla^k\mathbf{u} dxd\tau \\
			\lesssim & \int_{0}^t(\mathcal{T}+\tau)^{-\sigma}\|a \|_{H^{s}}(\|\mathbf{u} \|^2_{H^{3}}+\|\mathbf{u} \|_{H^{4}}+\|a \|^2_{H^{3}}+\| \mathbf{B}\|_{H^{3}} )\|\nabla\mathbf{u} \|_{H^{s-1}}  d\tau\\
			&+\int_{0}^t(\mathcal{T}+\tau)^{-\sigma}\| a \|_{H^{3}}(\|\mathbf{u} \|_{H^{3}}\|\nabla \mathbf{u} \|_{H^{s-1}}+\|\nabla \mathbf{u} \|_{H^{s}}+\| a\|_{H^{3}}\|\nabla \mathbf{u} \|_{H^{s}}+\| a\|_{H^{s-1}}\|\mathbf{u} \|_{H^{4}}\\
			&\qquad\qquad\qquad\qquad\qquad+(1+\| a\|_{H^{3}})\| a\|_{H^{s}}+\| \mathbf{B}\|_{H^{s}}+\|\mathbf{B} \|_{H^{3}}\| \mathbf{B}\|_{H^{s}} ) \|\nabla\mathbf{u} \|_{H^{s-1}}  d\tau\\
			\lesssim & \sup_{\tau\in[0,t]} \|a,\mathbf{u},\mathbf{B} \|_{H^{s-1}} \int_{0}^t (\mathcal{T}+\tau)^{\frac{-1-\sigma}{2}}\|(a,\nabla\mathbf{u},\mathbf{B}) \|_{H^{s}}(\mathcal{T}+\tau)^{\frac{1-\sigma}{2}}\|\nabla \mathbf{u} \|_{H^{s-1}} d\tau\\
			\lesssim &E_0^{\frac12}(t)E_1(t).
		\end{align*}
		The linear terms  $\partial_{x_3}\mathbf{B}-\nabla B_3$ on the right-hand side cancel out with the corresponding terms in the linear part of the equation for $\mathbf{B}$. Furthermore, the nonlinear terms $\mathbf{B}\cdot\nabla \mathbf{B}$ and $\frac12\nabla |\mathbf{B}|^2$ have already been treated in \eqref{est bnub hs} and \eqref{est bdub hs}. For the pressure term, invoking Fa\`a di Bruno formula and combining it with $A_{1,1}$, we have
		\begin{align*}
			- &\int_0^t\int_{\mathbb{T}^3}(\mathcal{T}+\tau)^{-\sigma}\nabla \nabla^s(P(1+a))\cdot\nabla^s\mathbf{u}dxd\tau+\int_0^t A_{1,1}d\tau\\
			=& \int_0^t\int_{\mathbb{T}^3}(\mathcal{T}+\tau)^{-\sigma}  \nabla^{s-1}(P^{\prime}(1+a)\nabla a)\cdot\nabla^s\operatorname{div}\mathbf{u}dxd\tau\\
			&- \int_0^t\int_{\mathbb{T}^3}(\mathcal{T}+\tau)^{-\sigma}P^{\prime}(1+a)\nabla^sa\cdot\nabla^s(\operatorname{div}\mathbf{u})dxd\tau\\
			= & -\sum_{l=1}^{s-1}\binom{s-1}{l}\int_0^t\int_{\mathbb{T}^3}(\mathcal{T}+\tau)^{-\sigma}  \nabla(\nabla^l(P^{\prime}(1+a))\nabla^{s-l } a)\cdot\nabla^{s-1}\operatorname{div}\mathbf{u}dxd\tau\\
			\lesssim & \sup_{\tau\in[0,t]}\|a \|_{H^{s-1}}(1+\| a\|^{s-1}_{H^{s-1}})\int_{0}^t(\mathcal{T}+\tau)^{\frac{-1-\sigma}{2}}\| a\|_{H^{s}}(\mathcal{T}+\tau)^{\frac{1-\sigma}{2}}\| \operatorname{div}\mathbf{u}\|_{H^{s-1}} d\tau\\
			\lesssim& E_0^{\frac12}(t)E_1(t).
		\end{align*}
		For the convection term,  we have
		\begin{align*}
			-\int_{0}^t&\int_{\mathbb{T}^3}(\mathcal{T}+\tau)^{-\sigma}\nabla^s\big(\rho(\mathbf{u}\cdot\nabla \mathbf{u})\big)\cdot\nabla^s\mathbf{u}dxd\tau\\
			=&-\sum_{l=0}^s\binom{s}{l}\int_{0}^t\int_{\mathbb{T}^3}(\mathcal{T}+\tau)^{-\sigma}\rho(\nabla^l\mathbf{u}\cdot\nabla\nabla^{s-l} \mathbf{u})\cdot\nabla^s\mathbf{u}dxd\tau\\
			&-\int_{0}^t\int_{\mathbb{T}^3}(\mathcal{T}+\tau)^{-\sigma}[\nabla^s,a](\mathbf{u}\cdot\nabla \mathbf{u})\cdot\nabla^s\mathbf{u}dxd\tau\\
			\lesssim & \sup_{\tau\in[0,t]}\| \mathbf{u}\|_{H^{s-1}}\int_{0}^t(\mathcal{T}+\tau)^{-\sigma}\| \nabla \mathbf{u}\|_{H^{s}}^2 d\tau\\
			&+\sup_{\tau\in[0,t]}(\| a\|_{H^{s}}\| \mathbf{u}\|_{H^{s-2}}+\| a\|_{H^{3}}\| \mathbf{u}\|_{H^{s-1}})\int_{0}^t(\mathcal{T}+\tau)^{-\sigma}\| \nabla \mathbf{u}\|_{H^{s}}^2 d\tau\\
			\lesssim & E_0(t)E_1^{\frac12}(t),
		\end{align*}
		where we used the fact that $s-1\geq 3$ and $\sigma <1$.
		Combining these estimates and invoking the Poincar\'e type inequality from Lemma \ref{lem poin}, we conclude that
		\begin{align*}
			E_0(t)=&\sup_{\tau\in[0,t]}(\mathcal{T}+\tau)^{-\sigma}\|(a,\mathbf{u},\mathbf{B}) \|^2_{{H}^{s}}+\int_{0}^t(\mathcal{T}+\tau)^{-\sigma}\|\nabla\mathbf{u} \|^2_{{H}^{s}} d\tau\\
			&+\int_{0}^{t}(\mathcal{T}+\tau)^{-1-\sigma}\| (a,\mathbf{B})(\tau,\cdot)\|^2_{H^{s}}d\tau\\
			\lesssim& \mathcal{T}^{-\sigma}\| (a_0,\mathbf{u}_0,\mathbf{B}_0)\|^2_{H^{s}}+ (E_0(t)+E_1(t)+E_2(t))^{\frac32}.
		\end{align*}
	\end{proof}

	\subsection{Estimate of $\overline{E}_1(t)$}
	This subsection treats the averaged components $(\overline a,\overline{\mathbf{u}},\overline{\mathbf{B}})$, for which the $x_3$-Poincar\'e mechanism is unavailable. As outlined in the introduction, $\overline a$ and $\overline{\mathbf{B}}$ admit no time-decay weights. Instead, we introduce  the averaged quantity $\overline{\mathcal D}$ to recover hidden damping, assigning the time weight only to $(\overline{\mathbf{u}},\overline{\mathcal D})$. The uncancelled linear couplings in the $\overline a$ and $\overline{\mathbf{B}}$ equations are then controlled by exploiting the shifted factor $(\mathcal T+t)$ to make them perturbatively small.
	\begin{lem}\label{est bE1}
		Under the assumptions of Theorem \ref{thm1} and assumption \eqref{ass}, we have
		\begin{align*}
			\overline{E}_1(t)\lesssim \mathcal{T}^{1-\sigma}\|(a_0,\mathbf{u}_0,\mathbf{B}_0) \|_{H^{s-1}}^2+\mathcal{T}^{\frac{2\sigma-1}{2}}E_0^{\frac12}(t)E_2^{\frac12}(t)+E_{total}^{\frac32}(t).
		\end{align*}
	\end{lem}
	\begin{proof}
		Averaging the continuity equation for $a$ over the  $x_3$-directions, we obtain
		\begin{align}\label{equ ba}
			\partial_t\overline{a}=-\operatorname{div}\overline{\mathbf{u}}-\overline{\mathbf{u}\cdot \nabla a}-\overline{a\operatorname{div}\mathbf{u}}.
		\end{align}
		Taking the $\dot{H}^{s-1}$ inner product of this equation with $\overline{a}$ gives
		\begin{align*}
			\frac12\frac{d}{d\tau} \| \overline{a}  \|^2_{\dot{H}^{s-1}}=&-\int_{\mathbb{T}^3}\nabla^{s-1}(\overline{(1+a) \operatorname{div}\mathbf{u}})\cdot\nabla^{s-1}\overline{a} dx- \int_{\mathbb{T}^3}\nabla^{s-1}(\overline{\mathbf{u}\cdot \nabla a})\cdot\nabla^{s-1}\overline{a} dx\\
			=&-\int_{\mathbb{T}^3}\overline{(1+a)\nabla^{s-1}\operatorname{div}\mathbf{u}}\cdot\nabla^{s-1}\overline{a} dx\\
			&- \sum_{l=1}^{s-1}\binom{s-1}{l}\int_{\mathbb{T}^3} \overline{\nabla^{l}(1+a)\nabla^{s-1-l}\operatorname{div}\mathbf{u}}\cdot\nabla^{s-1}\overline{a} dx \\
			&- \int_{\mathbb{T}^3}\nabla^{s-1}(\overline{\mathbf{u}\cdot \nabla a})\cdot\nabla^{s-1}\overline{a} dx.
		\end{align*}
		
		To bound the first term on the right-hand side, we integrate by parts to shift one derivative, which yields
		\begin{align*}
			-\int_{0}^t&\int_{\mathbb{T}^3}\overline{(1+a)\nabla^{s-1}\operatorname{div}\mathbf{u}}\cdot\nabla^{s-1}\overline{a} dxd\tau\\
			= &\int_{0}^t\int_{\mathbb{T}^3}\overline{\nabla a\nabla^{s-2}\operatorname{div}\mathbf{u}}\cdot\nabla^{s-1}\overline{a}+\overline{(1+a)\nabla^{s-2}\operatorname{div}\mathbf{u}} \cdot\nabla^{s}\overline{a}  dxd\tau\\
			\lesssim & \mathcal{T}^{\frac{2\sigma-1}{2}}\sup_{\tau\in[0,t]}(1+\|a \|_{H^{3}})\int_{0}^t (\mathcal{T}+\tau)^{\frac{2-\sigma}{2}}\|\operatorname{div}\mathbf{u} \|_{H^{s-2}}(\mathcal{T}+\tau)^{\frac{-1-\sigma}{2}}\|a \|_{H^{s}} d \tau\\
			\lesssim & \mathcal{T}^{\frac{2\sigma-1}{2}}E_0^{\frac12}(t)E_2^{\frac12}(t),
		\end{align*}
		where we have used the assumption $1-2\sigma >0$ to extract the small decay factor $\mathcal{T}^{\frac{2\sigma-1}{2}}$.
		
		Next, the lower-order terms arising from the Leibniz expansion are estimated via standard Sobolev inequalities:
		\begin{align*}
			- \sum_{l=1}^{s-1}&\binom{s-1}{l}\int_{0}^t\int_{\mathbb{T}^3} \overline{\nabla^{l}(1+a)\nabla^{s-1-l}\operatorname{div}\mathbf{u}}\cdot\nabla^{s-1}\overline{a} dxd\tau \\
			\lesssim & \int_{0}^t(\|\nabla a \|_{L^{\infty}}\| \operatorname{div}\mathbf{u} \|_{H^{s-2}}+\|a\|_{H^{s-1}}\|\operatorname{div}\mathbf{u} \|_{L^{\infty}})\| \overline{a}\|_{H^{s-1}}  d\tau\\
			\lesssim & \sup_{\tau\in[0,t]}\| a\|_{H^{s-1}}\int_{0}^t (\mathcal{T}+\tau)^{\frac{2-\sigma}{2}}\|\operatorname{div}\mathbf{u} \|_{H^{s-2}}(\mathcal{T}+\tau)^{\frac{-1-\sigma}{2}}\| a\|_{H^{s-1}} d\tau\\
			\lesssim & E_0^{\frac12}(t)E_1^{\frac12}(t)E_2^{\frac12}(t).
		\end{align*}
		Finally, the convective term is handled in a similar manner:
		\begin{align*}
			-\int_{0}^t\int_{\mathbb{T}^3}&\nabla^{s-1}(\overline{\mathbf{u}\cdot \nabla a})\cdot\nabla^{s-1}\overline{a} dxd\tau\\
			\lesssim & \int_{0}^t (\|\nabla \mathbf{u} \|_{H^{s-2}}\| \nabla a\|_{L^{\infty}}+\|\mathbf{u} \|_{L^{\infty}}\|  a \|_{H^{s}})\| a\|_{H^{s-1}} d\tau\\
			\lesssim  & \sup_{\tau\in[0,t]}\|a \|_{H^{s-1}}\int_{0}^t(\mathcal{T}+\tau)^{\frac{2-\sigma}{2}}\|\nabla \mathbf{u} \|_{H^{s-2}}(\mathcal{T}+\tau)^{\frac{-1-\sigma}{2}}\|a \|_{H^{s}}  d\tau\\
			\lesssim & E_0^{\frac12}(t)E_1^{\frac12}(t)E_2^{\frac12}(t).
		\end{align*}
		Integrating over time and combining the above estimates, it follows that
		\begin{align}\label{est ba hs-1}
			\frac12 \|\overline{a} \|^2_{\dot{H}^{s-1}}\lesssim&\| a_0 \|^2_{H^{s-1}}+\mathcal{T}^{\frac{2\sigma-1}{2}}E_0^{\frac12}(t)E_2^{\frac12}(t)+E_0^{\frac12}(t)E_1^{\frac12}(t)E_2^{\frac12}(t).
		\end{align}
		
		Averaging the equation for $\mathbf B$ in $\eqref{equ pcmhd}_3$ with respect to $x_3$, we obtain
		\begin{align}\label{equ bB}
			\partial_t\overline{\mathbf{B}} +\overline{\mathbf{u}\cdot \nabla \mathbf{B}}=-\mathbf{e}_3\overline{\operatorname{div} \mathbf{u}}+\overline{\mathbf{B}\cdot \nabla \mathbf{u}}-\overline{\mathbf{B}\operatorname{div}\mathbf{u}},
		\end{align}
		where we have used $\overline{\partial_{x_3}\mathbf u}=0$. Taking the $\dot H^{s-1}$ inner product of \eqref{equ bB} with $\overline{\mathbf B}$ yields
		\begin{align*}
			\frac{1}{2}\frac{d}{d\tau} \| \overline{\mathbf{B}} \|^2_{\dot{H}^{s-1}}=&- \int_{\mathbb{T}^3} \nabla^{s-1}(\overline{\operatorname{div} \mathbf{u}})\mathbf{e}_3\cdot\nabla^{s-1}\overline{\mathbf{B}} dx\\
			&+ \int_{\mathbb{T}^3} \nabla^{s-1}(-\overline{\mathbf{u}\cdot \nabla \mathbf{B}}+\overline{\mathbf{B}\cdot \nabla \mathbf{u}}-\overline{\mathbf{B}\operatorname{div}\mathbf{u}} )\cdot\nabla^{s-1}\overline{\mathbf{B}} dx.
		\end{align*}
		
		The linear term containing the uncancelled $\mathbf{e}_3$ coupling is handled by integrating by parts, which yields
		\begin{align*}
			- \int_{0}^t\int_{\mathbb{T}^3}& \nabla^{s-1}(\overline{\operatorname{div} \mathbf{u}})\mathbf{e}_3\cdot\nabla^{s-1}\overline{\mathbf{B}} dxd\tau\\ 
			= & \int_{0}^t\int_{\mathbb{T}^3} \nabla^{s-2}(\overline{\operatorname{div} \mathbf{u}})\mathbf{e}_3\cdot\nabla^{s }\overline{\mathbf{B}}dxd\tau\\
			\lesssim & \int_{0}^t  \| \operatorname{div}\mathbf{u}\|_{H^{s-2}}\|\mathbf{B} \|_{H^{s}} d\tau \\
			\lesssim &  \mathcal{T}^{\frac{2\sigma-1}{2}}\int_{0}^t(\mathcal{T}+\tau)^{\frac{2-\sigma}{2}}\| \operatorname{div}\mathbf{u}\|_{H^{s-2}}(\mathcal{T}+\tau)^{\frac{-1-\sigma}{2}}\|\mathbf{B} \|_{H^{s}} d\tau\\
			\lesssim & \mathcal{T}^{\frac{2\sigma-1}{2}}E_0^{\frac12}(t)E_2^{\frac12}(t).
		\end{align*}
		
		For the convective term, standard Sobolev embedding gives
		\begin{align*}
			- \int_{0}^t\int_{\mathbb{T}^3} &\nabla^{s-1} (\overline{\mathbf{u}\cdot \nabla \mathbf{B}}) \cdot\nabla^{s-1}\overline{\mathbf{B}} dxd\tau \\
			\lesssim& \int_{0}^t (\|\mathbf{u} \|_{H^{s-1}}\|\nabla \mathbf{B} \|_{L^{\infty}}+\|\mathbf{u} \|_{L^{\infty}}\|\nabla \mathbf{B} \|_{H^{s-1}}  )\|\mathbf{B} \|_{H^{s-1}} d\tau \\
			\lesssim & \sup_{\tau\in[0,t]}\|\mathbf{B} \|_{H^{s-1}}\int_{0}^t(\mathcal{T}+\tau)^{\frac{2-\sigma}{2}}\|\nabla \mathbf{u} \|_{H^{s-2}}(\mathcal{T}+\tau)^{\frac{-1-\sigma}{2}}\|\mathbf{B}\|_{H^{s}}  d\tau\\
			\lesssim & E_0^{\frac12}(t)E_1^{\frac12}(t)E_2^{\frac12}(t).
		\end{align*}
		
		Applying the Leibniz rule and integrating the highest-order velocity derivative by parts, this provides the bound:
		\begin{align*}
			\int_{0}^t \int_{\mathbb{T}^3}& \nabla^{s-1} (\overline{\mathbf{B}\cdot \nabla \mathbf{u}}) \cdot\nabla^{s-1}\overline{\mathbf{B}} dxd\tau \\
			=&  \sum_{l=1}^{s-1}\binom{s-1}{l}\int_{0}^t\int_{\mathbb{T}^3}   (\overline{\nabla^{l}\mathbf{B}\cdot \nabla\nabla^{s-1-l} \mathbf{u}}) \cdot\nabla^{s-1}\overline{\mathbf{B}} dxd\tau\\
			&- \int_{0}^t\int_{\mathbb{T}^3}     \overline{(\nabla^{s-1}\mathbf{u})\cdot   \mathbf{B}}\cdot\nabla\nabla^{s-1}\overline{\mathbf{B}} dxd\tau \\
			\lesssim& \int_{0}^t (\|\nabla \mathbf{B} \|_{L^{\infty}}\|\nabla\mathbf{u} \|_{H^{s-2}}+\|\mathbf{B} \|_{H^{s-1}}\| \nabla \mathbf{u} \|_{L^{\infty}}  )\|\mathbf{B} \|_{H^{s-1}} d\tau\\
			&+\int_{0}^t \|\nabla \mathbf{u} \|_{H^{s-2}}\| \mathbf{B} \|_{L^{\infty}} \|\mathbf{B} \|_{H^{s}} d\tau\\
			\lesssim & \sup_{\tau\in[0,t]}\|\mathbf{B} \|_{H^{s-1}}\int_{0}^t(\mathcal{T}+\tau)^{\frac{2-\sigma}{2}}\|\nabla \mathbf{u} \|_{H^{s-2}}(\mathcal{T}+\tau)^{\frac{-1-\sigma}{2}}\|\mathbf{B}\|_{H^{s}}  d\tau\\
			\lesssim & E_0^{\frac12}(t)E_1^{\frac12}(t)E_2^{\frac12}(t).
		\end{align*}
		
		Similarly, for the last term, we have
		\begin{align*}
			-  \int_{0}^t\int_{\mathbb{T}^3}& \nabla^{s-1} (\overline{\mathbf{B}  \operatorname{div} \mathbf{u}}) \cdot\nabla^{s-1}\overline{\mathbf{B} }dxd\tau \\
			=&-  \int_{0}^t\int_{\mathbb{T}^3} \nabla^{s-1} (\mathbf{B}  \operatorname{div} \mathbf{u}) \cdot\nabla^{s-1}\overline{\mathbf{B} }dxd\tau\\
			=& - \sum_{l=1}^{s-1}\binom{s-1}{l}\int_{0}^t\int_{\mathbb{T}^3}   (\nabla^{l}\mathbf{B}\operatorname{div}\nabla^{s-1-l} \mathbf{u}) \cdot\nabla^{s-1}\overline{\mathbf{B}} dxd\tau \\
			&+ \int_{0}^t\int_{\mathbb{T}^3}  \nabla^{s-1}   \mathbf{u}\cdot  \nabla(\mathbf{B}\cdot \nabla^{s-1}\overline{\mathbf{B}}) dxd\tau \\
			\lesssim & \int_{0}^t (\|\nabla \mathbf{B} \|_{L^{\infty}}\|\nabla\mathbf{u} \|_{H^{s-2}}+\| \mathbf{B} \|_{H^{s-1}}\| \operatorname{div} \mathbf{u} \|_{L^{\infty}}  )\|\mathbf{B} \|_{H^{s-1}} d\tau\\
			&+\int_{0}^t \|\nabla \mathbf{u} \|_{H^{s-2}}\| \mathbf{B} \|_{H^{3}} \|\mathbf{B} \|_{H^{s}} d\tau\\
			\lesssim &  \sup_{\tau\in[0,t]}\|\mathbf{B} \|_{H^{s-1}}\int_{0}^t(\mathcal{T}+\tau)^{\frac{2-\sigma}{2}}\|\nabla \mathbf{u} \|_{H^{s-2}}(\mathcal{T}+\tau)^{\frac{-1-\sigma}{2}}\|\mathbf{B}\|_{H^{s}}  d\tau\\
			\lesssim & E_0^{\frac12}(t)E_1^{\frac12}(t)E_2^{\frac12}(t).
		\end{align*}
		
		Integrating over time and summing the above contributions, we arrive at the desired bound for the averaged magnetic field
		\begin{align}\label{est b hs-1}
			\frac12 \|\overline{\mathbf{B}} \|^2_{\dot{H}^{s-1}}\lesssim \| \mathbf{B}_0 \|^2_{H^{s-1}}+\mathcal{T}^{\frac{2\sigma-1}{2}}E_0^{\frac12}(t)E_2^{\frac12}(t)+E_0^{\frac12}(t)E_1^{\frac12}(t)E_2^{\frac12}(t).
		\end{align}
		
		Recall  the equation for $\overline{\mathbf{u}}$:
		\begin{align}\label{equ bu}
			\partial_t \overline{\mathbf{u}}  -
			\mu\Delta  \overline{\mathbf{u}}-(\lambda+\mu)\nabla (\operatorname{div}\overline{\mathbf{u}})+  \nabla\overline{ \mathcal{D}}=\overline{F_{\overline{\mathbf{u}}}},
		\end{align} 
		where the nonlinear term is given by
		\begin{align*}
			F_{\overline{\mathbf{u}}}=&-\mathbf{u}\cdot \nabla \mathbf{u}+\mathbf{B}\cdot \nabla \mathbf{B} -\mu I(a){\Delta \mathbf{u}}-(\lambda +\mu)I(a){\nabla \operatorname{div}\mathbf{u}}\\
			&-I(a)(\partial_{x_3} \mathbf{B}+\mathbf{B}\cdot \nabla \mathbf{B}-\nabla\mathcal{D}),
		\end{align*}
		with $I(a)=\frac{a}{1+a}$. Recall the new quantity 	$\mathcal{D}= P(1+a)+B_3+\frac12|\mathbf{B}|^2$. By utilizing the evolution equations for $a$ and $\mathbf{B}$, we derive the governing equation for $\mathcal{D}$ as follows
		\begin{align*}
			\partial_t\mathcal{D}=&P^{\prime}(1+a)\partial_ta+\partial_tB_3+\mathbf{B}\cdot \partial_t\mathbf{B}\\
			=&P^{\prime}(1+a)(-\operatorname{div}\mathbf{u}-\mathbf{u}\cdot \nabla a-a\operatorname{div}\mathbf{u})+\partial_{x_3}u_3-\operatorname{div}\mathbf{u}-\mathbf{u}\cdot\nabla B_3+\mathbf{B}\cdot \nabla u_3-B_3\operatorname{div}\mathbf{u}\\
			&+\mathbf{B}\cdot (\partial_{x_3}\mathbf{u}-\mathbf{e}_3\operatorname{div}\mathbf{u}-\mathbf{u}\cdot\nabla \mathbf{B}+\mathbf{B}\cdot \nabla \mathbf{u}-\mathbf{B}\operatorname{div}\mathbf{u} )\\
			=& \partial_{x_3}u_3-2\operatorname{div}\mathbf{u}+F_{\mathcal{D}},
		\end{align*}
		where
		\begin{align*}
			F_{\mathcal{D}}=& (1-P^{\prime}(1+a))\operatorname{div}\mathbf{u}-P^{\prime}(1+a)(\mathbf{u}\cdot\nabla a+a\operatorname{div}\mathbf{u})-\mathbf{u}\cdot\nabla B_3+\mathbf{B}\cdot \nabla u_3-2B_3\operatorname{div}\mathbf{u}\\
			&+\mathbf{B}\cdot\partial_{x_3}\mathbf{u}-\frac12\mathbf{u}\cdot\nabla |\mathbf{B}|^2+\mathbf{B}\cdot (\mathbf{B}\cdot \nabla \mathbf{u})-|\mathbf{B}|^2\operatorname{div}\mathbf{u}.
		\end{align*}
		
		Averaging over $x_3$ and invoking $\overline{\partial_{x_3}u_3}=0$, we have
		\begin{align}\label{equ bD}
			\partial_t\overline{\mathcal{D}}+2\overline{\operatorname{div}\mathbf{u}}=\overline{F_{\mathcal{D}}}.
		\end{align}
		
		To handle the uncancelled linear terms $P(1+a)+B_3+\frac12|\mathbf{B}|^2$ in the equation for $\overline{\mathbf{u}}$,   we derive  the estimate for $\|\overline{\mathbf{u}} \|^2_{\dot{H}^{s-1}}$ and $\frac12\|\overline{\mathcal{D}} \|^2_{\dot{H}^{s-1}}$ as follows
		\begin{align}\label{est dtbubD hs-1}
			\frac12\frac{d}{d\tau}& \Big((\mathcal{T}+\tau)^{1-\sigma}(\|\overline{\mathbf{u}} \|^2_{\dot{H}^{s-1}} +\frac{1}{2}\|\overline{\mathcal{D}} \|^2_{\dot{H}^{s-1}})\Big) -\frac{1-\sigma}{2}(\mathcal{T}+\tau)^{-\sigma}(\|\overline{\mathbf{u}} \|^2_{\dot{H}^{s-1}}+\frac12\|\overline{\mathcal{D}} \|^2_{\dot{H}^{s-1}})\\
			&+
			\mu(\mathcal{T}+\tau)^{1-\sigma}\|\nabla \overline{\mathbf{u}} \|^2_{\dot{H}^{s-1}}+
			(\lambda+\mu)(\mathcal{T}+\tau)^{1-\sigma}\| \operatorname{div}  \overline{\mathbf{u}} \|^2_{\dot{H}^{s-1}} \nonumber
			\\
			=&\int_{\mathbb{T}^3}(\mathcal{T}+\tau)^{1-\sigma}(\nabla^{s-1}\overline{F_{\overline{\mathbf{u}}}}\cdot\nabla^{s-1}\overline{\mathbf{u}}+\frac12 \nabla^{s-1}\overline{F_{\mathcal{D}}}\nabla^{s-1}\overline{\mathcal{D}} )dx,\nonumber
		\end{align}
		where we have used the fact that
		\begin{align*}
			\int_{\mathbb{T}^3}(\mathcal{T}+\tau)^{1-\sigma}\big(\nabla^{s-1}\nabla\overline{\mathcal{D}}\cdot\nabla^{s-1}\overline{\mathbf{u}}+ \nabla^{s-1}\overline{\operatorname{div}\mathbf{u}}\nabla^{s-1}\overline{\mathcal{D}}\big) dx=0.
		\end{align*}
		
		To control the negative term arising from the time derivative on the left-hand side, Lemma \ref{lem poin} implies 
		\begin{align*}
			&-\frac{1-\sigma}{2}(\mathcal{T}+\tau)^{-\sigma}(\|\overline{\mathbf{u}} \|^2_{\dot{H}^{s-1}}+\frac12\|\overline{\mathcal{D}} \|^2_{\dot{H}^{s-1}})\\
			&\leq C \mathcal{T}^{-1}(\mathcal{T}+\tau)^{1-\sigma}\|\nabla\overline{\mathbf{u}} \|^2_{\dot{H}^{s-1}}+\mathcal{T}^{-1}(\mathcal{T}+\tau)^{1-\sigma}\|\overline{\mathcal{D}} \|^2_{\dot{H}^{s-1}}\\
			&\leq \frac{\mu}{4}(\mathcal{T}+\tau)^{1-\sigma}\|\nabla\overline{\mathbf{u}} \|^2_{\dot{H}^{s-1}}+\mathcal{T}^{-1}(\mathcal{T}+\tau)^{1-\sigma}\|\overline{\mathcal{D}} \|^2_{\dot{H}^{s-1}},
		\end{align*}
		provided that the time shift parameter is chosen sufficiently large,  $\mathcal{T}\geq \frac{4C}{\mu}$.
		
		For the nonlinear terms, applying the Cauchy-Schwarz inequality and integrating by parts yields
		\begin{align*}
			\int_{\mathbb{T}^3}&(\mathcal{T}+\tau)^{1-\sigma}(\nabla^{s-1}\overline{F_{\overline{\mathbf{u}}}}\cdot\nabla^{s-1}\overline{\mathbf{u}}+\frac12 \nabla^{s-1}\overline{F_{\mathcal{D}}}\nabla^{s-1}\overline{\mathcal{D}}) dx\\
			=&-\int_{\mathbb{T}^3}(\mathcal{T}+\tau)^{1-\sigma}(\nabla^{s-2}\overline{F_{\overline{\mathbf{u}}}}\cdot\nabla^{s}\overline{\mathbf{u}}-\frac12 \nabla^{s-1}\overline{F_{\mathcal{D}}}\nabla^{s-1}\overline{\mathcal{D}})dx\\
			\leq & (\mathcal{T}+\tau)^{1-\sigma}(\|\overline{F_{\overline{\mathbf{u}}}} \|_{\dot{H}^{s-2}}\|\nabla \overline{\mathbf{u}} \|_{\dot{H}^{s-1}}+\|\overline{F_{\mathcal{D}}} \|_{\dot{H}^{s-1}}\| \overline{\mathcal{D}} \|_{\dot{H}^{s-1}}).
		\end{align*}
		
		To bound the nonlinear term $F_{\overline{\mathbf{u}}}$, we decompose it into three parts:
		\begin{align*}
			F_{\overline{\mathbf{u}}}= F_{\overline{\mathbf{u}}}^{(1)}+F_{\overline{\mathbf{u}}}^{(2)}+F_{\overline{\mathbf{u}}}^{(3)},
		\end{align*}
		where
		\begin{align*}
			F_{\overline{\mathbf{u}}}^{(1)}=& -{\mathbf{u}\cdot \nabla \mathbf{u}}+{\mathbf{B}\cdot \nabla \mathbf{B}}-\mu I(a){\Delta \mathbf{u}}-(\lambda +\mu)I(a){\nabla \operatorname{div}\mathbf{u}},\\
			F_{\overline{\mathbf{u}}}^{(2)}=& -I(a)(\partial_{x_3} \mathbf{B}+\mathbf{B}\cdot \nabla \mathbf{B}),\\
			F_{\overline{\mathbf{u}}}^{(3)}=& I(a)\nabla\mathcal{D}.
		\end{align*}
		\begin{rem}
			We emphasize that due to the absence of magnetic dissipation, controlling the magnetic terms relies critically on symmetry assumptions and the damped wave structure of $\partial_{x_3}\mathbf{B}$. The corresponding estimates will be established as part of the bounds for $N_i(t)$.
		\end{rem}
		
		Recalling that $\mathbf{B}_h$ is an odd function with respect to $x_3$, it follows that
		\begin{align*}
			\overline{\mathbf{B}\cdot \nabla \mathbf{B}}=\overline{\mathbf{B}_h\cdot \nabla_h \mathbf{B}}+\overline{B_3\partial_{x_3} \mathbf{B}}
			=\overline{\mathbf{B}_h\cdot \widetilde{\nabla_h \mathbf{B}}}+\overline{\widetilde{B_3}\partial_{x_3} \mathbf{B}}.
		\end{align*}
		Invoking the above decomposition, we can estimate
		\begin{align*}
			\|\overline{F_{\overline{\mathbf{u}}}^{(1)}} \|_{\dot{H}^{s-2}}= &\|-\overline{\mathbf{u}\cdot \nabla \mathbf{u}}+\overline{\mathbf{B}_h\cdot \widetilde{\nabla_h \mathbf{B}}}+\overline{\widetilde{B_3}\partial_{x_3} \mathbf{B}}-\mu\overline{( I(a))\Delta \mathbf{u}}-(\lambda +\mu)\overline{(I(a))\nabla \operatorname{div}\mathbf{u}} \|_{\dot{H}^{s-2}}\\
			\lesssim & \|\mathbf{u} \|^2_{H^{s-1}}+\|\widetilde{\mathbf{B}_h} \|_{H^{s-2}}\|\widetilde{\mathbf{B}} \|_{H^{s-1}}+\| \widetilde{B_3}\|_{H^{s-2}}\| \partial_{x_3}\mathbf{B}\|_{H^{s-2}}+\| a\|_{H^{s-1}}\|\nabla\mathbf{u} \|_{H^{s-1}} \\
			\lesssim & \|(a,\mathbf{u},\widetilde{\mathbf{B}}) \|_{H^{s-1}}(\|\nabla\mathbf{u} \|_{H^{s-1}}+\|\partial_{x_3}\widetilde{\mathbf{B}} \|_{H^{s-2}}).
		\end{align*}
		
		By a similar way, the second part $F_{\overline{\mathbf{u}}}^{(2)}$ is bounded by
		\begin{align*}
			\|\overline{F_{\overline{\mathbf{u}}}^{(2)}} \|_{\dot{H}^{s-2}}=&\|\overline{\widetilde{I(a)}\partial_{x_3} \mathbf{B}}+\overline{ I(a)\mathbf{B}\cdot \nabla \mathbf{B}} \|_{\dot{H}^{s-2}}\\
			\lesssim & \|\widetilde{a} \|_{H^{s-2}}\|\partial_{x_3}\mathbf{B} \|_{H^{s-2}}+\|a \|_{H^{s-2}}(\|\widetilde{\mathbf{B}_h} \|_{H^{s-2}}\|\widetilde{\mathbf{B}} \|_{H^{s-1}}+\| \widetilde{B_3}\|_{H^{s-2}}\| \partial_{x_3}\mathbf{B}\|_{H^{s-2}})\\
			\lesssim & \| a\|_{H^{s-2}}\|\partial_{x_3}\widetilde{\mathbf{B}} \|_{H^{s-2}}(1+\| \widetilde{\mathbf{B}} \|_{H^{s-1}} ).
		\end{align*}
		Turning to $F_{\overline{\mathbf{u}}}^{(3)}$, we split the product into average and oscillatory parts based on the definition of $\mathcal{D}$:
		\begin{align*}
			\overline{I(a)\nabla\mathcal{D}}=&\overline{I(a)}\nabla\overline{\mathcal{D}}+\overline{\widetilde{I(a)}\nabla\widetilde{\mathcal{D}}}\\
			=& \overline{I(a)}\nabla\overline{\mathcal{D}}+\overline{\widetilde{I(a)}\nabla(\widetilde{P(1+a)}+\widetilde{B_3}+\frac{1}{2}\widetilde{|\mathbf{B}|^2})}.
		\end{align*}
		By the above equality and Proposition \ref{lem bftf},
		\begin{align*}
			\|\overline{F_{\overline{\mathbf{u}}}^{(3)}}\|_{\dot{H}^{s-2}}=& \|\overline{I(a)}\nabla\overline{\mathcal{D}}+\overline{\widetilde{I(a)}\nabla(\widetilde{P(1+a)}+\widetilde{B_3}+\frac{1}{2}\widetilde{|\mathbf{B}|^2})} \|_{\dot{H}^{s-2}}\\
			\lesssim & \|a \|_{H^{s-1}}\|\overline{\mathcal{D}} \|_{\dot{H}^{s-1}}+\|\widetilde{a} \|_{H^{s-2}}(\|\nabla\widetilde{P(1+a)}  \|_{H^{s-2}}+\|\nabla\widetilde{B_3} \|_{H^{s-2}}+\|\nabla\widetilde{|\mathbf{B}|^2} \|_{H^{s-2}} )\\
			\lesssim &\|a \|_{H^{s-1}}\|\overline{\mathcal{D}} \|_{\dot{H}^{s-1}}+(\| (a,\mathbf{B})\|_{H^{s-1}}+\| (a,\mathbf{B})\|^2_{H^{s-1}})\|\widetilde{a} \|_{H^{s-2}}.
		\end{align*}
		
		Combining the above estimates, we conclude that
		\begin{align}\label{est Fbu hs-2}
			\| \overline{F_{\overline{\mathbf{u}}}}\|_{\dot{H}^{s-2}}\lesssim&(1+\| (a,\mathbf{B} ) \|_{H^{s-1}})\|(a,\mathbf{u},  \mathbf{B} ) \|_{H^{s-1}}(\|\overline{\mathcal{D}} \|_{\dot{H}^{s-1}}+\|\nabla\mathbf{u} \|_{H^{s-1}}+\|(\widetilde{a},\nabla\mathbf{u},\partial_{x_3}\widetilde{\mathbf{B}}) \|_{H^{s-2}}).
		\end{align}
		
		Integrating over time, we arrive at
		\begin{align*}
			\int_{0}^t&(\mathcal{T}+\tau)^{1-\sigma}\|\overline{F_{\overline{\mathbf{u}}}} \|_{\dot{H}^{s-2}}\|\nabla \overline{\mathbf{u}} \|_{\dot{H}^{s-1}}d\tau\\
			\lesssim & \sup_{\tau\in[0,t]}(1+\| (a,\mathbf{B} ) \|_{H^{s-1}})\|(a,\mathbf{u},  \mathbf{B} ) \|_{H^{s-1}}\\
			&\times\int_{0}^t(\mathcal{T}+\tau)^{1-\sigma}(\|\overline{\mathcal{D}} \|_{\dot{H}^{s-1}}+\|\nabla\mathbf{u} \|_{H^{s-1}}+\|(\widetilde{a},\partial_{x_3}\widetilde{\mathbf{B}}) \|_{H^{s-2}})\|\nabla \overline{\mathbf{u}} \|_{\dot{H}^{s-1}}d\tau\\
			\lesssim & E_1 (t)(E_1^{\frac12}(t)+N_1^{\frac12}(t)).
		\end{align*}
		
		Analogously, we decompose the nonlinear term  $F_{\mathcal{D}}$ into three parts:
		\begin{align*}
			F_{\mathcal{D}} =F_{\mathcal{D}}^{(1)}+F_{\mathcal{D}}^{(2)}+F_{\mathcal{D}}^{(3)},
		\end{align*}
		where 
		\begin{align*}
			F_{\mathcal{D}}^{(1)}=& (1-P^{\prime}(1+a))\operatorname{div}\mathbf{u}-P^{\prime}(1+a)(\mathbf{u}\cdot\nabla a+a\operatorname{div}\mathbf{u}),\\
			F_{\mathcal{D}}^{(2)}=& -\mathbf{u}\cdot\nabla  B_3+\mathbf{B}\cdot \nabla u_3
			-2B_3\operatorname{div}\mathbf{u}+\mathbf{B}\cdot\partial_{x_3}\mathbf{u},\\
			F_{\mathcal{D}}^{(3)}=& -\frac12\mathbf{u}\cdot\nabla |\mathbf{B}|^2+\mathbf{B}\cdot (\mathbf{B}\cdot \nabla \mathbf{u})-|\mathbf{B}|^2\operatorname{div}\mathbf{u}.
		\end{align*}
		
		By Proposition \ref{lem bftf} and the property of the pressure function $P$, the first part is bounded by
		\begin{align*}
			\|\overline{F_{\mathcal{D}}^{(1)}} \|_{\dot{H}^{s-1}}=&\| \overline{(1-P^{\prime}(1+a))\operatorname{div}\mathbf{u}}+\overline{P^{\prime}(1+a)(\mathbf{u}\cdot\nabla a+a\operatorname{div}\mathbf{u})} \|_{\dot{H}^{s-1}} \\
			\lesssim & \|a \|_{H^{s-1}}\|\operatorname{div}\mathbf{u} \|_{H^{s-1}}+(1+\|a \|_{H^{s-1}})(\|a \|_{H^{s-1}}\|\nabla \mathbf{u} \|_{H^{s-1}}+\| a\|_{H^{s}}\| \mathbf{u}\|_{H^{2}} ) \\
			\lesssim   & (1+\|a \|_{H^{s-1}})(\|a \|_{H^{s-1}}\|\nabla \mathbf{u} \|_{H^{s-1}}+\| a\|_{H^{s}}\|\nabla \mathbf{u}\|_{H^{s-2}} ) .
		\end{align*}
		
		Turning to the second component, which consists of bilinear interactions between the velocity and magnetic fields, an application of standard product estimates yields
		\begin{align*}
			\|\overline{F_{\mathcal{D}}^{(2)}} \|_{\dot{H}^{s-1}}=&\| -\overline{\mathbf{u}\cdot\nabla B_3}+\overline{\mathbf{B}\cdot \nabla u_3}
			-2\overline{B_3\operatorname{div}\mathbf{u}}+\overline{\mathbf{B}\cdot\partial_{x_3}\mathbf{u}}\|_{\dot{H}^{s-1}}\\
			\lesssim & \| \mathbf{u}\|_{H^{s-1}}\|\nabla B_3 \|_{H^{2}}+\| \mathbf{u}\|_{H^{2}}\|\nabla B_3 \|_{H^{s-1}}+\| \mathbf{B}  \|_{H^{2}}\|\nabla \mathbf{u}\|_{H^{s-1}}+\| \mathbf{B}  \|_{H^{s-1}}\|\nabla \mathbf{u}\|_{H^{2}}\\
			&+\| \mathbf{B}  \|_{H^{s-1}}\|\nabla \mathbf{u}\|_{H^{2}}+\| \mathbf{B}  \|_{H^{2}}\|\nabla \mathbf{u}\|_{H^{s-1}} \\
			\lesssim &  \| \mathbf{B} \|_{H^{s-1}}\|\nabla \mathbf{u}\|_{H^{s-1}}+\|\mathbf{B}\|_{H^{s}}\| \mathbf{u}\|_{H^{s-2}} .
		\end{align*}
		
		Similarly, the third term can be bounded by
		\begin{align*}
			\|\overline{F_{\mathcal{D}}^{(3)}} \|_{\dot{H}^{s-1}}=&\|-\frac12\overline{\mathbf{u}\cdot\nabla |\mathbf{B}|^2}+\overline{\mathbf{B}\cdot (\mathbf{B}\cdot \nabla \mathbf{u})}-\overline{|\mathbf{B}|^2\operatorname{div}\mathbf{u}} \|_{\dot{H}^{s-1}}  \\
			\lesssim &   \| \mathbf{u}\|_{H^{s-1}}\|\mathbf{B} \|^2_{H^{3}}+\| \mathbf{u}\|_{H^{2}}\| \mathbf{B}\|_{H^{3}}\| \mathbf{B} \|_{H^{s}}+\| \mathbf{B}\|_{H^{s-1}}\| \mathbf{B}\|_{H^{2}}\| \nabla\mathbf{u} \|_{H^{2}}\\
			& +\| \mathbf{B}\|_{H^{2}}(\| \mathbf{B}\|_{H^{2}}\| \nabla\mathbf{u} \|_{H^{s-1}}+\| \mathbf{B}\|_{H^{s-1}}\| \nabla\mathbf{u} \|_{H^{2}})\\
			\lesssim & \|\mathbf{B} \|^2_{H^{s-1}}\|\nabla \mathbf{u} \|_{H^{s-1}}+\| \mathbf{B}\|_{H^{3}}\| \mathbf{B}\|_{H^{s}}\|\nabla \mathbf{u} \|_{H^{s-2}}.
		\end{align*}
		
		Combining the estimates for  $\overline{F_{\mathcal{D}}^{(1)}}$, $\overline{F_{\mathcal{D}}^{(2)}}$ and $\overline{F_{\mathcal{D}}^{(3)}}$ yield 
		\begin{align}\label{est FbD hs-1}
			\|\overline{F_{\mathcal{D}} } \|_{\dot{H}^{s-1}}\lesssim (1+\|(a,\mathbf{B}) \|_{H^{s-1}})(\|(a,\mathbf{B}) \|_{H^{s-1}}\|\nabla \mathbf{u} \|_{H^{s-1}}+\| (a,\mathbf{B})\|_{H^{s}}\|\nabla \mathbf{u}\|_{H^{s-2}} ) .
		\end{align}
		
		Multiplying by the corresponding time weight and integrating over time, we deduce
		\begin{align*}
			\int_{0}^t&(\mathcal{T}+\tau)^{1-\sigma} \|\overline{F_{\mathcal{D}}} \|_{\dot{H}^{s-1}}\| \overline{\mathcal{D}} \|_{\dot{H}^{s-1}} d\tau\\
			\lesssim& \sup_{\tau\in[0,t]}(1+\|(a,\mathbf{B}) \|_{H^{s-1}})\|(a,\mathbf{B}) \|_{H^{s-1}}\int_{0}^t(\mathcal{T}+\tau)^{1-\sigma} \|\nabla \mathbf{u} \|_{H^{s-1}}  \| \overline{\mathcal{D}} \|_{\dot{H}^{s-1}} d\tau\\
			&+  \sup_{\tau\in[0,t]}(\mathcal{T}+\tau)^{ -\frac{\sigma}{2}} (1+\|(a,\mathbf{B}) \|_{H^{s-1}})\| (a,\mathbf{B})\|_{H^{s}}\\
			&\quad\times\int_{0}^t  (\mathcal{T}+\tau)^{\frac{2-\sigma}{2}} \|\nabla \mathbf{u}\|_{H^{s-2}}    (\mathcal{T}+\tau)^{\frac{1-\sigma}{2}} \| \overline{\mathcal{D}} \|_{\dot{H}^{s-1}} d\tau\\
			\lesssim & E_1^{\frac32}(t)+E_0^{\frac12}(t)E_1^{\frac12}(t)E_2^{\frac12}(t).
		\end{align*}
		
		Inserting this estimate, along with the bound for $\overline{F_{\overline{\mathbf{u}}}}$, into the energy identity \eqref{est dtbubD hs-1} and integrating over the time interval $[0,t]$, we conclude that
		\begin{align}\label{est buD hs-1}
			(\mathcal{T}+t)^{1-\sigma}&(\|\overline{\mathbf{u}} \|^2_{\dot{H}^{s-1}} + \|\overline{\mathcal{D}} \|^2_{\dot{H}^{s-1}})+\mu\int_{0}^t(\mathcal{T}+\tau)^{1-\sigma}\|\nabla \overline{\mathbf{u}} \|^2_{\dot{H}^{s-1}}d\tau \\
			\lesssim& \mathcal{T}^{1-\sigma}\|(a_0,\mathbf{u}_0,\mathbf{B}_0) \|_{H^{s-1}}^2+ \mathcal{T}^{-1}\int_{0}^t(\mathcal{T}+\tau)^{1-\sigma}\|\overline{\mathcal{D}} \|^2_{\dot{H}^{s-1}}d\tau\nonumber
			\\
			&+E_1(t)(E_1^{\frac12}(t)+N_1^{\frac12}(t))+E_0^{\frac12}(t)E_1^{\frac12}(t)E_2^{\frac12}(t).\nonumber
		\end{align}
		
		To extract the hidden damping effect of $\overline{\mathcal{D}}$, we consider the cross term
		\begin{align}\label{est duD hs-2}
			\frac{d}{d\tau}&\Big[(\mathcal{T}+\tau)^{1-\sigma}\int_{\mathbb{T}^3}\nabla^{s-2}\overline{\mathbf{u}}\cdot \nabla^{s-2}\nabla\overline{\mathcal{D}}  dx\Big]-(1-\sigma)(\mathcal{T}+\tau)^{-\sigma}\int_{\mathbb{T}^3}\nabla^{s-2}\overline{\mathbf{u}}\cdot \nabla^{s-2}\nabla\overline{\mathcal{D}}  dx\\
			=&(\mathcal{T}+\tau)^{1-\sigma}\int_{\mathbb{T}^3}\nabla^{s-2}\partial_{\tau}\overline{\mathbf{u}}\cdot \nabla^{s-2}\nabla\overline{\mathcal{D}}  dx+(\mathcal{T}+\tau)^{1-\sigma}\int_{\mathbb{T}^3}\nabla^{s-2}\overline{\mathbf{u}}\cdot \nabla^{s-2}\nabla\partial_{\tau}\overline{\mathcal{D}}  dx.\nonumber
		\end{align}
		
		Substituting the equations of \eqref{equ bu} and \eqref{equ bD} into the right-hand side of \eqref{est duD hs-2} and applying standard Sobolev estimates, we deduce that
		\begin{align*}
			(\mathcal{T}+\tau)^{1-\sigma}&\int_{\mathbb{T}^3}\nabla^{s-2}\partial_{\tau}\overline{\mathbf{u}}\cdot \nabla^{s-2}\nabla\overline{\mathcal{D}}+\nabla^{s-2}\overline{\mathbf{u}}\cdot \nabla^{s-2}\nabla\partial_{\tau}\overline{\mathcal{D}} dx+(\mathcal{T}+\tau)^{1-\sigma}\|\overline{\mathcal{D}} \|^2_{\dot{H}^{s-1}}\\
			=&(\mathcal{T}+\tau)^{1-\sigma}\int_{\mathbb{T}^3}\nabla^{s-2} (\mu\Delta  \overline{\mathbf{u}}+(\lambda+\mu)\nabla (\operatorname{div} \overline{\mathbf{u}})+\overline{F_{\overline{\mathbf{u}}}} )\cdot \nabla^{s-2}\nabla\overline{\mathcal{D}}dx\\
			&+(\mathcal{T}+\tau)^{1-\sigma}\int_{\mathbb{T}^3}\nabla^{s-2}\overline{\mathbf{u}}\cdot \nabla^{s-2}\nabla (-2\overline{\operatorname{div}\mathbf{u}}+\overline{F_{\mathcal{D}}} ) dx\\
			\leq &\frac{1}{4}(\mathcal{T}+\tau)^{1-\sigma}\|\overline{\mathcal{D}} \|_{\dot{H}^{s-1}}^2+ C(\mathcal{T}+\tau)^{1-\sigma}\|\nabla \mathbf{u} \|_{\dot{H}^{s-1}}^2\\
			&+C(\mathcal{T}+\tau)^{1-\sigma}\|\overline{F_{\overline{\mathbf{u}}}}\|_{\dot{H}^{s-2}}^2+C(\mathcal{T}+\tau)^{1-\sigma}\|\overline{F_{\mathcal{D}}} \|_{\dot{H}^{s-2}}^2.
		\end{align*}
		
		Plugging the above estimate into \eqref{est duD hs-2} and invoking the nonlinear estimates \eqref{est Fbu hs-2} and \eqref{est FbD hs-1}, it follows that
		\begin{align*}
			\frac{d}{d\tau}&\Big[(\mathcal{T}+\tau)^{1-\sigma}\int_{\mathbb{T}^3}\nabla^{s-2}\overline{\mathbf{u}}\cdot \nabla^{s-2}\nabla\overline{\mathcal{D}}  dx\Big]+\frac12(\mathcal{T}+\tau)^{1-\sigma} \|\overline{\mathcal{D}} \|^2_{\dot{H}^{s-1}}\leq C(\mathcal{T}+\tau)^{1-\sigma}\|\nabla \mathbf{u} \|_{\dot{H}^{s-1}}^2\\
			&+C(\mathcal{T}+\tau)^{1-\sigma}(1+\| (a,\mathbf{B} ) \|_{H^{s-1}})^2\|(a,\mathbf{u},  \mathbf{B} ) \|^2_{H^{s-1}}(\|\overline{\mathcal{D}} \|_{\dot{H}^{s-1}}+\|\nabla\mathbf{u} \|_{H^{s-1}}+\|(\widetilde{a},\partial_{x_3}\widetilde{\mathbf{B}}) \|_{H^{s-2}})^2\\
			&+C(\mathcal{T}+\tau)^{1-\sigma} (1+\|(a,\mathbf{B})  \|_{H^{s-2}})^2(\|(a,\mathbf{B}) \|_{H^{s-2}}\|\nabla \mathbf{u} \|_{H^{s-2}}+\| (a,\mathbf{B})\|_{H^{s-1}}\|\nabla \mathbf{u}\|_{H^{s-3}})^2.
		\end{align*}
		
		Integrating over the time interval $[0,t]$, we obtain
		\begin{align}\label{est uD hs-2}
			(\mathcal{T}+t)^{1-\sigma}&\int_{\mathbb{T}^3}\nabla^{s-2}\overline{\mathbf{u}}\cdot \nabla^{s-2}\nabla\overline{\mathcal{D}}  dx+ \int_{0}^t(\mathcal{T}+\tau)^{1-\sigma} \|\overline{\mathcal{D}} \|^2_{\dot{H}^{s-1}}d\tau\\
			\leq& \mathcal{T}^{1-\sigma}\|(a_0,\mathbf{u}_0,\mathbf{B}_0) \|_{H^{s-1}}^2+ C\int_{0}^t(\mathcal{T}+\tau)^{1-\sigma}\|\nabla \mathbf{u} \|_{\dot{H}^{s-1}}^2d\tau\nonumber\\
			&+E_1(t)(E_1^{\frac12}(t)+N_1^{\frac12}(t))+E_0^{\frac12}(t)E_1^{\frac12}(t)E_2^{\frac12}(t).\nonumber
		\end{align}
		
		For the first term on the left-hand side, applying the Cauchy-Schwarz and Young inequalities together with the Poincar\'e inequality \eqref{Poi1'} of Lemma \ref{lem poin} yields
		\begin{align*}
			\left|(\mathcal{T}+t)^{1-\sigma}\int_{\mathbb{T}^3}\nabla^{s-2}\overline{\mathbf{u}}\cdot \nabla^{s-2}\nabla\overline{\mathcal{D}}  dx\right|\leq (\mathcal{T}+t)^{1-\sigma}\|\overline{\mathbf{u}} \|^2_{\dot{H}^{s-1}}+\frac12(\mathcal{T}+t)^{1-\sigma}\|\overline{\mathcal{D}} \|_{\dot{H}^{s-1}}^2.
		\end{align*}
		
		Multiplying   \eqref{est uD hs-2}  by a suitably  small constant  $\gamma=\min\{\mu/(4C),1/2\}$ and adding the result to \eqref{est buD hs-1}, we have
		\begin{align}\label{est buD hs-1F}
			(\mathcal{T}+t)^{1-\sigma}&(\|\overline{\mathbf{u}} \|^2_{\dot{H}^{s-1}} + \|\overline{\mathcal{D}} \|^2_{\dot{H}^{s-1}})+\int_{0}^t(\mathcal{T}+\tau)^{1-\sigma}\|\nabla \overline{\mathbf{u}} \|^2_{\dot{H}^{s-1}}+\gamma\| \overline{\mathcal{D}} \|^2_{\dot{H}^{s-1}} d\tau \\
			\lesssim& \mathcal{T}^{1-\sigma}\|(a_0,\mathbf{u}_0,\mathbf{B}_0) \|_{H^{s-1}}^2+E_1(t)(E_1^{\frac12}(t)+N_1^{\frac12}(t))+E_0^{\frac12}(t)E_1^{\frac12}(t)E_2^{\frac12}(t).\nonumber
		\end{align}
		Here, the time shift parameter $\mathcal{T}$ is chosen sufficiently large such that $\mathcal{T}^{-1} \leq \frac{\gamma}{8C}$.
		
		From \eqref{est ba hs-1}, \eqref{est b hs-1}, \eqref{est buD hs-1F} and the definition of $\overline{E_1}(t)$, we get
		\begin{align*}
			\overline{E_1}(t)\lesssim   &\mathcal{T}^{1-\sigma}\|(a_0,\mathbf{u}_0,\mathbf{B}_0) \|_{H^{s-1}}^2+\mathcal{T}^{\frac{2\sigma-1}{2}}E_0^{\frac12}(t)E_2^{\frac12}(t)+(E_0(t)+E_1(t)+E_2(t)+N_1(t))^{\frac32}.
		\end{align*}
	\end{proof}
	\subsection{Estimate of $\widetilde{E_1}(t)$}\label{Sec est tE1}
	This part addresses the oscillatory components $(\widetilde a,\widetilde{\mathbf{u}},\widetilde{\mathbf{B}})$. Here the Poincar\'e inequality in the $x_3$-direction provides effective magnetic damping. However, the decay available for $\widetilde{\mathbf{B}}$ is one derivative weaker than the top nonlinear coupling. To circumvent this derivative loss, we introduce the  weights $P'(1+\overline a)/(1+\overline a)^2$ and $1/(1+\overline a)$ into the density and magnetic energies, respectively. As anticipated in the introduction, this choice of weights guarantees the cancellation of the highest-order magnetic tension and magnetic pressure terms against those in the velocity equation.
	
	\begin{lem}\label{est tE1}
		Under the assumptions of Theorem \ref{thm1} and assumption \eqref{ass}, there holds
		\begin{align*}
			\widetilde{E}_1(t)\lesssim \mathcal{T}^{1-\sigma}\|(a_0,\mathbf{u}_0,\mathbf{B}_0) \|_{H^{s-1}}^2+ E_0^{\frac12}(t)N_1^{\frac12}(t)+E_{total}^{\frac32}(t).
		\end{align*}
	\end{lem}
	\begin{proof}
		We begin by establishing the estimate for the oscillatory density $\widetilde{a}$. Applying the operator $\nabla^{s-1}$ to the equation for $\widetilde{a}$, and taking the corresponding weighted $L^2$ inner product, we deduce
		\begin{align*}
			\frac12\frac{d}{d\tau}&\Big[ (\mathcal{T}+\tau)^{1-\sigma}\int_{\mathbb{T}^3}\frac{P^{\prime}(1+\overline{a})}{(1+\overline{a})^2}|\nabla^{s-1} \widetilde{a}|^2dx\Big]+(\mathcal{T}+\tau)^{1-\sigma}\int_{\mathbb{T}^3}\frac{P^{\prime}(1+\overline{a})}{(1+\overline{a})} \nabla^{s-1}\operatorname{div}\widetilde{\mathbf{u}}\cdot \nabla^{s-1} \widetilde{a}dx\\
			=&\frac{1-\sigma}{2} (\mathcal{T}+\tau)^{-\sigma}\int_{\mathbb{T}^3} \frac{P^{\prime}(1+\overline{a})}{(1+\overline{a})^2}|\nabla^{s-1}\widetilde{a}|^2dx+\mathcal{A}_1+\mathcal{A}_2+\mathcal{A}_3, 
		\end{align*}
		where the residual terms are defined as
		\begin{align*}
			\mathcal{A}_1=&\frac12(\mathcal{T}+\tau)^{1-\sigma}\int_{\mathbb{T}^3}\partial_{\tau}\Big(\frac{P^{\prime}(1+\overline{a})}{(1+\overline{a})^2} \Big)|\nabla^{s-1}\widetilde{a}|^2 dx,\\
			\mathcal{A}_2=&-(\mathcal{T}+\tau)^{1-\sigma}\int_{\mathbb{T}^3}\frac{P^{\prime}(1+\overline{a})}{(1+\overline{a})^2} ([\nabla^{s-1},\overline{a}]\operatorname{div}\widetilde{\mathbf{u}} )\cdot\nabla^{s-1} \widetilde{a}dx,\\
			\mathcal{A}_3=&- (\mathcal{T}+\tau)^{1-\sigma}\int_{\mathbb{T}^3}\frac{P^{\prime}(1+\overline{a})}{(1+\overline{a})^2} \nabla^{s-1}(\widetilde{a}\operatorname{div}\mathbf{u}+\widetilde{\mathbf{u}\cdot\nabla a})\cdot \nabla^{s-1}\widetilde{a} dx.
		\end{align*}
		Note that we have decomposed the nonlinear term $\nabla^{s-1}\widetilde{(1+a)\operatorname{div}\mathbf{u}}$ to extract its principal part  $(1+\overline{a})\nabla^{s-1}\operatorname{div}\widetilde{\mathbf{u}}$ which is shifted to the left-hand side to be eliminated later by the equation for $\mathbf{u}$. 
		
		For the first term on the right-hand side, integrating over time and applying  Sobolev interpolation, we have
		\begin{align*}
			\frac{1-\sigma}{2}&\int_{0}^t\int_{\mathbb{T}^3} (\mathcal{T}+\tau)^{-\sigma} \frac{P^{\prime}(1+\overline{a})}{(1+\overline{a})^2}|\nabla^{s-1}\widetilde{a}|^2dxd\tau\\
			\lesssim& \int_{0}^t(\mathcal{T}+\tau)^{-\sigma}\|\widetilde{a} \|_{H^{s-1}}^2d\tau\\
			\lesssim&\int_{0}^t (\mathcal{T}+\tau)^{-\sigma}\|\widetilde{a} \|_{H^{s}}\|\widetilde{a} \|_{H^{s-2}}d\tau\\
			\lesssim&\int_{0}^t (\mathcal{T}+\tau)^{\frac{-1-\sigma}{2}}\|\widetilde{a} \|_{H^{s}}(\mathcal{T}+\tau)^{\frac{1-\sigma}{2}}(\|\widetilde{d} \|_{H^{s-2}}+\|\partial_{x_3}\mathbf{B} \|_{H^{s-2}})d\tau\\
			\lesssim &E_0^{\frac12}(t)N_1^{\frac12}(t).
		\end{align*}
		
		To bound $\mathcal{A}_1$, we substitute the continuity equation for $\overline{a}$ into the time derivative of the weight, which yields
		\begin{align*}
			\int_{0}^t\mathcal{A}_1d\tau=&-\frac12\int_{0}^t\int_{\mathbb{T}^3} (\mathcal{T}+\tau)^{1-\sigma} \Big(\frac{P^{\prime}(1+\overline{a})}{(1+\overline{a})^2} \Big)^{\prime}(\operatorname{div}\overline{\mathbf{u}}+\overline{\mathbf{u}\cdot\nabla a}+\overline{a\operatorname{div}\mathbf{u}} )|\nabla^{s-1}\widetilde{a}|^2 dxd\tau\\
			\lesssim &\sup_{\tau\in[0,t]}(\mathcal{T}+\tau)^{1-\sigma}\|\widetilde{a} \|_{H^{s-1}}^2(1+\|a \|_{H^{3}} ) \int_{0}^t\| \mathbf{u}\|_{H^{3}}d\tau\\
			\lesssim & E_1(t)E_2^{\frac12}(t).
		\end{align*}
		
		For the commutator term $\mathcal{A}_2$, integrating by parts yields
		\begin{align*}
			\int_{0}^t\mathcal{A}_2d\tau=&\int_{0}^t\int_{\mathbb{T}^3}(\mathcal{T}+\tau)^{1-\sigma}\nabla\Big(\frac{P^{\prime}(1+\overline{a})}{(1+\overline{a})^2} [\nabla^{s-1},\overline{a}]\operatorname{div}\widetilde{\mathbf{u}} \Big)\cdot\nabla^{s-2} \widetilde{a}dxd\tau \\
			\lesssim &\sup_{\tau\in[0,t]}\| \overline{a}\|_{H^{3}}\int_{0}^t(\mathcal{T}+\tau)^{1-\sigma}\| \operatorname{div}\widetilde{\mathbf{u}}\|_{H^{s-1}}\|\widetilde{a} \|_{H^{s-2}}d\tau\\
			&+\sup_{\tau\in[0,t]}(\mathcal{T}+\tau)^{-\frac{\sigma}{2}}\| \overline{a}\|_{H^{s}} \int_{0}^t(\mathcal{T}+\tau)^{1-\frac{\sigma}{2}}\|\operatorname{div}\widetilde{\mathbf{u}} \|_{H^{3}} \|\widetilde{a} \|_{H^{s-2}}d\tau\\
			\lesssim & (E_0(t)+E_1(t)+E_2(t))N_1^{\frac12}(t).
		\end{align*}
		
		For the remaining nonlinear term $\mathcal{A}_3$, we decompose it as follows:
		\begin{align*}
			\int_{0}^t\mathcal{A}_3d\tau=&-\int_{0}^t\int_{\mathbb{T}^3} (\mathcal{T}+\tau)^{1-\sigma}\frac{P^{\prime}(1+\overline{a})}{(1+\overline{a})^2} \big(\nabla^{s-1}(\widetilde{a}\operatorname{div}\mathbf{u})+\widetilde{\mathbf{u}\cdot\nabla \nabla^{s-1} a}\big) \cdot \nabla^{s-1}\widetilde{a} dxd\tau\\
			&+\int_{0}^t\int_{\mathbb{T}^3}(\mathcal{T}+\tau)^{1-\sigma}\frac{P^{\prime}(1+\overline{a})}{(1+\overline{a})^2} \nabla[\nabla^{s-1 },\mathbf{u}\cdot\nabla ]a\cdot\nabla^{s-2}\widetilde{a}dxd\tau\\
			\lesssim &\int_{0}^t (\mathcal{T}+\tau)^{1-\sigma}\big(\|\widetilde{a} \|_{H^{s-1}}\|\operatorname{div}\mathbf{u} \|_{H^{2}}+\|\widetilde{a} \|_{H^{2}}\|\operatorname{div}\widetilde{\mathbf{u}} \|_{H^{s-1}}+\|\mathbf{u} \|_{H^{2}}\|a \|_{H^{s}}\big)\|\widetilde{a} \|_{H^{s-1}}d\tau\\
			&+\int_{0}^t(\mathcal{T}+\tau)^{1-\sigma}\big(\|\nabla\mathbf{u} \|_{H^{s-1}}\|a \|_{H^{4}}+\|\widetilde{\mathbf{u}} \|_{H^{3}}\|a \|_{H^{s}}\big)\|\widetilde{a} \|_{H^{s-2}}d\tau\\
			\lesssim &\sup_{\tau\in[0,t]} (\mathcal{T}+\tau)^{1-\sigma}\|\widetilde{a} \|_{H^{s-1}}^2\int_{0}^t \| \nabla\mathbf{u}\|_{H^{2}}d\tau\\
			&+\sup_{\tau\in[0,t]}\|a \|_{H^{s-1}}\int_{0}^t(\mathcal{T}+\tau)^{1-\sigma}\|\widetilde{a} \|_{H^{s-2}} \| \nabla\mathbf{u}\|_{H^{s-1}}d\tau\\
			&+ \sup_{\tau\in[0,t]}(\mathcal{T}+\tau)^{\frac12-\sigma}\|a\|_{H^{s}}\| \widetilde{ a} \|_{H^{s-1}}\int_{0}^t(\mathcal{T}+\tau)^{\frac12} \| \nabla\mathbf{u}\|_{H^{3}} d\tau\\
			\lesssim & E_1(t)E_2^{\frac12}(t)+E_1(t)N_1^{\frac12}(t)+E_0^{\frac12}(t)E_1^{\frac12}(t)E_3^{\frac12}(t).
		\end{align*}
		
		Combining all the above estimates and integrating over time, we obtain
		\begin{align}\label{est ta hs-1}
			\frac12(\mathcal{T}+t)^{1-\sigma}&\int_{\mathbb{T}^3}\frac{P^{\prime}(1+\overline{a})}{(1+\overline{a})^2}|\nabla^{s-1} \widetilde{a}|^2dx\\
			&+\int_{0}^t\int_{\mathbb{T}^3}(\mathcal{T}+\tau)^{1-\sigma}\frac{P^{\prime}(1+\overline{a})}{(1+\overline{a})} \nabla^{s-1}\operatorname{div}\widetilde{\mathbf{u}}\cdot \nabla^{s-1} \widetilde{a}dxd\tau\nonumber\\
			\lesssim & \mathcal{T}^{1-\sigma}\|a_0 \|_{H^{s-1}}^2+ (E_0(t)+E_1(t)+E_2(t)+E_3(t)+N_1(t))^{\frac32}+E_0^{\frac12}(t)N^{\frac12}_1(t).\nonumber
		\end{align}
		
		Next, we consider the evolution equation for the oscillatory magnetic field $\widetilde{\mathbf{B}}$:
		\begin{align*}
			\partial_t\widetilde{\mathbf{B}}+\mathbf{e}_3\widetilde{\operatorname{div} \mathbf{u}}-\partial_{x_3} \mathbf{u}=-\widetilde{\mathbf{u}\cdot \nabla \mathbf{B}}+\widetilde{\mathbf{B}\cdot \nabla \mathbf{u}}-\widetilde{\mathbf{B}\operatorname{div}\mathbf{u}}.
		\end{align*}
		Applying the operator $\nabla^{s-1}$ and taking the $L^2$ inner product with the weighted function $\frac{1}{1+\overline{a}}\nabla^{s-1}\widetilde{\mathbf{B}}$, we derive that
		\begin{align*}
			\frac{1}{2}\frac{d}{d\tau}&\Big[(\mathcal{T}+\tau)^{1-\sigma}\int_{\mathbb{T}^3}\frac{1}{1+\overline{a}}|  \nabla^{s-1}\widetilde{\mathbf{B}} |^2dx\Big]\\
			&-(\mathcal{T}+\tau)^{1-\sigma}\int_{\mathbb{T}^3}\frac{1}{1+\overline{a}}\nabla^{s-1}(\partial_{x_3}\mathbf{u}-\mathbf{e}_3\operatorname{div} \widetilde{\mathbf{u}})\cdot\nabla^{s-1}\widetilde{\mathbf{B}}dx\\
			=&\mathcal{B}_1+\mathcal{B}_2+\mathcal{B}_3,
		\end{align*}
		where 
		\begin{align*}
			\mathcal{B}_1=&\frac{1-\sigma}{2}(\mathcal{T}+\tau)^{-\sigma}\int_{\mathbb{T}^3}\frac{1}{1+\overline{a}}|  \nabla^{s-1}\widetilde{\mathbf{B}} |^2dx,\\
			\mathcal{B}_2=&-\frac12(\mathcal{T}+\tau)^{1-\sigma}\int_{\mathbb{T}^3}\frac{\partial_t\overline{a}}{(1+\overline{a})^2}|  \nabla^{s-1}\widetilde{\mathbf{B}} |^2dx,\\
			\mathcal{B}_3=&(\mathcal{T}+\tau)^{1-\sigma}\int_{\mathbb{T}^3}\frac{1}{1+\overline{a}}\nabla^{s-1}(-\widetilde{\mathbf{u}\cdot \nabla \mathbf{B}}+\widetilde{\mathbf{B}\cdot \nabla \mathbf{u}}-\widetilde{\mathbf{B}\operatorname{div}\mathbf{u}} )\cdot\nabla^{s-1}\widetilde{\mathbf{B}} dx.
		\end{align*}
		Note that the second term on the left-hand side is kept here to be cancelled later by the linear part in the equation for $\mathbf{u}$.
		
		From the definition of $E_i(t)$, we have
		\begin{align*}
			\int_{0}^t\mathcal{B}_1d\tau=&\frac{1-\sigma}{2}\int_{0}^t\int_{\mathbb{T}^3}(\mathcal{T}+\tau)^{-\sigma}\frac{1}{1+\overline{a}}|  \nabla^{s-1}\widetilde{\mathbf{B}} |^2dxd\tau\\
			\lesssim &\int_{0}^t\int_{\mathbb{T}^3}(\mathcal{T}+\tau)^{-\sigma}|  \nabla^{s-1}\widetilde{\mathbf{B}} |^2dxd\tau\\
			\lesssim & \int_{0}^t(\mathcal{T}+\tau)^{\frac{1-\sigma}{2}}\|\widetilde{\mathbf{B}} \|_{H^{s-2}}(\mathcal{T}+\tau)^{\frac{-1-\sigma}{2}}\|\widetilde{\mathbf{B}} \|_{H^{s}}d\tau\\
			\lesssim & E_0^{\frac12}(t)N_1^{\frac12}(t).
		\end{align*} 
		
		To estimate $\mathcal{B}_2$, we invoke the evolution equation for $\overline{a}$ and apply standard Sobolev embedding, 
		\begin{align*}
			\int_{0}^t\mathcal{B}_2d\tau=& -\frac12\int_{0}^t\int_{\mathbb{T}^3}(\mathcal{T}+\tau)^{1-\sigma}\frac{\partial_{\tau}\overline{a}}{(1+\overline{a})^2}|  \nabla^{s-1}\widetilde{\mathbf{B}} |^2dxd\tau\\
			=& \frac12\int_{0}^t\int_{\mathbb{T}^3}(\mathcal{T}+\tau)^{1-\sigma}\frac{1}{(1+\overline{a})^2}(\operatorname{div}\overline{\mathbf{u}}+\overline{\mathbf{u}\cdot \nabla a}+\overline{a\operatorname{div}\mathbf{u}} )|  \nabla^{s-1}\widetilde{\mathbf{B}} |^2dxd\tau\\
			\lesssim & \sup_{\tau\in[0,t]}(\mathcal{T}+\tau)^{1-\sigma}\| \widetilde{\mathbf{B}}\|^2_{H^{s-1}}\int_{0}^t\| \nabla \mathbf{u}\|_{H^{2}}d\tau\\
			\lesssim & E_1(t)E_2^{\frac12}(t).
		\end{align*}
		
		For the last term $\mathcal{B}_3$, we decompose it into two parts:
		\begin{align*}
			\mathcal{B}_3=&(\mathcal{T}+\tau)^{1-\sigma}\int_{\mathbb{T}^3}\frac{1}{1+\overline{a}}\nabla^{s-1}(-\widetilde{\mathbf{u}\cdot \nabla \mathbf{B}})\cdot\nabla^{s-1}\widetilde{\mathbf{B}} dx\\
			&+(\mathcal{T}+\tau)^{1-\sigma}\int_{\mathbb{T}^3}\frac{1}{1+\overline{a}}\nabla^{s-1}(\widetilde{\mathbf{B}\cdot \nabla \mathbf{u}}-\widetilde{\mathbf{B}\operatorname{div}\mathbf{u}} )\cdot\nabla^{s-1}\widetilde{\mathbf{B}} dx\\
			=&  \mathcal{B}_{3,1}+\mathcal{B}_{3,2}.
		\end{align*}
		Using integration by parts, $\mathcal{B}_{3,1}$ is bounded by
		\begin{align*}
			\mathcal{B}_{3,1}=&-(\mathcal{T}+\tau)^{1-\sigma}\int_{\mathbb{T}^3}\frac{1}{1+\overline{a}}\nabla^{s-1}(\widetilde{\mathbf{u}}\cdot \nabla \overline{\mathbf{B}}+\mathbf{u}\cdot \nabla \widetilde{\mathbf{B}})\cdot\nabla^{s-1}\widetilde{\mathbf{B}} dx\\
			=&  (\mathcal{T}+\tau)^{1-\sigma}\sum_{k=2}^{s-1}\binom{s-1}{k}\int_{\mathbb{T}^3}\nabla\Big(\frac{1}{1+\overline{a}}(\nabla^{k}\widetilde{\mathbf{u}}\cdot \nabla \nabla^{s-1-k}\overline{\mathbf{B}})\Big)\cdot\nabla^{s-2}\widetilde{\mathbf{B}} dx\\
			&-(\mathcal{T}+\tau)^{1-\sigma} \int_{\mathbb{T}^3} \frac{1}{1+\overline{a}}( \widetilde{\mathbf{u}}\cdot \nabla \nabla^{s-1 }\overline{\mathbf{B}}+(s-1)\nabla \widetilde{\mathbf{u}}\cdot \nabla \nabla^{s-2}\overline{\mathbf{B}} )\cdot\nabla^{s-1}\widetilde{\mathbf{B}} dx\\
			&-(\mathcal{T}+\tau)^{1-\sigma}\sum_{k=1}^{s-1}\binom{s-1}{k}\int_{\mathbb{T}^3}\frac{1}{1+\overline{a}}( \nabla^{k}\mathbf{u}\cdot\nabla\nabla^{s-1-k}  \widetilde{\mathbf{B}})\cdot\nabla^{s-1}\widetilde{\mathbf{B}} dx\\
			&+\frac12(\mathcal{T}+\tau)^{1-\sigma}\int_{\mathbb{T}^3}\operatorname{div}\big(\frac{\mathbf{u} }{1+\overline{a}} \big) |\nabla^{s-1}\widetilde{\mathbf{B}}|^2 dx\\
			\lesssim & (\mathcal{T}+\tau)^{1-\sigma}\|\overline{\mathbf{B}} \|_{H^{s-1}}\|\nabla\mathbf{u} \|_{H^{s-1}}\|\widetilde{\mathbf{B}} \|_{H^{s-2}}\\
			&+(\mathcal{T}+\tau)^{1-\sigma}(\|\overline{\mathbf{B}}\|_{H^{s}}\|  \mathbf{u}\|_{H^{2}}+\|\overline{\mathbf{B}}\|_{H^{s-1}}\| \nabla \mathbf{u}\|_{H^{2}})\| \widetilde{\mathbf{B}}\|_{H^{s-1}}\\
			&+ (\mathcal{T}+\tau)^{1-\sigma}(\|\nabla\mathbf{u} \|_{H^{s-1}}\|\widetilde{\mathbf{B}} \|_{H^{3}}+\|\mathbf{u} \|_{H^{3}}\|\widetilde{\mathbf{B}} \|_{H^{s-1}})\|\widetilde{\mathbf{B}} \|_{H^{s-1}}\\
			\lesssim& (\mathcal{T}+\tau)^{1-\sigma}(\|\overline{\mathbf{B}} \|_{H^{s-1}}\|\nabla\mathbf{u} \|_{H^{s-1}}\|\widetilde{\mathbf{B}} \|_{H^{s-2}} +\|\mathbf{B}\|_{H^{s}}\| \nabla \mathbf{u}\|_{H^{2}}\| \widetilde{\mathbf{B}}\|_{H^{s-1}}).
		\end{align*}
		
		In $\mathcal{B}_{3,2}$, the highest-order derivatives of $\widetilde{\mathbf{u}}$ produce two bad terms: 
		\begin{align}\label{equ bad b2}
			(\mathcal{T}+\tau)^{1-\sigma}\int_{\mathbb{T}^3}\frac{1}{1+\overline{a}}(\overline{\mathbf{B}}\cdot \nabla \nabla^{s-1}\widetilde{\mathbf{u}}-\overline{\mathbf{B}}\nabla^{s-1}\operatorname{div}\widetilde{\mathbf{u}} )\cdot\nabla^{s-1}\widetilde{\mathbf{B}} dx.
		\end{align}
		However, these terms are exactly cancelled by the corresponding terms arising in the estimate for $\widetilde{\mathbf{u}}$:
		\begin{align*}
			(\mathcal{T}+\tau)^{1-\sigma}&\int_{\mathbb{T}^3}\overline{\frac{1}{1+a}}\nabla^{s-1}(\widetilde{\mathbf{B}\cdot\nabla \mathbf{B}}-\frac12\widetilde{\nabla |\mathbf{B}|^2})\nabla^{s-1}\widetilde{\mathbf{u}} dx\\
			&=-(\mathcal{T}+\tau)^{1-\sigma}\int_{\mathbb{T}^3}\frac{1}{1+\overline{a}}(\overline{\mathbf{B}}\cdot \nabla \nabla^{s-1}\widetilde{\mathbf{u}}-\overline{\mathbf{B}}\nabla^{s-1}\operatorname{div}\widetilde{\mathbf{u}} )\cdot\nabla^{s-1}\widetilde{\mathbf{B}} dx+
			\cdots.
		\end{align*}
		Subtracting these bad terms, the remaining lower-order terms in $\mathcal{B}_{3,2}$ are bounded as follows:
		\begin{align*}
			\mathcal{B}_{3,2}&-(\mathcal{T}+\tau)^{1-\sigma}\int_{\mathbb{T}^3}\frac{1}{1+\overline{a}}(\overline{\mathbf{B}}\cdot \nabla \nabla^{s-1}\widetilde{\mathbf{u}}-\overline{\mathbf{B}}\nabla^{s-1}\operatorname{div}\widetilde{\mathbf{u}} )\cdot\nabla^{s-1}\widetilde{\mathbf{B}} dx\\
			=& (\mathcal{T}+\tau)^{1-\sigma}\int_{\mathbb{T}^3}\frac{1}{1+\overline{a}}\nabla^{s-1}(\widetilde{\mathbf{B}}\cdot \nabla \mathbf{u}-\widetilde{\mathbf{B}}\operatorname{div}\mathbf{u} )\cdot\nabla^{s-1}\widetilde{\mathbf{B}} dx\\
			&-(\mathcal{T}+\tau)^{1-\sigma}\sum_{k=1}^{s-1}\binom{s-1}{k}\int_{\mathbb{T}^3}\nabla\Big(\frac{1}{1+\overline{a}} \nabla^{k}\overline{\mathbf{B}}\cdot \nabla \nabla^{s-1-k}\widetilde{\mathbf{u}}\Big)\cdot\nabla^{s-2}\widetilde{\mathbf{B}} dx\\
			&+(\mathcal{T}+\tau)^{1-\sigma}\sum_{k=1}^{s-1}\binom{s-1}{k}\int_{\mathbb{T}^3}\nabla\Big(\frac{1}{1+\overline{a}} \nabla^{k}\overline{\mathbf{B}}\nabla^{s-1-k}\operatorname{div}\widetilde{\mathbf{u}}\Big)\cdot\nabla^{s-2}\widetilde{\mathbf{B}} dx\\
			\lesssim  &(\mathcal{T}+\tau)^{1-\sigma}(\| \widetilde{\mathbf{B}}\|_{H^{2}}\|\nabla \mathbf{u} \|_{H^{s-1}}+\| \widetilde{\mathbf{B}}\|_{H^{s-1}}\|\nabla \mathbf{u} \|_{H^{2}} )\| \widetilde{\mathbf{B}}\|_{H^{s-1}} \\
			&+(\mathcal{T}+\tau)^{1-\sigma}(\| \overline{\mathbf{B}}\|_{H^{4}}\|\nabla \mathbf{u} \|_{H^{s-1}}+\| \overline{\mathbf{B}}\|_{H^{s}}\|\nabla \mathbf{u} \|_{H^{3}} )\| \widetilde{\mathbf{B}}\|_{H^{s-2}} .
		\end{align*}
		
		Combining the estimates of $\mathcal{B}_{3,1}$ and $\mathcal{B}_{3,2}$, we obtain
		\begin{align*}
			\int_{0}^t&\mathcal{B}_3d\tau-\int_{0}^t\int_{\mathbb{T}^3}(\mathcal{T}+\tau)^{1-\sigma}\frac{1}{1+\overline{a}}(\overline{\mathbf{B}}\cdot \nabla \nabla^{s-1}\widetilde{\mathbf{u}}-\overline{\mathbf{B}}\nabla^{s-1}\operatorname{div}\widetilde{\mathbf{u}} )\cdot\nabla^{s-1}\widetilde{\mathbf{B}} dxd\tau\\
			\lesssim & \sup_{\tau\in[0,t]}\|\overline{\mathbf{B}} \|_{H^{s-1}}	\int_{0}^t (\mathcal{T}+\tau)^{1-\sigma}\|\nabla\mathbf{u} \|_{H^{s-1}}\|\widetilde{\mathbf{B}} \|_{H^{s-2}} d\tau\\
			&+\sup_{\tau\in[0,t]}(\mathcal{T}+\tau)^{\frac{1-\sigma}{2}}\| \widetilde{\mathbf{B}}\|_{H^{s-1}} 	\int_{0}^t(\mathcal{T}+\tau)^{\frac{-1-\sigma}{2}} \|\mathbf{B}\|_{H^{s}}(\mathcal{T}+\tau)  \| \nabla \mathbf{u}\|_{H^{2}} d\tau\\
			&+\sup_{\tau\in[0,t]}(\mathcal{T}+\tau)^{\frac{1-\sigma}{2}}\| \widetilde{\mathbf{B}}\|_{H^{s-1}}\int_{0}^t(\mathcal{T}+\tau)^{\frac{1-\sigma}{2}}(\| \widetilde{\mathbf{B}}\|_{H^{2}}\|\nabla \mathbf{u} \|_{H^{s-1}}+\| \widetilde{\mathbf{B}}\|_{H^{s-1}}\|\nabla \mathbf{u} \|_{H^{2}} )d\tau \\
			&+\sup_{\tau\in[0,t]}(\mathcal{T}+\tau)^{ -\frac{\sigma}{2}}\| \overline{\mathbf{B}}\|_{H^{s}}\int_{0}^t(\mathcal{T}+\tau)^{\frac{1-2\sigma}{2}}\|\nabla \mathbf{u} \|_{H^{3}}(\mathcal{T}+\tau)^{\frac{1- \sigma}{2}} \| \widetilde{\mathbf{B}}\|_{H^{s-2}}d\tau\\
			\lesssim& (E_0(t)+E_1(t)+E_2(t)+E_3(t)+N_1(t))^{\frac32} .
		\end{align*}
		
		Collecting these estimates and integrating over time, it follows that
		\begin{align}\label{est tb hs-1}
			\frac{1}{2}&\Big[(\mathcal{T}+t)^{1-\sigma}\int_{\mathbb{T}^3}\frac{1}{1+\overline{a}}|  \nabla^{s-1}\widetilde{\mathbf{B}} |^2dx\Big]\\
			&-\int_{0}^{t}\int_{\mathbb{T}^3}(\mathcal{T}+\tau)^{1-\sigma}\frac{1}{1+\overline{a}}\nabla^{s-1}(\partial_{x_3}\mathbf{u}-\mathbf{e}_3\operatorname{div} \widetilde{\mathbf{u}})\cdot\nabla^{s-1}\widetilde{\mathbf{B}}dxd\tau\nonumber\\
			&-\int_{0}^{t}\int_{\mathbb{T}^3}(\mathcal{T}+\tau)^{1-\sigma}\frac{1}{1+\overline{a}}(\overline{\mathbf{B}}\cdot \nabla \nabla^{s-1}\widetilde{\mathbf{u}}-\overline{\mathbf{B}}\nabla^{s-1}\operatorname{div}\widetilde{\mathbf{u}} )\cdot\nabla^{s-1}\widetilde{\mathbf{B}} dxd\tau\nonumber\\
			\lesssim & \mathcal{T}^{1-\sigma}\|(a_0,\mathbf{B}_0) \|_{H^{s-1}}^2+E_0^{\frac12}(t)N_1^{\frac12}(t)+(E_0(t)+E_1(t)+E_2(t)+E_3(t)+N_1(t))^{\frac32}.\nonumber
		\end{align}
		
		We now proceed to the momentum equation for the oscillatory velocity $\widetilde{\mathbf{u}}$:
		\begin{align*}
			\partial_t \widetilde{\mathbf{u}}  -
			\mu\Delta  \widetilde{\mathbf{u}}-(\lambda+\mu)\nabla (\operatorname{div} \widetilde{\mathbf{u}})+ \widetilde{h(a)\nabla a}
			-\widetilde{\rho^{-1}\partial_{x_3} \mathbf{B}}
			+\widetilde{\rho^{-1}\nabla  B_3}=\widetilde{F_{\mathbf{u}}},
		\end{align*} 
		where $h(a)=\frac{P^{\prime}(1+a)}{1+a}$ and the nonlinear term is given by
		\begin{align*}
			F_{\mathbf{u}}=&-\mathbf{u}\cdot \nabla \mathbf{u}+\rho^{-1}(\mathbf{B}\cdot \nabla \mathbf{B}-\frac12\nabla |\mathbf{B}|^2) -\mu I(a)\Delta \mathbf{u}-(\lambda +\mu) I(a)\nabla\operatorname{div}\mathbf{u}.
		\end{align*}

		Applying the operator $\nabla^{s-1}$ and taking the appropriately time-weighted $L^2$ inner product against $\nabla^{s-1}\widetilde{\mathbf{u}}$, we deduce that
		\begin{align*}
			\frac12\frac{d}{d\tau} & \Big((\mathcal{T}+\tau)^{1-\sigma}\|\widetilde{\mathbf{u}} \|^2_{\dot{H}^{s-1}}\Big)  -\frac{1-\sigma}{2}(\mathcal{T}+\tau)^{-\sigma}\|\widetilde{\mathbf{u}} \|^2_{\dot{H}^{s-1}}+
			\mu(\mathcal{T}+\tau)^{1-\sigma}\|\nabla \widetilde{\mathbf{u}} \|^2_{\dot{H}^{s-1}}\\
			&+
			(\lambda+\mu)(\mathcal{T}+\tau)^{1-\sigma}\| \operatorname{div}  \widetilde{\mathbf{u}} \|^2_{\dot{H}^{s-1}} 
			\\
			=&(\mathcal{T}+\tau)^{1-\sigma}\int_{\mathbb{T}^3}\nabla^{s-1}\Big[-\widetilde{h(a)\nabla a}
			+(\widetilde{\rho^{-1}\partial_{x_3} \mathbf{B}}
			-\widetilde{\rho^{-1}\nabla  B_3})+ \widetilde{F_{\mathbf{u}}} \Big]\cdot\nabla^{s-1}\widetilde{\mathbf{u}}dx\\
			=&(\mathcal{T}+\tau)^{1-\sigma}(\mathcal{U}_1+\mathcal{U}_2)+(\mathcal{T}+\tau)^{1-\sigma}\int_{\mathbb{T}^3}\nabla^{s-1}\widetilde{F_{\mathbf{u}}}\cdot\nabla^{s-1}\widetilde{\mathbf{u}}dx,
		\end{align*}
		the linear coupling terms are denoted by
		\begin{align*}
			\mathcal{U}_1=&-\int_{\mathbb{T}^3}\nabla^{s-1}\big(\widetilde{h(a)\nabla a}
			\big)\cdot\nabla^{s-1}\widetilde{\mathbf{u}}dx,\\
			\mathcal{U}_2=& \int_{\mathbb{T}^3}\nabla^{s-1}\big(
			\widetilde{\rho^{-1}\partial_{x_3} \mathbf{B}}
			-\widetilde{\rho^{-1}\nabla  B_3} \big)\cdot\nabla^{s-1}\widetilde{\mathbf{u}}dx.
		\end{align*}
		
		To absorb the negative term arising from the time weight on the left-hand side, we utilize the dissipative term via the Poincar\'e inequality:
		\begin{align*}
			\frac{1-\sigma}{2}(\mathcal{T}+\tau)^{-\sigma}\|\widetilde{\mathbf{u}} \|^2_{H^{s-1}}\leq \frac{1-\sigma}{2\mathcal{T}}(\mathcal{T}+\tau)^{1-\sigma}\|\nabla \widetilde{\mathbf{u}} \|^2_{H^{s-1}}\\
			\leq \frac{\mu}{4}(\mathcal{T}+\tau)^{1-\sigma}\|\nabla \widetilde{\mathbf{u}} \|^2_{H^{s-1}},
		\end{align*}
		provided the time shift is chosen large enough such that $\mathcal{T} \geq \frac{2(1-\sigma)}{\mu}$. 
		
		To handle $\mathcal{U}_1$, we invoke the Taylor expansions \eqref{equ bf} and \eqref{equ tf} from Proposition \ref{lem bftf} to extract the principal linear coupling. This allows us to rewrite $\mathcal{U}_1$ as
		\begin{align*}
			\mathcal{U}_1
			=&-\int_{\mathbb{T}^3}\nabla^{s-1}\Big[\overline{h(a)}  \nabla \widetilde{a}+\widetilde{h(a)} \nabla a \Big]\cdot\nabla^{s-1}\widetilde{\mathbf{u}}dx\\
			=&-\int_{\mathbb{T}^3}\nabla^{s-1}\Big[\big(h(\overline{a})+\frac{1}{2}h^{\prime\prime}(\overline{a})\overline{\widetilde{a}^2} +\overline{R_{h}(\overline{a},\widetilde{a})\widetilde{a}^3}\big)\nabla \widetilde{a} \Big]\cdot\nabla^{s-1}\widetilde{\mathbf{u}}dx\\
			&+\int_{\mathbb{T}^3}\nabla^{s-2}\big(\widetilde{h(a)}  \nabla a \big)\cdot\nabla^{s}\widetilde{\mathbf{u}}dx\\
			=&\int_{\mathbb{T}^3}h(\overline{a})\nabla^{s-1} \widetilde{a} \nabla^{s-1}\operatorname{div}\widetilde{\mathbf{u}} dx-\int_{\mathbb{T}^3}\nabla^{s-2} \widetilde{a}\nabla\big(\nabla h(\overline{a})\cdot  \nabla^{s-1}\widetilde{\mathbf{u}}\big)dx\\
			&+\sum_{k=1}^{s-2}\binom{s-1}{k}\int_{\mathbb{T}^3}\nabla^{s-1-k} \widetilde{a}\nabla \cdot(\nabla^{k}h(\overline{a}) \nabla^{s-1}\widetilde{\mathbf{u}}) dx-\int_{\mathbb{T}^3}\nabla^{s-1}h(\overline{a})  \nabla\widetilde{a} \cdot  \nabla^{s-1}\widetilde{\mathbf{u}}  dx\\
			& +\int_{\mathbb{T}^3}\nabla^{s-2}\Big[\big(\frac{1}{2}h^{\prime\prime}(\overline{a})\overline{\widetilde{a}^2} +\overline{R_{h}(\overline{a},\widetilde{a})\widetilde{a}^3}\big)\nabla \widetilde{a}+\widetilde{h(a)} \nabla a \Big]\cdot\nabla^{s}\widetilde{\mathbf{u}}dx,
		\end{align*}
		where the third equality follows from applying integration by parts twice:
		\begin{align*}
			-\int_{\mathbb{T}^3}&h(\overline{a})\nabla\nabla^{s-1} \widetilde{a}\cdot\nabla^{s-1}\widetilde{\mathbf{u}}dx\\
			=&\int_{\mathbb{T}^3}h(\overline{a})\nabla^{s-1} \widetilde{a} \cdot\nabla^{s-1}\operatorname{div}\widetilde{\mathbf{u}} dx+\int_{\mathbb{T}^3}\nabla^{s-1} \widetilde{a}\nabla h(\overline{a}) \cdot\nabla^{s-1} \widetilde{\mathbf{u}} dx\\
			=&\int_{\mathbb{T}^3}h(\overline{a})\nabla^{s-1} \widetilde{a}\cdot \nabla^{s-1}\operatorname{div}\widetilde{\mathbf{u}} dx-\int_{\mathbb{T}^3}\nabla^{s-2} \widetilde{a}\nabla\big(\nabla h(\overline{a})\cdot  \nabla^{s-1}\widetilde{\mathbf{u}}\big)dx.
		\end{align*}
		Coupling this principal part with the corresponding linear term previously retained in the estimate for $\widetilde{a}$, 
		\begin{align*}
			\mathcal{U}_1-\int_{\mathbb{T}^3}h(\overline{a}) \nabla^{s-1}\operatorname{div}\widetilde{\mathbf{u}} \nabla^{s-1} \widetilde{a}dx
			\lesssim  \|a \|_{H^{s-1}}\|\widetilde{a} \|_{H^{s-2}}\|\nabla\widetilde{\mathbf{u}} \|_{H^{s-1}}.
		\end{align*}
		
		Similarly, combining $\mathcal{U}_2$ with the linear terms arising from the estimate  of $\widetilde{\mathbf{B}}$ yields
		\begin{align*}
			\mathcal{U}_2&+\int_{\mathbb{T}^3}\frac{1}{1+\overline{a}}\nabla^{s-1}(\partial_{x_3}\mathbf{u}-\mathbf{e}_3\operatorname{div} \widetilde{\mathbf{u}})\cdot\nabla^{s-1}\widetilde{\mathbf{B}}dx\\
			=&\sum_{k=2}^{s-1}\binom{s-1}{k}\int_{\mathbb{T}^3}\nabla^{k}
			\frac{1}{1+\overline{a}}\nabla^{s-1-k}(\partial_{x_3} \widetilde{\mathbf{B}}
			-\nabla  \widetilde{\mathbf{B}}_3)\cdot\nabla^{s-1}\widetilde{\mathbf{u}}dx\\
			&-(s-1)\int_{\mathbb{T}^3}\nabla^{s-3}(\partial_{x_3} \widetilde{\mathbf{B}}
			-\nabla  \widetilde{\mathbf{B}}_3)\cdot\nabla\Big(\nabla
			\frac{1}{1+\overline{a}}\nabla^{s-1}\widetilde{\mathbf{u}}\Big)dx\\
			&-\int_{\mathbb{T}^3}\nabla^{s-2}\Big[\Big(\frac{1}{(1+\overline{a})^3}\overline{\widetilde{a}^2}+\overline{\mathcal{R}_{\frac{1}{1+a}}(\overline{a},\widetilde{a})\widetilde{a}^3}\Big)
			(\partial_{x_3} \widetilde{\mathbf{B}}
			-\nabla  \widetilde{\mathbf{B}}_3)\Big]\cdot\nabla^{s}\widetilde{\mathbf{u}}dx\\
			&-\int_{\mathbb{T}^3}\nabla^{s-2}
			\Big[\widetilde{\frac{1}{1+a}}(\partial_{x_3} \mathbf{B}
			-\nabla  B_3)\Big]\cdot\nabla^{s}\widetilde{\mathbf{u}}dx\\
			\lesssim & \|a \|_{H^{s-1}}\|\widetilde{\mathbf{B}} \|_{H^{s-2}}\|\nabla \mathbf{u} \|_{H^{s-1}}+(\|\widetilde{a} \|_{H^{s-2}}\|\mathbf{B} \|_{H^{3}}+\|\widetilde{a} \|_{H^{2}}\| \mathbf{B}\|_{H^{s-1}})\|\nabla \mathbf{u} \|_{H^{s-1}}\\
			\lesssim& \|(a,\mathbf{B}) \|_{H^{s-1}}\|(\widetilde{a},\widetilde{\mathbf{B}}) \|_{H^{s-2}}\|\nabla \mathbf{u} \|_{H^{s-1}}.
		\end{align*}
		
		We now address the nonlinear forcing $\widetilde{F_{\mathbf{u}}}$, decomposing it into convective, magnetic, and viscous components:
		\begin{align*}
			{F_{\mathbf{u}}}={F_{\mathbf{u}}}^{(1)}+{F_{\mathbf{u}}}^{(2)}+{F_{\mathbf{u}}}^{(3)},
		\end{align*}
		where 
		\begin{align*}
			{F_{\mathbf{u}}}^{(1)}=&-\mathbf{u}\cdot \nabla \mathbf{u},\\
			{F_{\mathbf{u}}}^{(2)}=&\rho^{-1}\left(\mathbf{B}\cdot\nabla \mathbf{B}-\frac12\nabla |\mathbf{B}|^2\right) ,\\
			{F_{\mathbf{u}}}^{(3)}=&-\mu I(a)\Delta \mathbf{u}-(\lambda +\mu)I(a)\nabla \operatorname{div}\mathbf{u}.
		\end{align*}
		By Proposition \ref{lem bftf},  the convective part is estimated as
		\begin{align*}
			\int_{\mathbb{T}^3}\nabla^{s-1}\widetilde{F_{\mathbf{u}}}^{(1)}\cdot\nabla^{s-1}\widetilde{\mathbf{u}}dx=&-\int_{\mathbb{T}^3}\nabla^{s-2}\Big(\widetilde{ \mathbf{u}\cdot \nabla \mathbf{u}}\Big) \cdot\nabla^{s }\widetilde{\mathbf{u}}dx\\
			\lesssim & \|\mathbf{u} \|_{H^{3}}\| \nabla \mathbf{u}\|_{H^{s-2}}\| \nabla \widetilde{\mathbf{u}}\|_{H^{s-1}}.
		\end{align*}
		
		To handle the magnetic nonlinearities $\widetilde{F_{\mathbf{u}}}^{(2)}$, we notationally write $\frac12\nabla |\mathbf{B}|^2$ as $\mathbf{B}\nabla\mathbf{B}$ to distinguish magnetic pressure from magnetic tension. Expanding  the density factor $\rho^{-1}$ into its average and oscillatory parts, and integrating the leading term by parts, we deduce
		\begin{align}\label{est tFu2}
			\int_{\mathbb{T}^3}&\nabla^{s-1}\widetilde{F_{\mathbf{u}}}^{(2)}\cdot\nabla^{s-1}\widetilde{\mathbf{u}}dx\\
			=&\int_{\mathbb{T}^3}\nabla^{s-1}\Big(\overline{\frac{1}{1+a}}(\widetilde{ \mathbf{B}\cdot \nabla \mathbf{B}}-\widetilde{ \mathbf{B}  \nabla \mathbf{B}})+\widetilde{\frac{1}{1+a}}(\mathbf{B}\cdot \nabla \mathbf{B}-\mathbf{B}  \nabla \mathbf{B})\Big) \cdot\nabla^{s-1 }\widetilde{\mathbf{u}}dx\nonumber\\
			=&\int_{\mathbb{T}^3}\nabla^{s-1}\Big(\frac{1}{1+\overline{a}}(\overline{\mathbf{B}}\cdot \nabla \widetilde{ \mathbf{B}}-\overline{\mathbf{B}}  \nabla \widetilde{ \mathbf{B}})\Big) \cdot\nabla^{s-1 }\widetilde{\mathbf{u}}dx
			\nonumber\\
			&-\int_{\mathbb{T}^3}\nabla^{s-2}\Big((\frac{\overline{\widetilde{a}^2}}{(1+\overline{a})^3}+\overline{\mathcal{R}_{\frac{1}{1+a}}(\overline{a},\widetilde{a})\widetilde{a}^3} )(\overline{\mathbf{B}}\cdot \nabla \widetilde{ \mathbf{B}}-\overline{\mathbf{B}}  \nabla \widetilde{ \mathbf{B}})\Big) \cdot\nabla^{s }\widetilde{\mathbf{u}}dx\nonumber\\
			&-\int_{\mathbb{T}^3}\nabla^{s-2}\Big(\overline{\frac{1}{1+a}}(\widetilde{\mathbf{B}}\cdot \nabla \mathbf{B}-\widetilde{\mathbf{B}}  \nabla  \mathbf{B})+\widetilde{\frac{1}{1+a}}(\mathbf{B}\cdot \nabla \mathbf{B}-\mathbf{B}  \nabla \mathbf{B})\Big) \cdot\nabla^{s}\widetilde{\mathbf{u}}dx.\nonumber
		\end{align}
		
		For the first term on the right-hand side, integrating by parts yields
		\begin{align*}
			\int_{\mathbb{T}^3}&\nabla^{s-1}\Big(\frac{1}{1+\overline{a}}( \overline{\mathbf{B}}\cdot \nabla \widetilde{ \mathbf{B}}-\overline{\mathbf{B}}  \nabla \widetilde{ \mathbf{B}})\Big) \cdot\nabla^{s-1 }\widetilde{\mathbf{u}}dx\\
			=& \int_{\mathbb{T}^3}\Big(\frac{1}{1+\overline{a}}(\overline{\mathbf{B}} \cdot\nabla\nabla^{s-1} \widetilde{ \mathbf{B}}-\overline{\mathbf{B}}  \nabla \nabla^{s-1}\widetilde{ \mathbf{B}})\Big) \cdot\nabla^{s-1 }\widetilde{\mathbf{u}}dx\\
			&+\int_{\mathbb{T}^3}\Big[\nabla^{s-1},\frac{1}{1+\overline{a}}(\overline{\mathbf{B}} \nabla - \overline{\mathbf{B}}\cdot \nabla)\Big]\widetilde{\mathbf{B}} \cdot\nabla^{s-1 }\widetilde{\mathbf{u}}dx\\
			=& \int_{\mathbb{T}^3}\Big(\frac{1}{1+\overline{a}}(\overline{\mathbf{B}} \operatorname{div}\nabla^{s-1} \widetilde{ \mathbf{u}}-\overline{\mathbf{B}} \cdot \nabla \nabla^{s-1}\widetilde{ \mathbf{u}})\Big) \cdot\nabla^{s-1 }\widetilde{\mathbf{B}}dx\\
			&-\int_{\mathbb{T}^3}\Big( \nabla^{s-2}\widetilde{\mathbf{B}}\cdot\nabla\big(\overline{\mathbf{B}}\cdot\nabla \frac{1}{1+\overline{a}}\nabla^{s-1 }\widetilde{\mathbf{u}}\big)-\nabla^{s-2}\widetilde{\mathbf{B}}\nabla\big(\nabla (\frac{\overline{\mathbf{B}}}{1+\overline{a}})\cdot\nabla^{s-1 }\widetilde{\mathbf{u}}\big)\Big) dx\\
			&-\sum_{k=2}^{s-1}\binom{s-1}{k}\int_{\mathbb{T}^3}\Big(\nabla^{k}\frac{\overline{\mathbf{B}} }{1+\overline{a}}\nabla\nabla^{s-1-k}\widetilde{\mathbf{B}} - \nabla^{k}\frac{\overline{\mathbf{B}} }{1+\overline{a}}\cdot \nabla \nabla^{s-1-k}\widetilde{\mathbf{B}}\Big) \cdot\nabla^{s-1 }\widetilde{\mathbf{u}}dx \\
			&+(s-1)\int_{\mathbb{T}^3} \Big(\nabla^{2}(\frac{\overline{\mathbf{B}}}{1+\overline{a}})\nabla \nabla^{s-3}\widetilde{\mathbf{B}}-\nabla^{2}(\frac{\overline{\mathbf{B}}}{1+\overline{a}})\cdot\nabla \nabla^{s-3}\widetilde{\mathbf{B}}\Big)\cdot\nabla^{s-1 }\widetilde{\mathbf{u}}dx \\
			&+(s-1)\int_{\mathbb{T}^3} \Big(\nabla(\frac{\overline{\mathbf{B}}}{1+\overline{a}})\nabla \nabla^{s-3}\widetilde{\mathbf{B}}-\nabla(\frac{\overline{\mathbf{B}}}{1+\overline{a}})\cdot\nabla \nabla^{s-3}\widetilde{\mathbf{B}}\Big)\cdot\nabla^{s }\widetilde{\mathbf{u}}dx .
		\end{align*}
		
		Except for the first term on the right side, all remaining terms always can be bounded by $\|(a,\mathbf{B}) \|_{H^{s-1}}\|\widetilde{\mathbf{B}} \|_{H^{s-2}}\|\nabla \mathbf{u} \|_{H^{s-1}}$. For the remainder terms in \eqref{est tFu2}, we similarly have
		\begin{align*}
			-\int_{\mathbb{T}^3}&\nabla^{s-2}\Big((\frac{\overline{\widetilde{a}^2}}{(1+\overline{a})^3}+\overline{\mathcal{R}_{\frac{1}{1+a}}(\overline{a},\widetilde{a})\widetilde{a}^3} )(\overline{\mathbf{B}}\cdot \nabla \widetilde{ \mathbf{B}}-\overline{\mathbf{B}}  \nabla \widetilde{ \mathbf{B}})\Big) \cdot\nabla^{s }\widetilde{\mathbf{u}}dx\\
			&-\int_{\mathbb{T}^3}\nabla^{s-2}\Big(\overline{\frac{1}{1+a}}(\widetilde{\mathbf{B}}\cdot \nabla \mathbf{B}-\widetilde{\mathbf{B}}  \nabla  \mathbf{B})+\widetilde{\frac{1}{1+a}}(\mathbf{B}\cdot \nabla \mathbf{B}-\mathbf{B}  \nabla \mathbf{B})\Big) \cdot\nabla^{s}\widetilde{\mathbf{u}}dx\\
			\lesssim & \|(a,\mathbf{B}) \|_{H^{s-1}}\|(\widetilde{a},\widetilde{\mathbf{B}}) \|_{H^{s-2}}\|\nabla \mathbf{u} \|_{H^{s-1}}.
		\end{align*}
		
		Combining the bad terms \eqref{equ bad b2} in the estimate of $\widetilde{\mathbf{B}}$ and above estimates, it follows that
		\begin{align*}
			\int_{\mathbb{T}^3}&\nabla^{s-1}\widetilde{F_{\mathbf{u}}}^{(2)}\cdot\nabla^{s-1}\widetilde{\mathbf{u}}dx-\int_{\mathbb{T}^3}\frac{1}{1+\overline{a}}(\overline{\mathbf{B}}\nabla^{s-1}\operatorname{div}\widetilde{\mathbf{u}} -\overline{\mathbf{B}}\cdot \nabla \nabla^{s-1}\widetilde{\mathbf{u}})\cdot\nabla^{s-1}\widetilde{\mathbf{B}} dx\\
			\lesssim & \|(a,\mathbf{B}) \|_{H^{s-1}}\|(\widetilde{a},\widetilde{\mathbf{B}}) \|_{H^{s-2}}\|\nabla \mathbf{u} \|_{H^{s-1}}.
		\end{align*}
		
		Integrating by parts and applying H\"older's inequality to the viscous remainder $\widetilde{F_{\mathbf{u}}}^{(3)}$ gives
		\begin{align*}
			\int_{\mathbb{T}^3}\nabla^{s-1}\widetilde{F_{\mathbf{u}}}^{(3)}\cdot\nabla^{s-1}\widetilde{\mathbf{u}}dx=&\int_{\mathbb{T}^3}\nabla^{s-2}\Big(\mu\widetilde{ I(a)\Delta \mathbf{u}}+(\lambda +\mu)\widetilde{I(a)\nabla \operatorname{div}\mathbf{u}}\Big)\cdot\nabla^{s}\widetilde{\mathbf{u}}dx\\
			\lesssim & \|a \|_{H^{s-1}}\|\nabla \mathbf{u} \|^2_{H^{s-1}}.
		\end{align*}
		Collecting all the above estimates and integrating over time, we get
		\begin{align*}
			(\mathcal{T}+t)^{1-\sigma}&\|\widetilde{\mathbf{u}} \|_{\dot{H}^{s-1}}^2+\int_{0}^{t}(\mathcal{T}+\tau)^{1-\sigma} \|\nabla \widetilde{\mathbf{u}} \|^2_{\dot{H}^{s-1}}d\tau\\
			&-\int_0^t\int_{\mathbb{T}^3}(\mathcal{T}+\tau)^{1-\sigma}h(\overline{a}) \nabla^{s-1}\operatorname{div}\widetilde{\mathbf{u}}\cdot \nabla^{s-1} \widetilde{a}dxd\tau\\
			&+\int_0^t\int_{\mathbb{T}^3}(\mathcal{T}+\tau)^{1-\sigma}\frac{1}{1+\overline{a}}\nabla^{s-1}(\partial_{x_3}\mathbf{u}-\mathbf{e}_3\operatorname{div} \widetilde{\mathbf{u}})\cdot\nabla^{s-1}\widetilde{\mathbf{B}}dxd\tau\\
			&+\int_0^t\int_{\mathbb{T}^3}(\mathcal{T}+\tau)^{1-\sigma}\frac{1}{1+\overline{a}}(\overline{\mathbf{B}}  \nabla \nabla^{s-1}\widetilde{ \mathbf{B}}- \overline{\mathbf{B}} \cdot\nabla\nabla^{s-1} \widetilde{ \mathbf{B}}) \cdot\nabla^{s-1 }\widetilde{\mathbf{u}}dxd\tau\\
			\lesssim&\mathcal{T}^{1-\sigma}\|\mathbf{u}_0\|_{H^{s-1}}^2+\sup_{\tau\in[0,t]}\|(a,\mathbf{u},\mathbf{B}) \|_{H^{s-1}}\int_{0}^{t}(\mathcal{T}+\tau)^{1-\sigma}(\|(\widetilde{a},\widetilde{\mathbf{B}}) \|_{H^{s-2}}^2+\|\nabla \mathbf{u} \|^2_{H^{s-1}})d\tau.
		\end{align*}
		
		Summing this inequality with \eqref{est ta hs-1} and \eqref{est tb hs-1},  it follows that
		\begin{align*}
			\widetilde{E}_1(t)=&\sup_{\tau\in[0,t]}(\mathcal{T}+\tau)^{1-\sigma}\|(\widetilde{a},\widetilde{\mathbf{B}},\widetilde{\mathbf{u}}) \|_{H^{s-1}}^2+\int_{0}^{t}(\mathcal{T}+\tau)^{1-\sigma} \|\nabla \widetilde{\mathbf{u}} \|^2_{H^{s-1}}d\tau\\
			\lesssim& \mathcal{T}^{1-\sigma}\|(a_0,\mathbf{u}_0,\mathbf{B}_0)\|_{H^{s-1}}^2+  E_0^{\frac12}(t)N_1^{\frac12}(t)+(E_0(t)+E_1(t)+E_2(t)+E_3(t)+N_1(t))^{\frac32}.
		\end{align*}
	\end{proof}
	\subsection{Estimate of $N_1(t)$}
	This subsection is devoted to extracting the hidden damped wave structure that replaces the full magnetic control supplied by a Diophantine background field. We estimate the oscillatory effective variables $(\widetilde d,\widetilde{\mathbf{G}})$ and then recover the directional dissipation of $\partial_{x_3}\mathbf{B}$ from the projected momentum equation. The argument is carried out in $H^{s-2}$ rather than $H^{s-1}$ because the available control of $\partial_{x_3}\mathbf{B}$ loses one derivative; the $x_3$-Poincar\'e inequality then converts this estimate into the required bound for $\widetilde{\mathbf{B}}$.
	\begin{lem}\label{est N1}
		Under the assumption of Theorem \ref{thm1} and assumption \eqref{ass}, we have
		\begin{align*}
			N_1(t)\lesssim \mathcal{T}^{1-\sigma}\| (a_0,\mathbf{u}_0,\mathbf{B}_0)\|_{H^{s}}^2+ E_{1}(t)+E_{total}^{\frac32}(t).
		\end{align*}
	\end{lem}
	\begin{proof}
		We begin by establishing the estimates for the auxiliary quantities $(\widetilde{d},\widetilde{\mathbf{G}})$, defined by
		\begin{align*}
			d=a+B_3,\qquad \mathbf{G}=\mathbb{Q}\mathbf{u}-\frac{1}{\nu} \Delta^{-1} \nabla d.
		\end{align*} 
		Taking the oscillatory projection of their corresponding equations reveals that $(\widetilde{d},\widetilde{\mathbf{G}})$ satisfy the following system
		\begin{equation}\label{equ dG0}
			\left\{ \begin{aligned}
				&\partial_t\widetilde{d}=\partial_{x_3}u_3-2\operatorname{div}\widetilde{\mathbf{u}}+\widetilde{F_{d}},\\
				&\partial_t\widetilde{\mathbf{G}}-\nu\Delta \widetilde{\mathbf{G}}=\frac{2}{\nu}\mathbb{Q}\widetilde{\mathbf{u}}-\frac{1}{\nu}\Delta^{-1}\nabla \partial_{x_3}u_3+\mathbb{Q}\widetilde{F_{\mathbf{G}}}.
			\end{aligned}\right.
		\end{equation}
		Here, the nonlinear terms are given by
		\begin{align*}
			F_{d}=&F_a+F_{\mathbf{B}}\cdot \mathbf{e}_3=-\mathbf{u}\cdot \nabla a-a\operatorname{div}\mathbf{u}-\mathbf{u}\cdot \nabla B_3+\mathbf{B}\cdot \nabla u_3-B_3\operatorname{div}\mathbf{u},\\
			F_{\mathbf{G}}=&-\mathbf{u}\cdot \nabla \mathbf{u}+\mathbf{B}\cdot \nabla \mathbf{B}-\mathbf{B}\nabla \mathbf{B} +k(a)\nabla a-I(a)( \mu\Delta \mathbf{u}+(\lambda +\mu)  \nabla\operatorname{div}\mathbf{u})\\
			&-I(a)(\partial_{x_3} \mathbf{B}+\mathbf{B}\cdot \nabla \mathbf{B}-\nabla B_3-\mathbf{B}\nabla \mathbf{B})-\frac{1}{\nu}\Delta^{-1}\nabla  F_{d},
		\end{align*}
		with the coefficient $k(a)=1-h(a)=1-\frac{P'(1+a)}{1+a}$.
		
		By the definition of the projection $\mathbb{Q}$ and the quantity $\mathbf{G}$, the velocity divergence can be expressed as
		\begin{align}\label{equ divqu}
			\operatorname{div}\mathbf{u}=\operatorname{div}\mathbb{Q}\mathbf{u}=\operatorname{div}\mathbf{G}+\frac{1}{\nu}d.
		\end{align}
		Substituting \eqref{equ divqu} back into the system \eqref{equ dG0}, we reveal the underlying damped structure for $\widetilde{d}$:
		\begin{equation}\label{equ dG}
			\left\{\begin{aligned}
				&\partial_t\widetilde{d}+\frac{2}{\nu}\widetilde{d}=\partial_{x_3}u_3-2\operatorname{div}\widetilde{\mathbf{G}}+\widetilde{F_{d}},\\
				&\partial_t\widetilde{\mathbf{G}}-\nu\Delta \widetilde{\mathbf{G}}=\frac{2}{\nu}\mathbb{Q}\widetilde{\mathbf{u}}-\frac{1}{\nu}\Delta^{-1}\nabla \partial_{x_3}u_3+\mathbb{Q}\widetilde{F_{\mathbf{G}}}.
			\end{aligned}\right.
		\end{equation}
		
		To estimate $\widetilde{d}$, we take the time-weighted $\dot{H}^{s-2}$ inner product of the first equation in \eqref{equ dG} with $\widetilde{d}$, which yields
		\begin{align*}
			\frac12&\frac{d}{d\tau}\Big((\mathcal{T}+\tau)^{1-\sigma}\|\widetilde{d} \|^2_{\dot{H}^{s-2}}\Big)-\frac{1-\sigma}{2}(\mathcal{T}+\tau)^{-\sigma}\| \widetilde{d}\|^2_{\dot{H}^{s-2}}+\frac{2}{\nu}(\mathcal{T}+\tau)^{1-\sigma}\| \widetilde{d}\|^2_{\dot{H}^{s-2}}\\
			&=\int_{\mathbb{T}^3}(\mathcal{T}+\tau)^{1-\sigma}\nabla^{s-2}(\partial_{x_3}u_3-2\operatorname{div}\widetilde{\mathbf{G}})\nabla^{s-2} \widetilde{d}  dx+\int_{\mathbb{T}^3}(\mathcal{T}+\tau)^{1-\sigma}\nabla^{s-2}\widetilde{F_d}\nabla^{s-2} \widetilde{d}  dx.
		\end{align*}
		
		For the linear coupling term on the right-hand side, an application of the Cauchy-Schwarz and Young inequalities gives
		\begin{align}\label{est hs-1 dlin}
			\int_{\mathbb{T}^3}&(\mathcal{T}+\tau)^{1-\sigma}\nabla^{s-2}(\partial_{x_3}u_3-2\operatorname{div}\widetilde{\mathbf{G}})\nabla^{s-2} \widetilde{d}  dx\\
			&\leq \frac{1}{2\nu}(\mathcal{T}+\tau)^{1-\sigma}\|\widetilde{d} \|^2_{H^{s-2}}+C(\mathcal{T}+\tau)^{1-\sigma}(\|\nabla \widetilde{\mathbf{G}} \|^2_{H^{s-2}}+\|\nabla \widetilde{\mathbf{u}} \|^2_{H^{s-2}}).\nonumber
		\end{align}
		
		Similarly, the nonlinear forcing term is bounded by
		\begin{align*}
			\int_{\mathbb{T}^3}(\mathcal{T}+\tau)^{1-\sigma}\nabla^{s-2}\widetilde{F_d}\nabla^{s-2} \widetilde{d}  dx\leq  \frac{1}{2\nu}(\mathcal{T}+\tau)^{1-\sigma}\|\widetilde{d} \|^2_{H^{s-2}}+C(\mathcal{T}+\tau)^{1-\sigma}\| \widetilde{F_{d}}\|^2_{H^{s-2}},
		\end{align*}
		where the $H^{s-2}$ norm of $\widetilde{F_d}$ satisfies the standard product estimates:
		\begin{align}\label{est ftd hs-2}
			\|\widetilde{F_d} \|_{H^{s-2}}\lesssim & \|\widetilde{\mathbf{u}\cdot \nabla a} \|_{H^{s-2}}+\|\widetilde{a\operatorname{div}\mathbf{u}} \|_{H^{s-2}}+\|\widetilde{\mathbf{u}\cdot \nabla B_3} \|_{H^{s-2}}\\
			&+\|\widetilde{\mathbf{B}\cdot \nabla u_3} \|_{H^{s-2}}+\|\widetilde{B_3\operatorname{div}\mathbf{u}} \|_{H^{s-2}}\nonumber\\
			\lesssim & \|a \|_{H^{s-1}}\| \mathbf{u} \|_{H^{s-2}}+\| a\|_{H^{s-2}}\| \nabla \mathbf{u}\|_{H^{s-2}}\nonumber\\
			&+\| \mathbf{B}\|_{H^{s-1}}\|\mathbf{u} \|_{H^{s-2}}+\|\mathbf{B} \|_{H^{s-2}}\| \nabla \mathbf{u}\|_{H^{s-2}}\nonumber\\
			\lesssim & \|(a,\mathbf{B}) \|_{H^{s-1}}\| \nabla\mathbf{u} \|_{H^{s-2}}.\nonumber
		\end{align}
		Integrating over time, we obtain
		\begin{align}\label{est hs-1 dnonl}
			\int_{0}^{t}&\int_{\mathbb{T}^3}(\mathcal{T}+\tau)^{1-\sigma}\nabla^{s-2}\widetilde{F_d}\nabla^{s-2} \widetilde{d}  dxd\tau  \\
			&\leq\frac{1}{2\nu}\int_{0}^{t}(\mathcal{T}+\tau)^{1-\sigma}\|\widetilde{d} \|^2_{H^{s-2}}d\tau +C\sup_{\tau\in[0,t]}  \|(a,\mathbf{B}) \|_{H^{s-1}}^2\int_{0}^{t}(\mathcal{T}+\tau)^{1-\sigma} \|\nabla \mathbf{u} \|_{H^{s-2}}^2d\tau.\nonumber
		\end{align}
		
		Combining \eqref{est hs-1 dlin}, \eqref{est hs-1 dnonl} and recalling the definitions of $E_0,E_i$, we deduce that
		\begin{align}\label{est td hs-2}
			(\mathcal{T}+t)^{1-\sigma}&\|\widetilde{d} \|^2_{H^{s-2}}+\frac{1}{\nu}\int_{0}^t(\mathcal{T}+\tau)^{1-\sigma}\| \widetilde{d}\|^2_{H^{s-2}}d\tau
			\\
			\lesssim& \mathcal{T}^{1-\sigma}\| (a_0,\mathbf{B}_0)\|_{H^{s-2}}^2+\int_{0}^t (\mathcal{T}+\tau)^{1-\sigma}(\|\nabla \widetilde{\mathbf{G}} \|^2_{H^{s-2}}+\|\nabla\widetilde{\mathbf{u}} \|^2_{H^{s-2}})d\tau+E_1^2(t).\nonumber
		\end{align}
		
		To derive the energy bounds for the quantity $\widetilde{\mathbf{G}}$, we take the time-weighted $\dot{H}^{s-2}$ inner product of the second equation in \eqref{equ dG} against $\widetilde{\mathbf{G}}$. This yields
		\begin{align}\label{est dtG hs-2}
			\frac12\frac{d}{d\tau}&\Big((\mathcal{T}+\tau)^{1-\sigma}\|\widetilde{\mathbf{G}} \|^2_{\dot{H}^{s-2}}\Big)-\frac{1-\sigma}{2}(\mathcal{T}+\tau)^{-\sigma}\| \widetilde{\mathbf{G}}\|^2_{\dot{H}^{s-2}}+\nu(\mathcal{T}+\tau)^{1-\sigma}\| \nabla\widetilde{\mathbf{G}}\|^2_{\dot{H}^{s-2}}\\
			=& \int_{\mathbb{T}^3}(\mathcal{T}+\tau)^{1-\sigma}\nabla^{s-2}(\frac{2}{\nu}\mathbb{Q}\widetilde{\mathbf{u}}-\frac{1}{\nu}\Delta^{-1}\nabla\partial_{x_3}u_3+\mathbb{Q}\widetilde{F_\mathbf{G}})\cdot\nabla^{s-2} \widetilde{\mathbf{G}} dx\nonumber\\
			\leq & \frac{\nu}{2}(\mathcal{T}+\tau)^{1-\sigma}\|\nabla \widetilde{\mathbf{G}} \|_{H^{s-2}}^2+C(\mathcal{T}+\tau)^{1-\sigma}(\|\widetilde{\mathbf{u}} \|_{H^{s-2}}^2+\|\widetilde{F_\mathbf{G}} \|_{H^{s-3}}^2).\nonumber
		\end{align}
		
		Invoking the estimate \eqref{est ftd hs-2} and standard product estimates, the nonlinear term $\widetilde{F_{\mathbf{G}}}$ can be bounded as follows:
		\begin{align*}
			\|\widetilde{F_{\mathbf{G}}} \|_{H^{s-3}}\lesssim & \|\widetilde{\mathbf{u}\cdot \nabla \mathbf{u}}\|_{H^{s-3}}+\|\widetilde{\mathbf{B}\cdot \nabla \mathbf{B}} \|_{H^{s-3}}+\|\widetilde{ \mathbf{B}\nabla \mathbf{B}} \|_{H^{s-3}}+\|\widetilde{k(a)\nabla a} \|_{H^{s-3}}+\|\widetilde{  I(a)\Delta \mathbf{u}}\|_{H^{s-3}}\\
			&+\|\widetilde{ I(a)\nabla\operatorname{div}\mathbf{u}} \|_{H^{s-3}}+\| \widetilde{I(a)\partial_{x_3} \mathbf{B} }\|_{H^{s-3}}+\| \widetilde{I(a)\nabla B_3 }\|_{H^{s-3}}\\
			&+\|\widetilde{I(a)\mathbf{B}\cdot \nabla \mathbf{B} } \|_{H^{s-3}}+\|\widetilde{I(a)\mathbf{B}\nabla \mathbf{B}}\|_{H^{s-3}}+\|\frac{1}{\nu}\Delta^{-1}\nabla  \widetilde{F_{d} } \|_{H^{s-3}}\\
			\lesssim & \| \widetilde{\mathbf{u}}\|_{H^{s-2}}\|\mathbf{u} \|_{H^{s-2}}+\|\widetilde{\mathbf{B}} \|_{H^{s-2}}\| \mathbf{B}\|_{H^{s-2}}\\
			&+\|\widetilde{a} \|_{H^{s-2}}\|a \|_{H^{s-2}}+\| \widetilde{a}\|_{H^{s-3}}\|\nabla\mathbf{u} \|_{H^{s-2}}+\| a\|_{H^{s-3}}\|\nabla\widetilde{\mathbf{u}}  \|_{H^{s-2}}\\
			&+\|a \|_{H^{s-3}}\|\partial_{x_3}\mathbf{B} \|_{H^{s-3}}
			+\|\widetilde{a} \|_{H^{s-3}}\|B_3 \|_{H^{s-2}}+\| \overline{a}\|_{H^{s-3}}\|\widetilde{\mathbf{B}}_3 \|_{H^{s-2}}\\
			&+\|\widetilde{a} \|_{H^{s-3}}\| \mathbf{B}\|^2_{H^{s-2}}+\|a \|_{H^{s-3}}\|\mathbf{B} \|_{H^{s-2}}\|\widetilde{\mathbf{B}} \|_{H^{s-2}}+\|(a,\mathbf{B}) \|_{H^{s-2}}\| \nabla\mathbf{u} \|_{H^{s-3}}\\
			\lesssim & (\|(a,\mathbf{u},\mathbf{B}) \|_{H^{s-2}}+\|(\mathbf{u},\mathbf{B}) \|^2_{H^{s-2}})\|(\widetilde{a},\nabla\mathbf{u},\widetilde{\mathbf{B}}) \|_{H^{s-2}}.
		\end{align*}
		Inserting this bound back into \eqref{est dtG hs-2} and integrating over time, we deduce
		\begin{align}\label{est tG hs-2}
			(\mathcal{T}+t)^{1-\sigma}&\|\widetilde{\mathbf{G}} \|^2_{H^{s-2}}+\int_{0}^t(\mathcal{T}+\tau)^{1-\sigma}\| \nabla\widetilde{\mathbf{G}}\|^2_{H^{s-2}}d\tau\\
			\lesssim & \mathcal{T}^{1-\sigma}\| (a_0,\mathbf{u}_0,\mathbf{B}_0)\|_{H^{s-2}}^2+E_1(t) +  (N_1(t)+E_1(t))^{2}.\nonumber
		\end{align}
		
		Multiplying \eqref{est td hs-2} with a suitable small constant and adding to \eqref{est tG hs-2}, we conclude that
		\begin{align}\label{est tdG hs-1}
			(\mathcal{T}+t)^{1-\sigma}&\|(\widetilde{d},\widetilde{\mathbf{G}}) \|^2_{H^{s-2}}+\int_{0}^t(\mathcal{T}+\tau)^{1-\sigma}\| (\widetilde{d},\nabla\widetilde{\mathbf{G}})\|^2_{H^{s-2}}d\tau
			\\
			\lesssim& \mathcal{T}^{1-\sigma}\| (a_0,\mathbf{u}_0,\mathbf{B}_0)\|_{H^{s}}^2+E_1(t)+   (N_1(t)+E_1(t))^{2}.\nonumber
		\end{align}
		
		Finally, we derive the estimate for $\partial_{x_3}\mathbf{B}$. Applying $\nabla^{s-2}$ to the equation for $\mathbb{P}\mathbf{u}$ and taking inner product with $\partial_{x_3}\nabla^{s-2}\mathbf{B}$,  we obtain
		\begin{align*}
			(\mathcal{T}+\tau)^{1-\sigma}&\|\partial_{x_3}\mathbf{B} \|^2_{\dot{H}^{s-2}}\\
			=& \int_{\mathbb{T}^3}(\mathcal{T}+\tau)^{1-\sigma}\nabla^{s-2}\mathbb{P}(\rho\partial_{\tau}\mathbf{u}+\rho \mathbf{u}\cdot \nabla \mathbf{u}-\mu \Delta \mathbf{u}-\mathbf{B}\cdot \nabla \mathbf{B} )\cdot\nabla^{s-2}\partial_{x_3}\mathbf{B}   dx\\
			&=\frac{d}{d\tau}\int_{\mathbb{T}^3}(\mathcal{T}+\tau)^{1-\sigma}\nabla^{s-2}\mathbb{P}(\rho\mathbf{u})\cdot\nabla^{s-2}\partial_{x_3}\mathbf{B}   dx+\sum_{i=1}^6\mathcal{N}_i,
		\end{align*}
		where
		\begin{align*}
			\mathcal{N}_1=&-(1-\sigma)\int_{\mathbb{T}^3}(\mathcal{T}+\tau)^{-\sigma}\nabla^{s-2}\mathbb{P}(\rho\mathbf{u})\cdot\nabla^{s-2}\partial_{x_3}\mathbf{B}   dx,\\
			\mathcal{N}_2=& -\int_{\mathbb{T}^3}(\mathcal{T}+\tau)^{1-\sigma}\nabla^{s-2}\mathbb{P}(\partial_t\rho\mathbf{u})\cdot\nabla^{s-2}\partial_{x_3}\mathbf{B}   dx, \\
			\mathcal{N}_3=&\int_{\mathbb{T}^3}(\mathcal{T}+\tau)^{1-\sigma}\nabla^{s-2}\mathbb{P}\partial_{x_3}(\rho\mathbf{u})\cdot\nabla^{s-2}\partial_t\mathbf{B}   dx, \\
			\mathcal{N}_4=&\int_{\mathbb{T}^3}(\mathcal{T}+\tau)^{1-\sigma}\nabla^{s-2}\mathbb{P}( \rho \mathbf{u}\cdot \nabla \mathbf{u})\cdot\nabla^{s-2}\partial_{x_3}\mathbf{B}   dx, \\
			\mathcal{N}_5=&-\int_{\mathbb{T}^3}(\mathcal{T}+\tau)^{1-\sigma}\nabla^{s-2}\mathbb{P} \mu\Delta \mathbf{u}\cdot\nabla^{s-2}\partial_{x_3}\mathbf{B}   dx, \\
			\mathcal{N}_6=&-\int_{\mathbb{T}^3}(\mathcal{T}+\tau)^{1-\sigma}\nabla^{s-2}\mathbb{P}( \mathbf{B}\cdot \nabla \mathbf{B} )\cdot\nabla^{s-2}\partial_{x_3}\mathbf{B}   dx.
		\end{align*}
		For the first term, we have
		\begin{align*}
			\mathcal{N}_1=&-(1-\sigma) \int_{\mathbb{T}^3}(\mathcal{T}+\tau)^{-\sigma}\nabla^{s-2}\mathbb{P}(\rho\mathbf{u})\cdot\nabla^{s-2}\partial_{x_3}\mathbf{B}   dx\\
			\lesssim & \mathcal{T}^{-1} (\mathcal{T}+\tau)^{1-\sigma}(1+\| a\|_{H^{s-2}})\|\nabla \mathbf{u} \|_{H^{s-3}}\|\partial_{x_3}\mathbf{B} \|_{H^{s-2}}.
		\end{align*}
		Then, via Young's inequality, the magnetic term $\|\partial_{x_3}\mathbf{B}\|_{H^{s-2}}$ can be absorbed into the left-hand side by choosing a suitably small coefficient, leaving the velocity term to be bounded by the energy functional $E_1(t)$.
		
		Next, using the equation for $a$, we can bound $\mathcal{N}_2$ as follows:
		\begin{align*}
			\mathcal{N}_2=&-\int_{\mathbb{T}^3}(\mathcal{T}+\tau)^{1-\sigma}\nabla^{s-2}\mathbb{P}(\partial_{\tau}\rho\mathbf{u})\cdot\nabla^{s-2}\partial_{x_3}\mathbf{B}   dx\\
			=& \int_{\mathbb{T}^3}(\mathcal{T}+\tau)^{1-\sigma}\nabla^{s-2}\mathbb{P}(( \operatorname{div} \mathbf{u}+\mathbf{u}\cdot \nabla a+a\operatorname{div}\mathbf{u}) \mathbf{u})\cdot\nabla^{s-2}\partial_{x_3}\mathbf{B}   dx\\
			\lesssim & (\mathcal{T}+\tau)^{1-\sigma}(\|\nabla \mathbf{u} \|_{H^{s-2}}+\|\mathbf{u} \|_{H^{s-2}}\| \nabla a\|_{H^{s-2}}+\| a\|_{H^{s-2}}\|\operatorname{div}\mathbf{u} \|_{H^{s-2}})\| \mathbf{u}\|_{H^{s-2}}\|\partial_{x_3}\mathbf{B} \|_{H^{s-2}}\\
			\lesssim & (\mathcal{T}+\tau)^{1-\sigma}(1+\|(a,\mathbf{u}) \|_{H^{s-1}})\|\nabla \mathbf{u} \|_{H^{s-2}}^2\|\partial_{x_3}\mathbf{B} \|_{H^{s-2}}.
		\end{align*}
		
		Turning to $\mathcal{N}_3$, integrating by parts with respect to $x_3$ and substituting the equation for $\partial_\tau \mathbf{B}$, we obtain
		\begin{align*}
			\mathcal{N}_3=&\int_{\mathbb{T}^3}(\mathcal{T}+\tau)^{1-\sigma}\nabla^{s-2}\mathbb{P}\partial_{x_3}(\rho\mathbf{u})\cdot\nabla^{s-2}\partial_t\mathbf{B}   dx\\
			=&\int_{\mathbb{T}^3} (\mathcal{T}+\tau)^{1-\sigma}\nabla^{s-2}\mathbb{P}\partial_{x_3}(\rho\mathbf{u})\cdot\nabla^{s-2} (\mathbf{B}\cdot \nabla \mathbf{u}-\mathbf{u}\cdot \nabla \mathbf{B}-\mathbf{B}\operatorname{div}\mathbf{u}-\mathbf{e}_3\operatorname{div}\mathbf{u}+\partial_{x_3}\mathbf{u} )   dx\\
			\lesssim & (\mathcal{T}+\tau)^{1-\sigma}(1+\|a \|_{H^{s-1}})\|\nabla \mathbf{u} \|_{H^{s-2}}(\|\mathbf{B} \|_{H^{s-2}}\| \nabla \mathbf{u}\|_{H^{s-2}}+\| \mathbf{u}\|_{H^{s-2}}\| \mathbf{B}\|_{H^{s-1}}\\
			&+\| \mathbf{B}\|_{H^{s-2}}\|\operatorname{div}\mathbf{u} \|_{H^{s-2}}+\|\operatorname{div}\mathbf{u} \|_{H^{s-2}}+\|\nabla \mathbf{u} \|_{H^{s-2}} )\\
			\lesssim &  (\mathcal{T}+\tau)^{1-\sigma}(1+\|(a,\mathbf{B}) \|_{H^{s-1}} )\|\nabla \mathbf{u} \|_{H^{s-2}}^2.
		\end{align*}
		
		The convection term $\mathcal{N}_4$ can be estimated directly:
		\begin{align*}
			\mathcal{N}_4=&\int_{\mathbb{T}^3}(\mathcal{T}+\tau)^{1-\sigma}\nabla^{s-2}\mathbb{P}( \rho \mathbf{u}\cdot \nabla \mathbf{u})\cdot\nabla^{s-2}\partial_{x_3}\mathbf{B}   dx\\
			\lesssim  &(\mathcal{T}+\tau)^{1-\sigma}(1+\|a \|_{H^{s-2}})\| \mathbf{u}\|_{H^{s-2}}\|\nabla \mathbf{u} \|_{H^{s-2}}\|\partial_{x_3}\mathbf{B} \|_{H^{s-2}}.
		\end{align*}
		
		As for the viscous term, H\"older inequality yields
		\begin{align*}
			\mathcal{N}_5=&-\int_{\mathbb{T}^3}(\mathcal{T}+\tau)^{1-\sigma}\nabla^{s-2}\mathbb{P} \mu\Delta \mathbf{u}\cdot\nabla^{s-2}\partial_{x_3}\mathbf{B}   dx\\
			\lesssim&  (\mathcal{T}+\tau)^{1-\sigma}\|\nabla \mathbf{u} \|_{H^{s-1}}\|\partial_{x_3}\mathbf{B} \|_{H^{s-2}}.
		\end{align*}
		This term can be handled in the same manner as $\mathcal{N}_1$.
		
		For the last term, decomposing the magnetic field into horizontal and vertical components, we deduce the bound for the Lorentz force term:
		\begin{align*}
			\mathcal{N}_6=&-\int_{\mathbb{T}^3}(\mathcal{T}+\tau)^{1-\sigma}\nabla^{s-2}\mathbb{P}( \mathbf{B}\cdot \nabla \mathbf{B} )\cdot\nabla^{s-2}\partial_{x_3}\mathbf{B}   dx \\
			=&-\int_{\mathbb{T}^3}(\mathcal{T}+\tau)^{1-\sigma}\nabla^{s-2}\mathbb{P}( \mathbf{B}_h\cdot \nabla_h \mathbf{B}+B_3\partial_{x_3}\mathbf{B} )\cdot\nabla^{s-2}\partial_{x_3}\mathbf{B}   dx \\
			\lesssim& (\mathcal{T}+\tau)^{1-\sigma}(\|\mathbf{B}_h \|_{H^{s-2}}\| \mathbf{B}\|_{H^{s-1}}+\|B_3 \|_{H^{s-2}}\| \partial_{x_3}\mathbf{B}\|_{H^{s-2}})\|\partial_{x_3}\mathbf{B} \|_{H^{s-2}}\\
			\lesssim & (\mathcal{T}+\tau)^{1-\sigma}\| \mathbf{B}\|_{H^{s-1}}\|\partial_{x_3}\mathbf{B} \|_{H^{s-2}}^2.
		\end{align*}
		Under the bootstrap smallness assumption, the quadratic remainders in \eqref{est tdG hs-1} are bounded by the corresponding $3/2$-powers. Invoking the above estimates and integrating over time, we obtain
		\begin{align*}
			\int_{0}^t(\mathcal{T}+\tau)^{1-\sigma}&\|\nabla^{s-2}\partial_{x_3}\mathbf{B} \|^2_{L^{2}}d\tau	\lesssim \mathcal{T}^{1-\sigma}\| (a_0,\mathbf{u}_0,\mathbf{B}_0)\|_{H^{s}}^2+E_1(t)+  (E_1(t)+N_1(t) )^{\frac32},
		\end{align*}
		where we use the fact 
		\begin{align*}
			\int_{0}^t\frac{d}{d\tau}&\int_{\mathbb{T}^3}(\mathcal{T}+\tau)^{1-\sigma}\nabla^{s-2}\mathbb{P}(\rho\mathbf{u})\cdot\nabla^{s-2}\partial_{x_3}\mathbf{B}   dxd\tau\\
			\leq &2\mathcal{T}^{1-\sigma}\| (\mathbf{u}_0,\mathbf{B}_0)\|_{H^{s-1}}^2+\frac{1}{2}\int_{0}^t(\mathcal{T}+\tau)^{1-\sigma}\|\nabla^{s-2}\partial_{x_3}\mathbf{B} \|^2_{L^{2}}d\tau+E_1(t).
		\end{align*}
		Combining with the  estimate \eqref{est tdG hs-1}, it follows that
		\begin{align*}
			N_1(t)=&\sup_{\tau\in[0,t]}(\mathcal{T}+\tau)^{1-\sigma} \| (\widetilde{d},\widetilde{\mathbf{G}})\|^2_{H^{s-2}}+\int_{0}^{t}(\mathcal{T}+\tau)^{1-\sigma}\| (\nabla \widetilde{\mathbf{G}},\widetilde{d},\partial_{x_3}\mathbf{B})(\tau,\cdot)\|^2_{H^{s-2}}d\tau\\
			\lesssim&  \mathcal{T}^{1-\sigma}\| (a_0,\mathbf{u}_0,\mathbf{B}_0)\|_{H^{s}}^2+ E_1(t)+(E_1(t)+N_1 (t))^{\frac32}.
		\end{align*}
	\end{proof}
	\subsection{Estimate of $E_3(t)$ and $N_3(t)$}
	This subsection establishes the analogous average, oscillatory, and hidden-damping estimates at a lower regularity level, but equipped with the stronger time decay weight $(\mathcal T+t)^{3-\sigma}$. Consequently, we skip the explicit estimates for $E_2(t)$ and $N_2(t)$, as they can be bounded via interpolation between $E_1(t), E_3(t)$ and $N_1(t), N_3(t)$, respectively. Although the increased weight generates additional terms, these are controlled by interpolating between the top-order energy $E_0$ and the low-order energy $E_3(t),\ N_3(t)$.
	
	\begin{lem}\label{lem E3}
		Under the assumptions of Theorem \ref{thm1} and assumption \eqref{ass}, there holds
		\begin{align*}
			E_3(t)\lesssim   & \mathcal{T}^{3-\sigma}\|(a_0,\mathbf{u}_0,\mathbf{B}_0) \|_{H^{s}}^2+\mathcal{T}^{\frac{2\sigma-1}{2}}E_0^{\frac12}(t)E_2^{\frac12}(t)+E_0^{\frac14}(t)N_3^{\frac34}(t)
			+E_{total}^{\frac32}(t),\\
			N_3(t)
			\lesssim&  \mathcal{T}^{3-\sigma}\| (a_0,\mathbf{u}_0,\mathbf{B}_0)\|_{H^{s}}^2+ E_3(t)+E_{total}^{\frac32}(t).
		\end{align*}
	\end{lem}
	\begin{proof}
		The estimate for  $\| (\overline{a},\overline{\mathbf{B}})\|_{H^{s-3}}$ follows by an argument analogous to that used for $\overline{E}_1$. Taking the $H^{s-3}$ inner products of \eqref{equ ba} and \eqref{equ bB} with $\overline{a}$ and $\overline{\mathbf B}$, respectively, and integrating in time, we obtain
		\begin{align*}
			\frac12  \| \overline{a} (t) \|^2_{\dot{H}^{s-3}}=&\frac12\|a_0\|_{H^{s-3}}^2- \int_0^t\int_{\mathbb{T}^3}\nabla^{s-3}(\overline{\mathbf{u}\cdot \nabla a})\cdot\nabla^{s-3}\overline{a} dxd \tau\\
			&-\int_0^t\int_{\mathbb{T}^3}\nabla^{s-3}(\overline{(1+a) \operatorname{div}\mathbf{u}})\cdot\nabla^{s-3}\overline{a} dxd \tau\\
			\lesssim &\|a_0\|_{H^{s-3}}^2+\sup_{\tau\in[0,t]}\|a \|_{H^{s-2}}\int_{0}^t(\mathcal{T}+\tau)^{\frac{2-\sigma}{2}}\|\nabla \mathbf{u} \|_{H^{s-2}}(\mathcal{T}+\tau)^{\frac{-1-\sigma}{2}}\|a \|_{H^{s}}  d\tau\\
			&+\mathcal{T}^{\frac{2\sigma-1}{2}}\sup_{\tau\in[0,t]}(1+\|a \|_{H^{s-3}})\int_{0}^t (\mathcal{T}+\tau)^{\frac{2-\sigma}{2}}\|\operatorname{div}\mathbf{u} \|_{H^{s-2}}(\mathcal{T}+\tau)^{\frac{-1-\sigma}{2}}\|a \|_{H^{s}} d \tau\\
			\lesssim &\|a_0\|_{H^{s-3}}^2+\mathcal{T}^{\frac{2\sigma-1}{2}}E_0^{\frac12}(t)E_2^{\frac12}(t)+E_0^{\frac12}(t)E_1^{\frac12}(t)E_2^{\frac12}(t),
		\end{align*}
		and
		\begin{align*}
			\frac{1}{2}  \| \overline{\mathbf{B}}(t) \|^2_{\dot{H}^{s-3}}=&\frac12\|\mathbf{B}_0\|_{H^{s-3}}^2- \int_{0}^t\int_{\mathbb{T}^3} \nabla^{s-3}(\overline{\operatorname{div} \mathbf{u}})\mathbf{e}_3\cdot\nabla^{s-3}\overline{\mathbf{B}} dxd\tau\\
			&+ \int_{0}^t\int_{\mathbb{T}^3} \nabla^{s-3}(-\overline{\mathbf{u}\cdot \nabla \mathbf{B}}+\overline{\mathbf{B}\cdot \nabla \mathbf{u}}-\overline{\mathbf{B}\operatorname{div}\mathbf{u}} )\cdot\nabla^{s-3}\overline{\mathbf{B}} dxd\tau\\
			\lesssim & \|\mathbf{B}_0\|_{H^{s-3}}^2+ \mathcal{T}^{\frac{2\sigma-1}{2}}\int_{0}^t(\mathcal{T}+\tau)^{\frac{2-\sigma}{2}}\| \operatorname{div}\mathbf{u}\|_{H^{s-2}}(\mathcal{T}+\tau)^{\frac{-1-\sigma}{2}}\|\mathbf{B} \|_{H^{s}} d\tau\\
			&+  \sup_{\tau\in[0,t]}\|\mathbf{B} \|_{H^{s-1}}\int_{0}^t(\mathcal{T}+\tau)^{\frac{2-\sigma}{2}}\|\nabla \mathbf{u} \|_{H^{s-2}}(\mathcal{T}+\tau)^{\frac{-1-\sigma}{2}}\|\mathbf{B}\|_{H^{s}}  d\tau\\
			\lesssim & \|\mathbf{B}_0\|_{H^{s-3}}^2+ \mathcal{T}^{\frac{2\sigma-1}{2}}E_0^{\frac12}(t)E_2^{\frac12}(t)+E_0^{\frac12}(t)E_1^{\frac12}(t)E_2^{\frac12}(t).
		\end{align*}
		
		To estimate $\|\overline{\mathbf{u}} \|^2_{\dot{H}^{s-3}}+\frac{1}{2}\|\overline{\mathcal{D}} \|^2_{\dot{H}^{s-3}}$, we require a proper distribution of the time weight $(\mathcal{T}+\tau)^{3-\sigma}$:
		\begin{align}\label{est dtbubD hs-3}
			\frac12\frac{d}{d\tau}& \Big((\mathcal{T}+\tau)^{3-\sigma}(\|\overline{\mathbf{u}} \|^2_{\dot{H}^{s-3}} +\frac12\|\overline{\mathcal{D}} \|^2_{\dot{H}^{s-3}}) \Big)-\frac{3-\sigma}{2}(\mathcal{T}+\tau)^{2-\sigma}(\|\overline{\mathbf{u}} \|^2_{\dot{H}^{s-3}}+\frac12\|\overline{\mathcal{D}} \|^2_{\dot{H}^{s-3}})\\
			&+
			\mu(\mathcal{T}+\tau)^{3-\sigma}\|\nabla \overline{\mathbf{u}} \|^2_{\dot{H}^{s-3}}+
			(\lambda+\mu)(\mathcal{T}+\tau)^{3-\sigma}\| \operatorname{div}  \overline{\mathbf{u}} \|^2_{\dot{H}^{s-3}} \nonumber
			\\
			=&\int_{\mathbb{T}^3}(\mathcal{T}+\tau)^{3-\sigma}\big(\nabla^{s-3}\nabla\overline{\mathcal{D}}\cdot\nabla^{s-3}\overline{\mathbf{u}}+ \nabla^{s-3}\overline{\operatorname{div}\mathbf{u}}\cdot\nabla^{s-3}\overline{\mathcal{D}}\big) dx\nonumber\\
			&+\int_{\mathbb{T}^3}(\mathcal{T}+\tau)^{3-\sigma}(\nabla^{s-3}\overline{F_{\overline{\mathbf{u}}}}\cdot\nabla^{s-3}\overline{\mathbf{u}}+\frac12 \nabla^{s-3}\overline{F_{\mathcal{D}}}\cdot\nabla^{s-3}\overline{\mathcal{D}} )dx,\nonumber
		\end{align}
		where we have again utilized the cancellation of the linear coupling terms:
		\begin{align*}
			\int_{\mathbb{T}^3}(\mathcal{T}+\tau)^{3-\sigma}\big(\nabla^{s-3}\nabla\overline{\mathcal{D}}\cdot\nabla^{s-3}\overline{\mathbf{u}}+ \nabla^{s-3}\overline{\operatorname{div}\mathbf{u}}\cdot\nabla^{s-3}\overline{\mathcal{D}}\big) dx=0.
		\end{align*}
		
		For the remaining nonlinear terms on the right-hand side of \eqref{est dtbubD hs-3}, integration by parts yields
		\begin{align*}
			\int_{\mathbb{T}^3}&(\mathcal{T}+\tau)^{3-\sigma}(\nabla^{s-3}\overline{F_{\overline{\mathbf{u}}}}\cdot\nabla^{s-3}\overline{\mathbf{u}}+\frac12 \nabla^{s-3}\overline{F_{\mathcal{D}}}\cdot\nabla^{s-3}\overline{\mathcal{D}}) dx\\
			=&-\int_{\mathbb{T}^3}(\mathcal{T}+\tau)^{3-\sigma}(\nabla^{s-4}\overline{F_{\overline{\mathbf{u}}}}\cdot\nabla^{s-2}\overline{\mathbf{u}}-\frac12 \nabla^{s-3}\overline{F_{\mathcal{D}}}\cdot\nabla^{s-3}\overline{\mathcal{D}})dx\\
			\leq & (\mathcal{T}+\tau)^{3-\sigma}(\|\overline{F_{\overline{\mathbf{u}}}} \|_{\dot{H}^{s-4}}\|\nabla \overline{\mathbf{u}} \|_{\dot{H}^{s-3}}+\|\overline{F_{\mathcal{D}}} \|_{\dot{H}^{s-3}}\| \overline{\mathcal{D}} \|_{\dot{H}^{s-3}}).
		\end{align*}
		
		Applying the estimates \eqref{est Fbu hs-2} and \eqref{est FbD hs-1},  we have
		\begin{align}\label{est Fbu hs-4}
			\| \overline{F_{\overline{\mathbf{u}}}}\|_{\dot{H}^{s-4}}\lesssim&(1+\| (a,\mathbf{B} ) \|_{H^{s-1}})\|(a,\mathbf{u},  \mathbf{B} ) \|_{H^{s-1}}(\|\overline{\mathcal{D}} \|_{\dot{H}^{s-3}}+\|\nabla\mathbf{u}\|_{H^{s-3}}+\|(\widetilde{a},\partial_{x_3}\widetilde{\mathbf{B}}) \|_{H^{s-4}}),
		\end{align}
		and
		\begin{align}\label{est FbD hs-3}
			\|\overline{F_{\mathcal{D}} } \|_{\dot{H}^{s-3}}\lesssim (1+\|(a,\mathbf{B}) \|_{H^{s-1}})(\|(a,\mathbf{B}) \|_{H^{s-3}}\|\nabla \mathbf{u} \|_{H^{s-3}}+\| (a,\mathbf{B})\|_{H^{s-2}}\|\nabla \mathbf{u}\|_{H^{s-4}} ) .
		\end{align}
		
		Multiplying these bounds by time weights and integrating over time yields
		\begin{align*}
			\int_{0}^t&(\mathcal{T}+\tau)^{3-\sigma}\|\overline{F_{\overline{\mathbf{u}}}} \|_{\dot{H}^{s-4}}\|\nabla \overline{\mathbf{u}} \|_{\dot{H}^{s-3}}d\tau\\
			\lesssim & \sup_{\tau\in[0,t]}(1+\| (a,\mathbf{B} ) \|_{H^{s-1}})\|(a,\mathbf{u},  \mathbf{B} ) \|_{H^{s-1}}\\
			&\qquad\times\int_{0}^t(\mathcal{T}+\tau)^{3-\sigma}(\|\overline{\mathcal{D}} \|_{\dot{H}^{s-3}}+\|\nabla\mathbf{u}\|_{H^{s-3}}+\|(\widetilde{a},\partial_{x_3}\widetilde{\mathbf{B}}) \|_{H^{s-4}})\|\nabla \overline{\mathbf{u}} \|_{\dot{H}^{s-3}}d\tau\\
			\lesssim & E_1^{\frac12}(t)E_3^{\frac12}(t)(E_3^{\frac12}(t)+N_3^{\frac12}(t)),
		\end{align*}
		and
		\begin{align*}
			\int_{0}^t&(\mathcal{T}+\tau)^{3-\sigma} \|\overline{F_{\mathcal{D}}} \|_{\dot{H}^{s-3}}\| \overline{\mathcal{D}} \|_{\dot{H}^{s-3}} d\tau\\
			\lesssim& \sup_{\tau\in[0,t]}(1+\|(a,\mathbf{B}) \|_{H^{s-1}})\|(a,\mathbf{B}) \|_{H^{s-1}}\int_{0}^t(\mathcal{T}+\tau)^{3-\sigma} \|\nabla \mathbf{u} \|_{H^{s-3}}  \| \overline{\mathcal{D}} \|_{\dot{H}^{s-3}} d\tau\\
			\lesssim & E_1^{\frac32}(t)+E_3^{\frac32}(t).
		\end{align*}
		
		Consequently, inserting these bounds back into the \eqref{est dtbubD hs-3} provides the coupled energy estimate
		\begin{align}\label{est buD hs-3}
			(\mathcal{T}+t)^{3-\sigma}&(\|\overline{\mathbf{u}} (t)\|^2_{\dot{H}^{s-3}} + \|\overline{\mathcal{D}} (t)\|^2_{\dot{H}^{s-3}})+\int_{0}^t(\mathcal{T}+\tau)^{3-\sigma}\|\nabla \overline{\mathbf{u}} \|^2_{\dot{H}^{s-3}}d\tau \\
			&\lesssim \mathcal{T}^{3-\sigma}\|(a_0,\mathbf{u}_0,\mathbf{B}_0) \|_{H^{s-3}}^2+ \mathcal{T}^{-1}\int_{0}^t(\mathcal{T}+\tau)^{3-\sigma}\|\overline{\mathcal{D}} \|^2_{\dot{H}^{s-3}}d\tau\nonumber \\
			&\quad+E_1^{\frac32}(t)+E_3^{\frac32}(t)+N_3^{\frac32}(t).\nonumber
		\end{align}
		
		To capture the damping effect of $\overline{\mathcal{D}}$ at lower regularity level, we employ an analogous energy functional. Taking its time derivative and applying the product rule yields
		\begin{align*}
			\frac{d}{d\tau}&\Big[(\mathcal{T}+\tau)^{3-\sigma}\int_{\mathbb{T}^3}\nabla^{s-4}\overline{\mathbf{u}}\cdot \nabla^{s-4}\nabla\overline{\mathcal{D}}  dx\Big]-(3-\sigma)(\mathcal{T}+\tau)^{2-\sigma}\int_{\mathbb{T}^3}\nabla^{s-4}\overline{\mathbf{u}}\cdot \nabla^{s-4}\nabla\overline{\mathcal{D}}  dx\\
			=&(\mathcal{T}+\tau)^{3-\sigma}\int_{\mathbb{T}^3}\nabla^{s-4}\partial_t\overline{\mathbf{u}}\cdot \nabla^{s-4}\nabla\overline{\mathcal{D}}  dx+(\mathcal{T}+\tau)^{3-\sigma}\int_{\mathbb{T}^3}\nabla^{s-4}\overline{\mathbf{u}}\cdot \nabla^{s-4}\nabla\partial_t\overline{\mathcal{D}}  dx.\nonumber
		\end{align*}
		
		Substituting the equations \eqref{equ bu} and \eqref{equ bD} into the right-hand side, we get
		\begin{align*}
			(\mathcal{T}+\tau)^{3-\sigma}&\int_{\mathbb{T}^3}\nabla^{s-4}\partial_t\overline{\mathbf{u}}\cdot \nabla^{s-4}\nabla\overline{\mathcal{D}}+\nabla^{s-4}\overline{\mathbf{u}}\cdot \nabla^{s-4}\nabla\partial_t\overline{\mathcal{D}} dx+(\mathcal{T}+\tau)^{3-\sigma}\|\overline{\mathcal{D}} \|^2_{\dot{H}^{s-3}}\\
			=&(\mathcal{T}+\tau)^{3-\sigma}\int_{\mathbb{T}^3}\nabla^{s-4} (\mu\Delta  \overline{\mathbf{u}}+(\lambda+\mu)\nabla (\operatorname{div} \overline{\mathbf{u}})+\overline{F_{\overline{\mathbf{u}}}} )\cdot \nabla^{s-4}\nabla\overline{\mathcal{D}}dx\\
			&+(\mathcal{T}+\tau)^{3-\sigma}\int_{\mathbb{T}^3}\nabla^{s-4}\overline{\mathbf{u}}\cdot \nabla^{s-4}\nabla (-2\overline{\operatorname{div}\mathbf{u}}+\overline{F_{\mathcal{D}}} ) dx\\
			\leq &\frac{1}{4}(\mathcal{T}+\tau)^{3-\sigma}\|\overline{\mathcal{D}} \|_{\dot{H}^{s-3}}^2+ C(\mathcal{T}+\tau)^{3-\sigma}\|\nabla \mathbf{u} \|_{\dot{H}^{s-3}}^2\\
			&+C(\mathcal{T}+\tau)^{3-\sigma}\|\overline{F_{\overline{\mathbf{u}}}}\|_{\dot{H}^{s-4}}^2+C(\mathcal{T}+\tau)^{3-\sigma}\|\overline{F_{\mathcal{D}}} \|_{\dot{H}^{s-4}}^2.
		\end{align*}
		
		Plugging the nonlinear estimates \eqref{est Fbu hs-4} and \eqref{est FbD hs-3} back into the differential inequality, we deduce that
		\begin{align*}
			\frac{d}{d\tau}&\Big[(\mathcal{T}+\tau)^{3-\sigma}\int_{\mathbb{T}^3}\nabla^{s-4}\overline{\mathbf{u}}\cdot \nabla^{s-4}\nabla\overline{\mathcal{D}}  dx\Big]+\frac12(\mathcal{T}+\tau)^{3-\sigma} \|\overline{\mathcal{D}} \|^2_{\dot{H}^{s-3}}\lesssim (\mathcal{T}+\tau)^{3-\sigma}\|\nabla \mathbf{u} \|_{\dot{H}^{s-3}}^2\\
			&+(\mathcal{T}+\tau)^{3-\sigma}(1+\| (a,\mathbf{B} ) \|_{H^{s-1}})^2\|(a,\mathbf{u},  \mathbf{B} ) \|_{H^{s-1}}^2(\|\overline{\mathcal{D}} \|_{\dot{H}^{s-3}}+\|\nabla\mathbf{u}\|_{H^{s-3}}+\|(\widetilde{a},\partial_{x_3}\widetilde{\mathbf{B}}) \|_{H^{s-4}})^2\\
			&+(\mathcal{T}+\tau)^{3-\sigma} (1+\|(a,\mathbf{B}) \|_{H^{s-1}})^2(\|(a,\mathbf{B}) \|_{H^{s-3}}\|\nabla \mathbf{u} \|_{H^{s-3}}+\| (a,\mathbf{B})\|_{H^{s-2}}\|\nabla \mathbf{u}\|_{H^{s-4}} )^2.
		\end{align*}
		
		Integrating this over time and applying H\"older inequality, we deduce 
		\begin{align*}
			(\mathcal{T}+t)^{3-\sigma}&\int_{\mathbb{T}^3}\nabla^{s-4}\overline{\mathbf{u}}\cdot \nabla^{s-4}\nabla\overline{\mathcal{D}}  dx+\frac12\int_{0}^t(\mathcal{T}+\tau)^{3-\sigma} \|\overline{\mathcal{D}} \|^2_{\dot{H}^{s-3}}d\tau \\
			\lesssim& \mathcal{T}^{3-\sigma}\|(a_0,\mathbf{u}_0,\mathbf{B}_0) \|_{H^{s-1}}^2+  \int_{0}^t(\mathcal{T}+\tau)^{3-\sigma}\|\nabla \mathbf{u} \|_{\dot{H}^{s-3}}^2d\tau
			+E_1^2(t)+E_3^2(t)+N_3^2(t).
		\end{align*}
		
		Multiplying this inequality  by a sufficiently small positive constant and adding it  to \eqref{est buD hs-3},  this yields the estimate for the averaged velocity and  the quantity $\overline{\mathcal{D}}$:
		\begin{align*}
			(\mathcal{T}+t)^{3-\sigma}&(\|\overline{\mathbf{u}}(t) \|^2_{\dot{H}^{s-3}} + \|\overline{\mathcal{D}}(t) \|^2_{\dot{H}^{s-3}})+\int_{0}^t(\mathcal{T}+\tau)^{3-\sigma}(\|\nabla \overline{\mathbf{u}} \|^2_{\dot{H}^{s-3}}+\| \overline{\mathcal{D}} \|^2_{\dot{H}^{s-3}} )d\tau \\
			\lesssim& \mathcal{T}^{3-\sigma}\|(a_0,\mathbf{u}_0,\mathbf{B}_0) \|_{H^{s-1}}^2+E_1^{\frac32}(t)+E_3^{\frac32}(t)+N_3^{\frac32}(t).\nonumber
		\end{align*}
		
		Combining this with the estimates for $\|(\overline{a},\overline{\mathbf{B}}) \|_{\dot{H}^{s-3}}$,  and recalling the definition of the lower-order averaged energy $\overline{E_3}(t)$, we conclude that
		\begin{align}\label{est bE3 s-3}
			\overline{E_3}(t)\lesssim   & \mathcal{T}^{3-\sigma}\|(a_0,\mathbf{u}_0,\mathbf{B}_0) \|_{H^{s-1}}^2+\mathcal{T}^{\frac{2\sigma-1}{2}}E_0^{\frac12}(t)E_2^{\frac12}(t)
			+E_{total}^{\frac32}(t).
		\end{align}
		
		We now turn to the oscillatory counterparts $(\widetilde{a},\widetilde{\mathbf{u}},\widetilde{\mathbf{B}})$.  For $\|\widetilde{a} \|_{\dot{H}^{s-3}}$, we follow the weighted energy identity from the high-order derivation with the time weight $(\mathcal{T}+\tau)^{3-\sigma}$:
		\begin{align*}
			\frac12\frac{d}{d\tau}&\Big[ (\mathcal{T}+\tau)^{3-\sigma}\int_{\mathbb{T}^3}\frac{P^{\prime}(1+\overline{a})}{(1+\overline{a})^2}|\nabla^{s-3} \widetilde{a}|^2dx\Big]+(\mathcal{T}+\tau)^{3-\sigma}\int_{\mathbb{T}^3}\frac{P^{\prime}(1+\overline{a})}{(1+\overline{a})} \nabla^{s-3}\operatorname{div}\widetilde{\mathbf{u}}\cdot \nabla^{s-3} \widetilde{a}dx\\
			=&\frac{3-\sigma}{2} (\mathcal{T}+\tau)^{2-\sigma}\int_{\mathbb{T}^3} \frac{P^{\prime}(1+\overline{a})}{(1+\overline{a})^2}|\nabla^{s-3}\widetilde{a}|^2dx+\mathcal{A}_1^*+\mathcal{A}_2^*+\mathcal{A}_3^*, 
		\end{align*}
		where the residual terms are defined as
		\begin{align*}
			\mathcal{A}_1^*=&\frac12(\mathcal{T}+\tau)^{3-\sigma}\int_{\mathbb{T}^3}\partial_{\tau}\Big(\frac{P^{\prime}(1+\overline{a})}{(1+\overline{a})^2} \Big)|\nabla^{s-3}\widetilde{a}|^2 dx,\\
			\mathcal{A}_2^*=&-(\mathcal{T}+\tau)^{3-\sigma}\int_{\mathbb{T}^3}\frac{P^{\prime}(1+\overline{a})}{(1+\overline{a})^2} [\nabla^{s-3},\overline{a}]\operatorname{div}\widetilde{\mathbf{u}}\cdot \nabla^{s-3} \widetilde{a}dx,\\
			\mathcal{A}_3^*=&- (\mathcal{T}+\tau)^{3-\sigma}\int_{\mathbb{T}^3}\frac{P^{\prime}(1+\overline{a})}{(1+\overline{a})^2} \nabla^{s-3}(\widetilde{a}\operatorname{div}\mathbf{u}+\widetilde{\mathbf{u}\cdot\nabla a}) \cdot\nabla^{s-3}\widetilde{a} dx.
		\end{align*}
		
		For the first term on the right-hand side, the increased time weight introduces a different scaling. Integrating over time and applying Sobolev interpolation, we obtain
		\begin{align*}
			\frac{3-\sigma}{2}\int_{0}^t&\int_{\mathbb{T}^3} (\mathcal{T}+\tau)^{2-\sigma} \frac{P^{\prime}(1+\overline{a})}{(1+\overline{a})^2}|\nabla^{s-3}\widetilde{a}|^2dxd\tau\\
			\lesssim& \int_{0}^t(\mathcal{T}+\tau)^{2-\sigma}\|\widetilde{a} \|_{H^{s-3}}^2d\tau\\
			\lesssim&\int_{0}^t (\mathcal{T}+\tau)^{\frac{-1-\sigma}{4}}\| \widetilde{a}\|_{H^{s}}^{\frac12}(\mathcal{T}+\tau)^{\frac{3(3-\sigma)}{4}}\|\widetilde{a} \|^{\frac32}_{H^{s-4}}d\tau\\
			\lesssim & E_0^{\frac14}(t)N_3^{\frac34}(t).
		\end{align*}
		
		To estimate $\mathcal{A}_1^*$, we substitute the equation for $\overline{a}$ into the time derivative of the weight, which yields
		\begin{align*}
			\int_{0}^t\mathcal{A}_1^*d\tau=&-\frac12\int_{0}^t\int_{\mathbb{T}^3} (\mathcal{T}+\tau)^{3-\sigma} \Big(\frac{P^{\prime}(1+\overline{a})}{(1+\overline{a})^2} \Big)^{\prime}(\operatorname{div}\overline{\mathbf{u}}+\overline{\mathbf{u}\cdot\nabla a}+\overline{a\operatorname{div}\mathbf{u}} )|\nabla^{s-3}\widetilde{a}|^2 dxd\tau\\
			\lesssim &\sup_{\tau\in[0,t]}(\mathcal{T}+\tau)^{3-\sigma}\|\widetilde{a} \|_{H^{s-3}}^2(1+\|a \|_{H^{3}} ) \int_{0}^t\| \mathbf{u}\|_{H^{3}}d\tau\\
			\lesssim & E_3^{\frac32}(t) .
		\end{align*}
		
		For the commutator term $\mathcal{A}_2^*$, integration by parts provides the bound
		\begin{align*}
			\int_{0}^t\mathcal{A}_2^*d\tau=&\int_{0}^t\int_{\mathbb{T}^3}(\mathcal{T}+\tau)^{3-\sigma}\nabla\Big(\frac{P^{\prime}(1+\overline{a})}{(1+\overline{a})^2} [\nabla^{s-3},\overline{a}]\operatorname{div}\widetilde{\mathbf{u}} \Big)\cdot\nabla^{s-4} \widetilde{a}dxd\tau \\
			\lesssim &\sup_{\tau\in[0,t]}\| \overline{a}\|_{H^{s-1}}\int_{0}^t(\mathcal{T}+\tau)^{3-\sigma}\| \operatorname{div}\widetilde{\mathbf{u}}\|_{H^{s-3}}\|\widetilde{a} \|_{H^{s-4}}d\tau\\
			\lesssim & (E_1(t)+E_3(t)+N_3(t))^{\frac32}.
		\end{align*}
		
		Similarly, a direct computation for the remaining nonlinear convection term $\mathcal{A}_3^*$ gives 
		\begin{align*}
			\int_{0}^t\mathcal{A}_3^*d\tau=&\int_{0}^t\int_{\mathbb{T}^3} (\mathcal{T}+\tau)^{3-\sigma}\nabla\Big(\frac{P^{\prime}(1+\overline{a})}{(1+\overline{a})^2} \nabla^{s-3}(\widetilde{a}\operatorname{div}\mathbf{u}+\widetilde{\mathbf{u}\cdot\nabla a})\Big)\cdot \nabla^{s-4}\widetilde{a} dxd\tau\\
			\lesssim &\sup_{\tau\in[0,t]} \|(a,\mathbf{u}) \|_{H^{s-1}} \int_{0}^t(\mathcal{T}+\tau)^{3-\sigma} \|  (\widetilde{a},\nabla\mathbf{u})\|_{H^{3}}\|\widetilde{a} \|_{H^{s-4}}d\tau\\
			\lesssim & (E_1(t)+E_3(t)+N_3(t))^{\frac32}.
		\end{align*}
		
		Combining these estimates and integrating over time, we obtain
		\begin{align}\label{est ta hs-3}
			\frac12(\mathcal{T}+t)^{3-\sigma}&\int_{\mathbb{T}^3}\frac{P^{\prime}(1+\overline{a})}{(1+\overline{a})^2}|\nabla^{s-3} \widetilde{a}|^2dx+\int_{0}^t\int_{\mathbb{T}^3}(\mathcal{T}+\tau)^{3-\sigma}\frac{P^{\prime}(1+\overline{a})}{(1+\overline{a})} \nabla^{s-3}\operatorname{div}\widetilde{\mathbf{u}}\cdot \nabla^{s-3} \widetilde{a}dxd\tau\\
			\lesssim & \mathcal{T}^{3-\sigma}\|a_0 \|_{H^{s-1}}^2+ E_0^{\frac14}(t)N_3^{\frac34}(t)+(E_1(t)+E_3(t)+N_3(t))^{\frac32}.\nonumber
		\end{align}
		
		For the oscillatory magnetic field $\widetilde{\mathbf{B}}$, we employ an analogous time-weighted energy identity:
		\begin{align*}
			\frac{1}{2}&\frac{d}{d\tau}\Big[(\mathcal{T}+\tau)^{3-\sigma}\int_{\mathbb{T}^3}\frac{1}{1+\overline{a}}|  \nabla^{s-3}\widetilde{\mathbf{B}} |^2dx\Big]\\
			&-(\mathcal{T}+\tau)^{3-\sigma}\int_{\mathbb{T}^3}\frac{1}{1+\overline{a}}\nabla^{s-3}(\partial_{x_3}\mathbf{u}-\mathbf{e}_3\operatorname{div} \widetilde{\mathbf{u}})\cdot\nabla^{s-3}\widetilde{\mathbf{B}}dx\\
			=&\frac{3-\sigma}{2}(\mathcal{T}+\tau)^{2-\sigma}\int_{\mathbb{T}^3}\frac{1}{1+\overline{a}}|  \nabla^{s-3}\widetilde{\mathbf{B}} |^2dx-\frac12(\mathcal{T}+\tau)^{3-\sigma}\int_{\mathbb{T}^3}\frac{\partial_{\tau}\overline{a}}{(1+\overline{a})^2}|  \nabla^{s-3}\widetilde{\mathbf{B}} |^2dx\\
			&+ (\mathcal{T}+\tau)^{3-\sigma}\int_{\mathbb{T}^3}\frac{1}{1+\overline{a}}\nabla^{s-3}(-\widetilde{\mathbf{u}\cdot \nabla \mathbf{B}}+\widetilde{\mathbf{B}\cdot \nabla \mathbf{u}}-\widetilde{\mathbf{B}\operatorname{div}\mathbf{u}} )\cdot\nabla^{s-3}\widetilde{\mathbf{B}} dx\\
			=&:\mathcal{B}_1^*+\mathcal{B}_2^*+\mathcal{B}_3^*.
		\end{align*}
		
		The second term on the left-hand side is isolated to be cancelled by the corresponding linear part in the velocity equation.   
		To control the term $\mathcal{B}_1^*$ generated by the time weight, we interpolate between the top-order magnetic energy and the lower-order magnetic damping:
		\begin{align*}
			\int_{0}^t\mathcal{B}_1^*d\tau=&\frac{3-\sigma}{2}\int_{0}^t\int_{\mathbb{T}^3}(\mathcal{T}+\tau)^{2-\sigma}\frac{1}{1+\overline{a}}|  \nabla^{s-3}\widetilde{\mathbf{B}} |^2dxd\tau\\
			\lesssim &\int_{0}^t\int_{\mathbb{T}^3}(\mathcal{T}+\tau)^{2-\sigma}|  \nabla^{s-3}\widetilde{\mathbf{B}} |^2dxd\tau\\
			\lesssim & \int_{0}^t(\mathcal{T}+\tau)^{\frac{-1-\sigma}{4}}\|\widetilde{\mathbf{B}} \|_{H^{s}}^{\frac12} (\mathcal{T}+\tau)^{\frac{3(3-\sigma)}{4}}\|\partial_{x_3}\widetilde{\mathbf{B}} \|_{H^{s-4}}^{\frac32}d\tau\\
			\lesssim &  E_0^{\frac14}(t)N_3^{\frac34}(t).
		\end{align*} 
		
		Invoking the equation for $\overline{a}$, $\mathcal{B}_2^*$ is bounded by
		\begin{align*}
			\int_{0}^t\mathcal{B}_2^*d\tau=& -\frac12\int_{0}^t\int_{\mathbb{T}^3}(\mathcal{T}+\tau)^{3-\sigma}\frac{\partial_{\tau}\overline{a}}{(1+\overline{a})^2}|  \nabla^{s-3}\widetilde{\mathbf{B}} |^2dxd\tau\\
			=& \frac12\int_{0}^t\int_{\mathbb{T}^3}(\mathcal{T}+\tau)^{3-\sigma}\frac{1}{(1+\overline{a})^2}(\operatorname{div}\overline{\mathbf{u}}+\overline{\mathbf{u}\cdot \nabla a}+\overline{a\operatorname{div}\mathbf{u}} )|  \nabla^{s-3}\widetilde{\mathbf{B}} |^2dxd\tau\\
			\lesssim & \sup_{\tau\in[0,t]}(\mathcal{T}+\tau)^{\frac{3-\sigma}{2}}\| \widetilde{\mathbf{B}}\|_{H^{s-3}}\int_{0}^t(\mathcal{T}+\tau)^{\frac{1}{2}}\| \nabla \mathbf{u}\|_{H^{2}}(\mathcal{T}+\tau)^{\frac{2-\sigma}{2}}\|\partial_{x_3} \widetilde{\mathbf{B}}\|_{H^{s-3}}d\tau\\
			\lesssim & E_3(t)N_1^{\frac14}(t)N_3^{\frac14}(t).
		\end{align*} 
		
		For the nonlinear magnetic terms $\mathcal{B}_3^*$, we again split them into two components:
		\begin{align*}
			\mathcal{B}_3^*=&(\mathcal{T}+\tau)^{3-\sigma}\int_{\mathbb{T}^3}\frac{1}{1+\overline{a}}\nabla^{s-3}(-\widetilde{\mathbf{u}\cdot \nabla \mathbf{B}})\cdot\nabla^{s-3}\widetilde{\mathbf{B}} dx\\
			&+(\mathcal{T}+\tau)^{3-\sigma}\int_{\mathbb{T}^3}\frac{1}{1+\overline{a}}\nabla^{s-3}(\widetilde{\mathbf{B}\cdot \nabla \mathbf{u}}-\widetilde{\mathbf{B}\operatorname{div}\mathbf{u}} )\cdot\nabla^{s-3}\widetilde{\mathbf{B}} dx\\
			=&: \mathcal{B}_{3,1}^*+\mathcal{B}_{3,2}^*.
		\end{align*}
		The estimate for $\mathcal{B}_{3,1}^*$ is slightly different from Subsection \ref{Sec est tE1}. Integrating by parts gives 
		\begin{align*}
			\mathcal{B}_{3,1}^*=& -(\mathcal{T}+\tau)^{3-\sigma}\int_{\mathbb{T}^3}\frac{1}{1+\overline{a}}\nabla^{s-2}(\widetilde{\mathbf{u}\cdot \nabla   \mathbf{B} } )\cdot\nabla^{s-4}\widetilde{\mathbf{B}} dx\\
			&-(\mathcal{T}+\tau)^{3-\sigma}\int_{\mathbb{T}^3}\nabla\Big(\frac{1}{1+\overline{a}}\Big)\nabla^{s-3}(\widetilde{\mathbf{u}\cdot \nabla   \mathbf{B} } )\cdot\nabla^{s-4}\widetilde{\mathbf{B}} dx\\
			\lesssim & (\mathcal{T}+\tau)^{3-\sigma}\|\mathbf{B} \|_{H^{s-1}}\|\nabla\mathbf{u} \|_{H^{s-3}}\|\widetilde{\mathbf{B}} \|_{H^{s-4}}.
		\end{align*}
		
		Following the same procedure, we first isolate the top-order terms that will cancel with those in the velocity equation; the residual components of  $\mathcal{B}_{3,2}^*$ are bounded via product estimates,
		\begin{align*}
			\mathcal{B}_{3,2}^*&-(\mathcal{T}+\tau)^{3-\sigma}\int_{\mathbb{T}^3}\frac{1}{1+\overline{a}}(\overline{\mathbf{B}}\cdot \nabla \nabla^{s-3}\widetilde{\mathbf{u}}-\overline{\mathbf{B}}\nabla^{s-3}\operatorname{div}\widetilde{\mathbf{u}} )\cdot\nabla^{s-3}\widetilde{\mathbf{B}} dx\\
			=&-(\mathcal{T}+\tau)^{3-\sigma}\sum_{k=1}^{s-3}\binom{s-3}{k}\\
			&\qquad\qquad\times\int_{\mathbb{T}^3}\nabla\Big(\frac{1}{1+\overline{a}} (\nabla^{k}\overline{\mathbf{B}}\cdot \nabla \nabla^{s-3-k}\widetilde{\mathbf{u}}-\nabla^{k}\overline{\mathbf{B}}\nabla^{s-3-k}\operatorname{div}\widetilde{\mathbf{u}} )\Big)\cdot\nabla^{s-4}\widetilde{\mathbf{B}} dx \\
			&-(\mathcal{T}+\tau)^{3-\sigma}\int_{\mathbb{T}^3}\nabla\Big(\frac{1}{1+\overline{a}}\nabla^{s-3}(\widetilde{\mathbf{B}}\cdot \nabla \mathbf{u}-\widetilde{\mathbf{B}}\operatorname{div}\mathbf{u} )\Big)\cdot\nabla^{s-4}\widetilde{\mathbf{B}} dx\\
			\lesssim  &(\mathcal{T}+\tau)^{3-\sigma}(\| \overline{\mathbf{B}}\|_{H^{4}}\|\nabla \mathbf{u} \|_{H^{s-3}}+\| \overline{\mathbf{B}}\|_{H^{s-2}}\|\nabla \mathbf{u} \|_{H^{3}} )\| \widetilde{\mathbf{B}}\|_{H^{s-4}} \\
			&+(\mathcal{T}+\tau)^{3-\sigma}(\| \widetilde{\mathbf{B}}\|_{H^{3}}\|\nabla \mathbf{u} \|_{H^{s-2}}+\| \widetilde{\mathbf{B}}\|_{H^{s-2}}\|\nabla \mathbf{u} \|_{H^{3}} )\| \widetilde{\mathbf{B}}\|_{H^{s-4}} .
		\end{align*}
		
		Integrating over time yields
		\begin{align*}
			\int_{0}^t\mathcal{B}_3^*d\tau&-\int_{0}^t\int_{\mathbb{T}^3}(\mathcal{T}+\tau)^{3-\sigma}\frac{1}{1+\overline{a}}(\overline{\mathbf{B}}\cdot \nabla \nabla^{s-3}\widetilde{\mathbf{u}}-\overline{\mathbf{B}}\nabla^{s-3}\operatorname{div}\widetilde{\mathbf{u}} )\cdot\nabla^{s-3}\widetilde{\mathbf{B}} dxd\tau\\
			\lesssim &\sup_{\tau\in[0,t]}\|\overline{\mathbf{B}} \|_{H^{s-1}}\int_{0}^t (\mathcal{T}+\tau)^{3-\sigma}\|\nabla\mathbf{u} \|_{H^{s-3}}\|\widetilde{\mathbf{B}} \|_{H^{s-4}}d\tau\\
			&+\sup_{\tau\in[0,t]}\|  \mathbf{u} \|_{H^{s-1}}\int_{0}^t(\mathcal{T}+\tau)^{3-\sigma}  \| \widetilde{\mathbf{B}}\|^2_{H^{s-4}} d\tau\\
			&+\sup_{\tau\in[0,t]}\| \widetilde{\mathbf{B}}\|_{H^{s-2}}\int_{0}^t(\mathcal{T}+\tau)^{3-\sigma} \|\nabla \mathbf{u} \|_{H^{3}}  \| \widetilde{\mathbf{B}}\|_{H^{s-4}} d\tau\\
			\lesssim& (E_1(t)+E_3(t)+N_1(t)+N_3(t))^{\frac32} .
		\end{align*}
		
		Combining these estimates and integrating over time, it follows that
		\begin{align}\label{est tb hs-3}
			(\mathcal{T}+t)^{3-\sigma}&\|\widetilde{\mathbf{B}} \|_{H^{s-3}}^2-\int_{0}^{t}\int_{\mathbb{T}^3}(\mathcal{T}+\tau)^{3-\sigma}\frac{1}{1+\overline{a}}\nabla^{s-3}(\partial_{x_3}\mathbf{u}-\mathbf{e}_3\operatorname{div} \widetilde{\mathbf{u}})\cdot\nabla^{s-3}\widetilde{\mathbf{B}}dxd\tau\\
			&-\int_{0}^{t}\int_{\mathbb{T}^3}(\mathcal{T}+\tau)^{3-\sigma}\frac{1}{1+\overline{a}}(\overline{\mathbf{B}}\cdot \nabla \nabla^{s-3}\widetilde{\mathbf{u}}-\overline{\mathbf{B}}\nabla^{s-3}\operatorname{div}\widetilde{\mathbf{u}} )\cdot\nabla^{s-3}\widetilde{\mathbf{B}} dxd\tau\nonumber\\
			\lesssim & \mathcal{T}^{3-\sigma}\|(a_0,\mathbf{B}_0) \|_{H^{s-1}}^2+E_0^{\frac14}(t)N_3^{\frac34}(t)+(E_1(t)+E_3(t)+N_1(t)+N_3(t))^{\frac32}.\nonumber
		\end{align}
		
		For the oscillatory velocity  $\widetilde{\mathbf{u}}$, we recall its equation:
		\begin{align*}
			\partial_t \widetilde{\mathbf{u}}  -
			\mu\Delta  \widetilde{\mathbf{u}}-(\lambda+\mu)\nabla (\operatorname{div} \widetilde{\mathbf{u}})+ \widetilde{h(a)\nabla a}
			-\widetilde{\rho^{-1}\partial_{x_3} \mathbf{B}}
			+\widetilde{\rho^{-1}\nabla  B_3}=\widetilde{F_{\mathbf{u}}},
		\end{align*} 
		where $h(a)=\frac{P^{\prime}(1+a)}{1+a}$ and the nonlinear term is defined by
		\begin{align*}
			F_{\mathbf{u}}=&-\mathbf{u}\cdot \nabla \mathbf{u}+\rho^{-1}(\mathbf{B}\cdot \nabla \mathbf{B}-\mathbf{B}\nabla \mathbf{B})-\mu I(a)\Delta \mathbf{u}-(\lambda +\mu) I(a)\nabla\operatorname{div}\mathbf{u}.
		\end{align*}
		Taking the appropriate time-weighted $\dot{H}^{s-3}$ inner product against $\widetilde{\mathbf{u}}$ yields
		\begin{align*}
			\frac12\frac{d}{d\tau} & \Big((\mathcal{T}+\tau)^{3-\sigma}\|\widetilde{\mathbf{u}} \|^2_{\dot{H}^{s-3}}\Big)  -\frac{3-\sigma}{2}(\mathcal{T}+\tau)^{2-\sigma}\|\widetilde{\mathbf{u}} \|^2_{\dot{H}^{s-3}}+
			\mu(\mathcal{T}+\tau)^{3-\sigma}\|\nabla \widetilde{\mathbf{u}} \|^2_{\dot{H}^{s-3}}\\
			&+
			(\lambda+\mu)(\mathcal{T}+\tau)^{3-\sigma}\| \operatorname{div}  \widetilde{\mathbf{u}} \|^2_{\dot{H}^{s-3}} 
			\\
			=&(\mathcal{T}+\tau)^{3-\sigma}\int_{\mathbb{T}^3}\nabla^{s-3}\Big[-\widetilde{h(a)\nabla a}
			+(\widetilde{\rho^{-1}\partial_{x_3} \mathbf{B}}
			-\widetilde{\rho^{-1}\nabla  B_3})+\widetilde{F_{\mathbf{u}}} \Big]\cdot\nabla^{s-3}\widetilde{\mathbf{u}}dx\\
			=&(\mathcal{T}+\tau)^{3-\sigma}(\mathcal{U}_1^*+\mathcal{U}_2^*)+(\mathcal{T}+\tau)^{3-\sigma}\int_{\mathbb{T}^3}\nabla^{s-3}\widetilde{F_{\mathbf{u}}}\cdot\nabla^{s-3}\widetilde{\mathbf{u}}dx,
		\end{align*}
		where  the coupling terms are denoted by
		\begin{align*}
			\mathcal{U}_1^*=&-\int_{\mathbb{T}^3}\nabla^{s-3}\big(\widetilde{h(a)\nabla a}
			\big)\cdot\nabla^{s-3}\widetilde{\mathbf{u}}dx,\\
			\mathcal{U}_2^*=& \int_{\mathbb{T}^3}\nabla^{s-3}\big(
			\widetilde{\rho^{-1}\partial_{x_3} \mathbf{B}}
			-\widetilde{\rho^{-1}\nabla  B_3} \big)\cdot\nabla^{s-3}\widetilde{\mathbf{u}}dx.
		\end{align*}
		
		To absorb the negative term arising from the time weight on the left-hand side, the Poincar\'e inequality yields
		\begin{align*}
			\frac{3-\sigma}{2}(\mathcal{T}+\tau)^{2-\sigma}\|\widetilde{\mathbf{u}} \|^2_{H^{s-3}}\leq& \frac{3-\sigma}{2\mathcal{T}}(\mathcal{T}+\tau)^{3-\sigma}\|\nabla \widetilde{\mathbf{u}} \|^2_{H^{s-3}}\\
			\leq& \frac{\mu}{4}(\mathcal{T}+\tau)^{3-\sigma}\|\nabla \widetilde{\mathbf{u}} \|^2_{H^{s-3}},
		\end{align*}
		provided we choose $\mathcal{T} \geq \frac{2(3-\sigma)}{\mu}$. 
		
		For the coupling terms $\mathcal{U}_1^*$ and $\mathcal{U}_2^*$, similar procedure as in Section \ref{Sec est tE1} gives that
		\begin{align*}
			\mathcal{U}_1^*-\int_{\mathbb{T}^3}h(\overline{a}) \nabla^{s-3}\operatorname{div}\widetilde{\mathbf{u}} \nabla^{s-3} \widetilde{a}dx
			\lesssim  \|a \|_{H^{s-3}}\|\widetilde{a} \|_{H^{s-4}}\|\nabla\widetilde{\mathbf{u}} \|_{H^{s-3}},
		\end{align*}
		and
		\begin{align*}
			\mathcal{U}_2^*&+\int_{\mathbb{T}^3}\frac{1}{1+\overline{a}}\nabla^{s-3}(\partial_{x_3}\mathbf{u}-\mathbf{e}_3\operatorname{div} \widetilde{\mathbf{u}})\cdot\nabla^{s-3}\widetilde{\mathbf{B}}dx
			\lesssim \|(a,\mathbf{B}) \|_{H^{s-3}}\|(\widetilde{a},\widetilde{\mathbf{B}}) \|_{H^{s-4}}\|\nabla \mathbf{u} \|_{H^{s-3}}.
		\end{align*}
		
		Moreover, the lower regularity needed here allows the nonlinear term to be bounded by a straightforward product estimate,
		\begin{align*}
			\int_{\mathbb{T}^3}\nabla^{s-3}\widetilde{F_{\mathbf{u}}} \cdot\nabla^{s-3}\widetilde{\mathbf{u}}dx &+\int_{\mathbb{T}^3}\frac{1}{1+\overline{a}}(\overline{\mathbf{B}}\cdot \nabla \nabla^{s-3}\widetilde{\mathbf{u}}-\overline{\mathbf{B}}\nabla^{s-3}\operatorname{div}\widetilde{\mathbf{u}} )\cdot\nabla^{s-3}\widetilde{\mathbf{B}} dx\\
			\lesssim &  \|(a,\mathbf{u},\mathbf{B}) \|_{H^{s-3}}(\|(\widetilde{a},\widetilde{\mathbf{B}}) \|_{H^{s-4}}+\|\nabla   \mathbf{u}\|_{H^{s-3}})\|\nabla \mathbf{u} \|_{H^{s-3}}.
		\end{align*}
		
		Invoking the above estimates, isolating the exact cancellation terms, and integrating over time, we obtain 
		\begin{align*}
			(\mathcal{T}+t)^{3-\sigma}&\|\widetilde{\mathbf{u}} \|_{\dot{H}^{s-3}}^2+\int_{0}^{t}(\mathcal{T}+\tau)^{3-\sigma} \|\nabla \widetilde{\mathbf{u}} \|^2_{\dot{H}^{s-3}}d\tau\\
			&-\int_{0}^{t}\int_{\mathbb{T}^3}(\mathcal{T}+\tau)^{3-\sigma}h(\overline{a}) \nabla^{s-3}\operatorname{div}\widetilde{\mathbf{u}} \cdot\nabla^{s-3} \widetilde{a}dxd\tau\\
			&+\int_{0}^{t}\int_{\mathbb{T}^3}(\mathcal{T}+\tau)^{3-\sigma}\frac{1}{1+\overline{a}}\nabla^{s-3}(\partial_{x_3}\mathbf{u}-\mathbf{e}_3\operatorname{div} \widetilde{\mathbf{u}})\cdot\nabla^{s-3}\widetilde{\mathbf{B}}dxd\tau\\
			&+\int_{0}^{t}\int_{\mathbb{T}^3}(\mathcal{T}+\tau)^{3-\sigma}\Big(\frac{1}{1+\overline{a}}(\overline{\mathbf{B}}  \nabla \nabla^{s-3}\widetilde{ \mathbf{B}}- \overline{\mathbf{B}} \cdot\nabla\nabla^{s-3} \widetilde{ \mathbf{B}})\Big) \cdot\nabla^{s-3 }\widetilde{\mathbf{u}}dxd\tau\\
			\lesssim& \mathcal{T}^{3-\sigma}\|(a_0,\mathbf{u}_0,\mathbf{B}_0) \|_{H^{s}}^2+(E_1(t)+E_3(t)+N_3(t))^{\frac32}.
		\end{align*}
		
		Finally, combining this estimate with the bounds \eqref{est ta hs-3} and \eqref{est tb hs-3} for $\widetilde{a}$ and $\widetilde{\mathbf{B}}$, the remaining higher-order couplings cancel out, thereby it follows that 
		\begin{align}\label{est tE3 s-3}
			\widetilde{E}_3(t)\lesssim \mathcal{T}^{3-\sigma}\|(a_0,\mathbf{u}_0,\mathbf{B}_0) \|_{H^{s}}^2+E_0^{\frac14}(t)N_3^{\frac34}(t)+E_{total}^{\frac32}(t).
		\end{align}

		Taking the weighted inner product for $\widetilde{d}$ yields
		\begin{align*}
			\frac12\frac{d}{d\tau}\Big((\mathcal{T}+\tau)^{3-\sigma}&\|\widetilde{d} \|^2_{\dot{H}^{s-4}}\Big)-\frac{3-\sigma}{2}(\mathcal{T}+\tau)^{2-\sigma}\| \widetilde{d}\|^2_{\dot{H}^{s-4}}+\frac{2}{\nu}(\mathcal{T}+\tau)^{3-\sigma}\| \widetilde{d}\|^2_{\dot{H}^{s-4}}\\
			=&\int_{\mathbb{T}^3}(\mathcal{T}+\tau)^{3-\sigma}\nabla^{s-4}(\partial_{x_3}u_3-2\operatorname{div}\widetilde{\mathbf{G}})\cdot\nabla^{s-4} \widetilde{d}  dx\\
			&+\int_{\mathbb{T}^3}(\mathcal{T}+\tau)^{3-\sigma}\nabla^{s-4}\widetilde{F}_d\cdot\nabla^{s-4} \widetilde{d}  dx.
		\end{align*}
		
		Applying Cauchy-Schwarz and Young inequalities to the linear coupling term gives
		\begin{align}\label{est hs-4 dlin}
			\int_{\mathbb{T}^3}&(\mathcal{T}+\tau)^{3-\sigma}\nabla^{s-4}(\partial_{x_3}u_3-2\operatorname{div}\widetilde{\mathbf{G}})\cdot\nabla^{s-4} \widetilde{d}  dx\\
			&\leq \frac{1}{2\nu}(\mathcal{T}+\tau)^{3-\sigma}\|\widetilde{d} \|^2_{H^{s-4}}+C(\mathcal{T}+\tau)^{3-\sigma}(\|\nabla \widetilde{\mathbf{G}} \|^2_{H^{s-4}}+\|\nabla \widetilde{\mathbf{u}} \|^2_{H^{s-4}}).\nonumber
		\end{align}
		
		The corresponding nonlinear forcing term $\widetilde{F_d}$ is bounded similarly. Utilizing the standard product estimates, we find
		\begin{align*}
			\int_{\mathbb{T}^3}(\mathcal{T}+\tau)^{3-\sigma}\nabla^{s-4}\widetilde{F_d}\nabla^{s-4} \widetilde{d}  dx\leq  \frac{1}{2\nu}(\mathcal{T}+\tau)^{3-\sigma}\|\widetilde{d} \|^2_{H^{s-4}}+C(\mathcal{T}+\tau)^{3-\sigma}\| \widetilde{F_{d}}\|^2_{H^{s-4}},
		\end{align*}
		where 
		\begin{align}\label{est ftd hs-4}
			\|\widetilde{F_d} \|_{H^{s-4}}\lesssim & \|\widetilde{\mathbf{u}\cdot \nabla a} \|_{H^{s-4}}+\|\widetilde{a\operatorname{div}\mathbf{u}} \|_{H^{s-4}}+\|\widetilde{\mathbf{u}\cdot \nabla B_3} \|_{H^{s-4}}\\
			&+\|\widetilde{\mathbf{B}\cdot \nabla u_3} \|_{H^{s-4}}+\|\widetilde{B_3\operatorname{div}\mathbf{u}} \|_{H^{s-4}}\nonumber\\
			\lesssim & \|a \|_{H^{s-3}}\| \mathbf{u} \|_{H^{s-4}}+\| a\|_{H^{s-4}}\| \nabla \mathbf{u}\|_{H^{s-4}}\nonumber\\
			&+\| \mathbf{B}\|_{H^{s-3}}\|\mathbf{u} \|_{H^{s-4}}+\|\mathbf{B} \|_{H^{s-4}}\| \nabla \mathbf{u}\|_{H^{s-4}}\nonumber\\
			\lesssim & \|(a,\mathbf{B}) \|_{H^{s-3}}\| \nabla\mathbf{u} \|_{H^{s-3}}.\nonumber
		\end{align}
		Integrating over time, this yields
		\begin{align}\label{est hs-4 dnonl}
			\int_{0}^{t}&\int_{\mathbb{T}^3}(\mathcal{T}+\tau)^{3-\sigma}\nabla^{s-4}\widetilde{F_d}\cdot\nabla^{s-4} \widetilde{d}  dxd\tau  \\
			&\leq\frac{1}{2\nu}\int_{0}^{t}(\mathcal{T}+\tau)^{3-\sigma}\|\widetilde{d} \|^2_{H^{s-4}}d\tau +C\sup_{\tau\in[0,t]}  \|(a,\mathbf{B}) \|_{H^{s-3}}^2\int_{0}^{t}(\mathcal{T}+\tau)^{3-\sigma} \|\nabla \mathbf{u} \|_{H^{s-3}}^2d\tau.\nonumber
		\end{align}
		
		Combining \eqref{est hs-4 dlin}, \eqref{est hs-4 dnonl}, and the definition of the energy functionals $E_i(t)$, we deduce that
		\begin{align}\label{est td hs-4}
			(\mathcal{T}+t)^{3-\sigma}&\|\widetilde{d} \|^2_{H^{s-4}}+\frac{1}{\nu}\int_{0}^t(\mathcal{T}+\tau)^{3-\sigma}\| \widetilde{d}\|^2_{H^{s-4}}d\tau
			\\
			\lesssim& \mathcal{T}^{3-\sigma}\| (a_0,\mathbf{B}_0)\|_{H^{s-4}}^2+\int_{0}^t (\mathcal{T}+\tau)^{3-\sigma}(\|\nabla \widetilde{\mathbf{G}} \|^2_{H^{s-4}}+\|\nabla\widetilde{\mathbf{u}} \|^2_{H^{s-4}})d\tau+E_1(t)E_3(t).\nonumber
		\end{align}
		
		Turning to the auxiliary quantity $\widetilde{\mathbf{G}}$, taking the weighted $\dot{H}^{s-4}$ inner product of its equation \eqref{equ dG0} provides
		\begin{align}\label{est dtG hs-4}
			\frac12\frac{d}{d\tau}&\Big((\mathcal{T}+\tau)^{3-\sigma}\|\widetilde{\mathbf{G}} \|^2_{\dot{H}^{s-4}}\Big)-\frac{3-\sigma}{2}(\mathcal{T}+\tau)^{2-\sigma}\| \widetilde{\mathbf{G}}\|^2_{\dot{H}^{s-4}}+\nu(\mathcal{T}+\tau)^{3-\sigma}\| \nabla\widetilde{\mathbf{G}}\|^2_{\dot{H}^{s-4}}\\
			=& \int_{\mathbb{T}^3}(\mathcal{T}+\tau)^{3-\sigma}\nabla^{s-4}(\frac{2}{\nu}\mathbb{Q}\widetilde{\mathbf{u}}-\frac{1}{\nu}\Delta^{-1}\nabla\partial_{x_3}u_3+\mathbb{Q}\widetilde{F_\mathbf{G}})\nabla^{s-4} \widetilde{\mathbf{G}} dx\nonumber\\
			\leq & \frac{\nu}{2}(\mathcal{T}+\tau)^{3-\sigma}\|\nabla \widetilde{\mathbf{G}} \|_{H^{s-4}}^2+C(\mathcal{T}+\tau)^{3-\sigma}(\|\widetilde{\mathbf{u}} \|_{H^{s-5}}^2+\|\widetilde{F_\mathbf{G}} \|_{H^{s-5}}^2),\nonumber
		\end{align}
		Provided the time shift parameter $\mathcal{T}$ is chosen sufficiently large, the negative term arising from the time weight on the left-hand side is absorbed by the viscous dissipation.  For the nonlinear forcing $\widetilde{F_{\mathbf{G}}}$, invoking the estimate \eqref{est ftd hs-4}
		\begin{align*}
			\|\widetilde{F_{\mathbf{G}}} \|_{H^{s-5}}\lesssim & \|\widetilde{\mathbf{u}\cdot \nabla \mathbf{u}}\|_{H^{s-5}}+\|\widetilde{\mathbf{B}\cdot \nabla \mathbf{B}} \|_{H^{s-5}}+\|\widetilde{ \mathbf{B}\nabla \mathbf{B}} \|_{H^{s-5}}+\|\widetilde{k(a)\nabla a} \|_{H^{s-5}}+\|\widetilde{ I(a)\Delta \mathbf{u}}\|_{H^{s-5}}\\
			&+\|\widetilde{ I(a)\nabla\operatorname{div}\mathbf{u}} \|_{H^{s-5}}+\| \widetilde{I(a)\partial_{x_3} \mathbf{B} }\|_{H^{s-5}}+\| \widetilde{I(a)\nabla B_3 }\|_{H^{s-5}}\\
			&+\|\widetilde{I(a)\mathbf{B}\cdot \nabla \mathbf{B} } \|_{H^{s-5}}+\|\widetilde{I(a)\mathbf{B}\nabla \mathbf{B}}\|_{H^{s-5}}+\|\frac{1}{\nu}\Delta^{-1}\nabla  \widetilde{F_{d} } \|_{H^{s-5}}\\
			\lesssim & \| \widetilde{\mathbf{u}}\|_{H^{s-4}}\|\mathbf{u} \|_{H^{s-4}}+\|\widetilde{\mathbf{B}} \|_{H^{s-4}}\| \mathbf{B}\|_{H^{s-4}}+\|\widetilde{a} \|_{H^{s-4}}\|a \|_{H^{s-4}}\\
			&+\| \widetilde{a}\|_{H^{s-5}}\|\nabla\mathbf{u} \|_{H^{s-4}}+\| a\|_{H^{s-5}}\|\nabla\widetilde{\mathbf{u}}  \|_{H^{s-4}}\\
			&+\|a \|_{H^{s-5}}\|\partial_{x_3}\mathbf{B} \|_{H^{s-5}}
			+\|\widetilde{a} \|_{H^{s-5}}\|B_3 \|_{H^{s-4}}+\| \overline{a}\|_{H^{s-5}}\|\widetilde{\mathbf{B}}_3 \|_{H^{s-4}}\\
			&+\|\widetilde{a} \|_{H^{s-5}}\| \mathbf{B}\|^2_{H^{s-4}}+\|a \|_{H^{s-5}}\|\mathbf{B} \|_{H^{s-4}}\|\widetilde{\mathbf{B}} \|_{H^{s-4}}+\|(a,\mathbf{B}) \|_{H^{s-4}}\| \nabla\mathbf{u} \|_{H^{s-5}}\\
			\lesssim & (\|(a,\mathbf{u},\mathbf{B}) \|_{H^{s-4}}+\|(\mathbf{u},\mathbf{B}) \|^2_{H^{s-4}})\|(\widetilde{a},\nabla\mathbf{u},\widetilde{\mathbf{B}}) \|_{H^{s-4}}.
		\end{align*}
		Inserting the above estimate into \eqref{est dtG hs-4} and integrating over time, it follows that
		\begin{align}\label{est tG hs-4}
			(\mathcal{T}+t)^{3-\sigma}&\|\widetilde{\mathbf{G}} \|^2_{H^{s-4}}+\int_{0}^t(\mathcal{T}+\tau)^{3-\sigma}\| \nabla\widetilde{\mathbf{G}}\|^2_{H^{s-4}}d\tau\\
			\lesssim & \mathcal{T}^{3-\sigma}\| (a_0,\mathbf{u}_0,\mathbf{B}_0)\|_{H^{s-4}}^2+E_3(t) +  (E_1(t)+E_3(t)+N_3(t))^{ 2}.\nonumber
		\end{align}
		
		Multiplying \eqref{est td hs-4} by a suitably small constant and adding it to \eqref{est tG hs-4}, 
		\begin{align}\label{est tdG hs-4}
			(\mathcal{T}+t)^{3-\sigma}&\|(\widetilde{d},\widetilde{\mathbf{G}}) \|^2_{H^{s-4}}+\int_{0}^t(\mathcal{T}+\tau)^{3-\sigma}\| (\widetilde{d},\nabla\widetilde{\mathbf{G}})\|^2_{H^{s-4}}d\tau
			\\
			\lesssim& \mathcal{T}^{3-\sigma}\| (a_0,\mathbf{u}_0,\mathbf{B}_0)\|_{H^{s}}^2+E_3(t)+   (E_1(t)+E_3(t)+N_3(t))^{ 2}.\nonumber
		\end{align}
		
		In the following,  we establish the estimate for $\|\partial_{x_3}\mathbf{B} \|_{H^{s-4}}$. Applying the Leray projector $\mathbb{P}$ to the  momentum equation for $\mathbf{u}$ and taking the $\dot{H}^{s-4}$  inner product against $\partial_{x_3}\mathbf{B}$, we have
		\begin{align*}
			(\mathcal{T}+\tau)^{3-\sigma}&\|\partial_{x_3}\mathbf{B} \|^2_{\dot{H}^{s-4}}\\
			= &\int_{\mathbb{T}^3}(\mathcal{T}+\tau)^{3-\sigma}\nabla^{s-4}\mathbb{P}(\rho\partial_t\mathbf{u}+\rho \mathbf{u}\cdot \nabla \mathbf{u}-\mu \Delta \mathbf{u}-\mathbf{B}\cdot \nabla \mathbf{B} )\cdot\nabla^{s-4}\partial_{x_3}\mathbf{B}   dx\\
			=&\frac{d}{d\tau}\int_{\mathbb{T}^3}(\mathcal{T}+\tau)^{3-\sigma}\nabla^{s-4}\mathbb{P}(\rho\mathbf{u})\cdot\nabla^{s-4}\partial_{x_3}\mathbf{B}   dx+\sum_{i=1}^6\mathcal{N}_i^*.
		\end{align*}
		where the residual terms are defined as
		\begin{align*}
			\mathcal{N}^*_1=&-(3-\sigma)\int_{\mathbb{T}^3}(\mathcal{T}+\tau)^{2-\sigma}\nabla^{s-4}\mathbb{P}(\rho\mathbf{u})\cdot\nabla^{s-4}\partial_{x_3}\mathbf{B}   dx,\\
			\mathcal{N}_2^*=& -\int_{\mathbb{T}^3}(\mathcal{T}+\tau)^{3-\sigma}\nabla^{s-4}\mathbb{P}(\partial_{\tau}\rho\mathbf{u})\cdot\nabla^{s-4}\partial_{x_3}\mathbf{B}   dx, \\
			\mathcal{N}_3^*=&\int_{\mathbb{T}^3}(\mathcal{T}+\tau)^{3-\sigma}\nabla^{s-4}\mathbb{P}\partial_{x_3}(\rho\mathbf{u})\cdot\nabla^{s-4}\partial_t\mathbf{B}   dx, \\
			\mathcal{N}_4^*=&\int_{\mathbb{T}^3}(\mathcal{T}+\tau)^{3-\sigma}\nabla^{s-4}\mathbb{P}( \rho \mathbf{u}\cdot \nabla \mathbf{u})\cdot\nabla^{s-4}\partial_{x_3}\mathbf{B}   dx, \\
			\mathcal{N}_5^*=&-\int_{\mathbb{T}^3}(\mathcal{T}+\tau)^{3-\sigma} \nabla^{s-4}\mathbb{P} \mu\Delta \mathbf{u}\cdot\nabla^{s-4}\partial_{x_3}\mathbf{B}   dx, \\
			\mathcal{N}_6^*=&-\int_{\mathbb{T}^3}(\mathcal{T}+\tau)^{3-\sigma}\nabla^{s-4}\mathbb{P}( \mathbf{B}\cdot \nabla \mathbf{B} )\cdot\nabla^{s-4}\partial_{x_3}\mathbf{B}   dx.
		\end{align*}
		
		To estimate the terms $\mathcal{N}_i^*$ ($i=1,\dots,6$), we proceed in a manner analogous to the higher-order case,  with the derivative reduced from $s-2$ to $s-4$ and the time weight adapted accordingly. 
		
		For the first term, 
		\begin{align*}
			\mathcal{N}_1^*=&-(3-\sigma) \int_{\mathbb{T}^3}(\mathcal{T}+\tau)^{2-\sigma}\nabla^{s-4}\mathbb{P}(\rho\mathbf{u})\cdot\nabla^{s-4}\partial_{x_3}\mathbf{B}   dx\\
			\lesssim & \mathcal{T}^{-1} (\mathcal{T}+\tau)^{3-\sigma}(1+\| a\|_{H^{s-4}})\|\nabla \mathbf{u} \|_{H^{s-3}}\|\partial_{x_3}\mathbf{B} \|_{H^{s-4}}.
		\end{align*}
		
		Substituting the continuity equation into $\mathcal{N}_2^*$, we deduce
		\begin{align*}
			\mathcal{N}_2^*=&-\int_{\mathbb{T}^3}(\mathcal{T}+\tau)^{3-\sigma}\nabla^{s-4}\mathbb{P}(\partial_{\tau}\rho\mathbf{u})\cdot\nabla^{s-4}\partial_{x_3}\mathbf{B}   dx\\
			=& \int_{\mathbb{T}^3}(\mathcal{T}+\tau)^{3-\sigma}\nabla^{s-4}\mathbb{P}\big((\nabla \cdot \mathbf{u}+\mathbf{u}\cdot \nabla a+a\operatorname{div}\mathbf{u}) \mathbf{u}\big)\cdot\nabla^{s-4}\partial_{x_3}\mathbf{B}   dx\\
			\lesssim & (\mathcal{T}+\tau)^{3-\sigma}(\|\nabla \mathbf{u} \|_{H^{s-4}}\| \mathbf{u}\|_{H^{s-4}}+\|  a\|_{H^{s-3}}\|\mathbf{u} \|^2_{H^{s-4}}+\| a\|_{H^{s-4}}\|\mathbf{u} \|_{H^{s-4}}\|\operatorname{div}\mathbf{u} \|_{H^{s-4}})\\
			&\qquad\qquad\times\|\partial_{x_3}\mathbf{B} \|_{H^{s-4}}\\
			\lesssim & (\mathcal{T}+\tau)^{3-\sigma}(1+\|(a,\mathbf{u}) \|_{H^{s-3}})\|\nabla \mathbf{u} \|_{H^{s-4}}^2\|\partial_{x_3}\mathbf{B} \|_{H^{s-4}}.
		\end{align*}
		
		Integrating by parts and substituting the equation for $\mathbf{B}$ yields the bound for $\mathcal{N}_3^*$:
		\begin{align*}
			\mathcal{N}_3^*=&\int_{\mathbb{T}^3}(\mathcal{T}+\tau)^{3-\sigma}\nabla^{s-4}\mathbb{P}\partial_{x_3}(\rho\mathbf{u})\cdot\nabla^{s-4}\partial_t\mathbf{B}   dx\\
			=&\int_{\mathbb{T}^3} (\mathcal{T}+\tau)^{3-\sigma}\nabla^{s-4}\mathbb{P}\partial_{x_3}(\rho\mathbf{u})\cdot\nabla^{s-4} (\mathbf{B}\cdot \nabla \mathbf{u}-\mathbf{u}\cdot \nabla \mathbf{B}-\mathbf{B}\operatorname{div}\mathbf{u}-\mathbf{e}_3\operatorname{div}\mathbf{u}+\partial_{x_3}\mathbf{u} )   dx\\
			\lesssim & (\mathcal{T}+\tau)^{3-\sigma}(1+\|a \|_{H^{s-3}})\|\nabla \mathbf{u} \|_{H^{s-4}}(\|\mathbf{B} \|_{H^{s-4}}\| \nabla \mathbf{u}\|_{H^{s-4}}+\| \mathbf{u}\|_{H^{s-4}}\| \mathbf{B}\|_{H^{s-3}}\\
			&+\| \mathbf{B}\|_{H^{s-4}}\|\operatorname{div}\mathbf{u} \|_{H^{s-4}}+\|\operatorname{div}\mathbf{u} \|_{H^{s-4}}+\|\nabla \mathbf{u} \|_{H^{s-4}} )\\
			\lesssim &  (\mathcal{T}+\tau)^{3-\sigma}(1+\|(a,\mathbf{B}) \|_{H^{s-3}} )\|\nabla \mathbf{u} \|_{H^{s-4}}^2.
		\end{align*}
		
		The bounds for the convective term $\mathcal{N}_4^*$ and the diffusion term $\mathcal{N}_5^*$ follow similar steps as their higher-order counterparts, yielding
		\begin{align*}
			\mathcal{N}_4^*=&\int_{\mathbb{T}^3}(\mathcal{T}+\tau)^{3-\sigma}\nabla^{s-4}\mathbb{P}( \rho \mathbf{u}\cdot \nabla \mathbf{u})\cdot\nabla^{s-4}\partial_{x_3}\mathbf{B}   dx\\
			\lesssim  &(\mathcal{T}+\tau)^{3-\sigma}(1+\|a \|_{H^{s-4}})\| \mathbf{u}\|_{H^{s-4}}\|\nabla \mathbf{u} \|_{H^{s-4}}\|\partial_{x_3}\mathbf{B} \|_{H^{s-4}},
		\end{align*}
		and
		\begin{align*}
			\mathcal{N}_5^*=&-\int_{\mathbb{T}^3}(\mathcal{T}+\tau)^{3-\sigma}\nabla^{s-4}\mathbb{P} \mu\Delta \mathbf{u}\cdot\nabla^{s-4}\partial_{x_3}\mathbf{B}   dx\\
			\lesssim&  (\mathcal{T}+\tau)^{3-\sigma}\|\nabla \mathbf{u} \|_{H^{s-3}}\|\partial_{x_3}\mathbf{B} \|_{H^{s-4}}.
		\end{align*}
		
		Finally, applying the same horizontal-vertical decomposition to the Lorentz force term gives
		\begin{align*}
			\mathcal{N}_6^*=&-\int_{\mathbb{T}^3}(\mathcal{T}+\tau)^{3-\sigma}\nabla^{s-4}\mathbb{P}( \mathbf{B}\cdot \nabla \mathbf{B} )\cdot\nabla^{s-4}\partial_{x_3}\mathbf{B}   dx \\
			=&-\int_{\mathbb{T}^3}(\mathcal{T}+\tau)^{3-\sigma}\nabla^{s-4}\mathbb{P}( \mathbf{B}_h\cdot \nabla_h \mathbf{B}+B_3\partial_{x_3}\mathbf{B} )\cdot\nabla^{s-4}\partial_{x_3}\mathbf{B}   dx \\
			\lesssim& (\mathcal{T}+\tau)^{3-\sigma}(\|\mathbf{B}_h \|_{H^{s-4}}\| \mathbf{B}\|_{H^{s-3}}+\|B_3 \|_{H^{s-4}}\| \partial_{x_3}\mathbf{B}\|_{H^{s-4}})\|\partial_{x_3}\mathbf{B} \|_{H^{s-4}}\\
			\lesssim & (\mathcal{T}+\tau)^{3-\sigma}\| \mathbf{B}\|_{H^{s-3}}\|\partial_{x_3}\mathbf{B} \|_{H^{s-4}}^2.
		\end{align*}
		
		Collecting the estimates for $\mathcal{N}_1^*$ through $\mathcal{N}_6^*$, absorbing the magnetic terms via Young's inequality, and integrating over time, we conclude that
		\begin{align*}
			\int_{0}^t(\mathcal{T}+\tau)^{3-\sigma}&\|\partial_{x_3}\mathbf{B} \|^2_{\dot{H}^{s-4}}d\tau	\lesssim \mathcal{T}^{3-\sigma}\| (a_0,\mathbf{u}_0,\mathbf{B}_0)\|_{H^{s}}^2+E_3(t)+  (E_1(t)+E_3(t)+N_3(t) )^{\frac32}.
		\end{align*}
		Combining with the bound  \eqref{est tdG hs-4}, we conclude the estimate for $N_3(t)$:
		\begin{align*}
			N_3(t)=&\sup_{\tau\in[0,t]}(\mathcal{T}+\tau)^{3-\sigma} \| (\widetilde{d},\widetilde{\mathbf{G}})\|^2_{H^{s-4}}+\int_{0}^{t}(\mathcal{T}+\tau)^{3-\sigma}\| (\nabla \widetilde{\mathbf{G}},\widetilde{d},\partial_{x_3}\mathbf{B})(\tau,\cdot)\|^2_{H^{s-4}}d\tau\\
			\lesssim&  \mathcal{T}^{3-\sigma}\| (a_0,\mathbf{u}_0,\mathbf{B}_0)\|_{H^{s}}^2+ E_3(t)+(E_1(t)+E_3(t)+N_3(t) )^{\frac32}.
		\end{align*}
		Together with \eqref{est bE3 s-3} and \eqref{est tE3 s-3}, this completes the proof.
	\end{proof}

	\vskip .2in 
	\section{Closure of the a priori estimates and proof of Theorem~\ref{thm1}}
	\label{sec:bootstrap-closure}

	This section combines these estimates and completes the bootstrap. In particular, interpolation supplies the omitted intermediate energies, and the large shift $\mathcal T$ absorbs the linear average couplings highlighted in the introduction. Throughout this section, $\sigma\in(0,1/2)$ is fixed. The parameter $\mathcal T\geq 1$ will be fixed below, depending only on $\sigma$, $\mu$, $\lambda$, the pressure law, and the constants in the preceding estimates. Once $\mathcal T$ is fixed, all constants are independent of $t$ and of the size of the initial data. Set
	\[
	\varepsilon_0:=\|(a_0,\mathbf u_0,\mathbf B_0)\|_{H^s}.
	\]
	
	We first recover the intermediate-order energies $E_2(t)$ and $N_2(t)$ via Sobolev interpolation and the Cauchy--Schwarz inequality:
	\begin{equation}\label{eq:E2N2-interpolation}
		E_2(t)\leq C E_1(t)^{1/2}E_3(t)^{1/2},
		\qquad
		N_2(t)\leq C N_1(t)^{1/2}N_3(t)^{1/2}.
	\end{equation}
	Consequently, the total energy is bounded by the remaining terms:
	\begin{equation}\label{eq:total-reduced-equivalence}
		E_{total}(t)\leq C( E_0(t)+E_1(t)+E_3(t)+N_1(t)+N_3(t)).
	\end{equation}
	
	Under the smallness assumption required in the preceding \textit{a priori} estimates, Lemmas~\ref{est E0}, \ref{est bE1}, \ref{est tE1}, and \ref{est N1}, together with the estimates for $E_3$ and $N_3$, imply
	\begin{align}
		E_0(t)
		&\leq C\mathcal{T}^{-\sigma}\varepsilon^2_0+CE_{total}(t)^{3/2},\label{eq:closure-E0}\\
		E_1(t)
		&\leq C\mathcal{T}^{1-\sigma}\varepsilon^2_0
		+C\mathcal T^{\frac{2\sigma-1}{2}}E_{total}(t)
		+C E_0(t)^{1/2}N_1(t)^{1/2}
		+CE_{total}(t)^{3/2},\label{eq:closure-E1}\\
		N_1(t)
		&\leq C\mathcal{T}^{1-\sigma}\varepsilon^2_0
		+C E_1(t)+CE_{total}(t)^{3/2},\label{eq:closure-N1}\\
		E_3(t)
		&\leq C\mathcal{T}^{3-\sigma}\varepsilon^2_0
		+C\mathcal T^{\frac{2\sigma-1}{2}}E_{total}(t)
		+C E_0(t)^{1/4}N_3(t)^{3/4}
		+CE_{total}(t)^{3/2},\label{eq:closure-E3}\\
		N_3(t)
		&\leq C\mathcal{T}^{3-\sigma}\varepsilon^2_0
		+C E_3(t)+CE_{total}(t)^{3/2}.\label{eq:closure-N3}
	\end{align}
	These inequalities cannot simply be summed, since $N_1(t)$ and $N_3(t)$ are bounded in terms of $E_1(t)$ and $E_3(t)$, which in turn involve $N_1(t)$ and $N_3(t)$. We therefore proceed as follows. Let $\eta\in(0,1)$ be a constant to be fixed, depending only on the constants $C$ appearing in \eqref{eq:closure-E0}--\eqref{eq:closure-N3}. Multiplying \eqref{eq:closure-N1} by $\eta$ and adding the result to \eqref{eq:closure-E1}, the term $C\eta E_1(t)$ produced on the right-hand side is absorbed by the term $E_1(t)$ on the left, provided $\eta\leq 1/(2C)$. The mixed term is handled by Young's inequality,
	\begin{align*}
		C E_0(t)^{1/2}N_1(t)^{1/2}\leq \frac{\eta}{2}N_1(t)+\frac{C^2}{2\eta}E_0(t),
	\end{align*}
	and the resulting multiple of $E_0(t)$ is estimated by \eqref{eq:closure-E0}, which contributes $C\eta^{-1}\mathcal{T}^{-\sigma}\varepsilon_0^2$ together with a cubic remainder. The pair \eqref{eq:closure-E3}, \eqref{eq:closure-N3} is treated in exactly the same way, using $CE_0(t)^{1/4}N_3(t)^{3/4}\leq \frac{\eta}{2}N_3(t)+C_\eta E_0(t)$. Finally, the terms $C\mathcal{T}^{(2\sigma-1)/2}E_{total}(t)$ in \eqref{eq:closure-E1} and \eqref{eq:closure-E3} are absorbed into the left-hand side by choosing $\mathcal{T}$ so large that $C\mathcal{T}^{(2\sigma-1)/2}\leq 1/4$, which is possible because $\sigma<1/2$. The order of the choices is thus: first $\eta$, in terms of the constants $C$; then $\mathcal{T}$, in terms of $\eta$, $\sigma$, $\mu$ and $\lambda$; and only then $\varepsilon_0$, in terms of $\mathcal{T}$.
	Combining these steps with \eqref{eq:total-reduced-equivalence}, we deduce that
	\begin{equation}\label{eq:closed-Etotal}
		E_{total}(t)
		\leq C\mathcal{T}^{3-\sigma}\varepsilon^2_0
		+CE_{total}(t)^{3/2}.
	\end{equation}
	A standard bootstrap argument then yields the global uniform bound for $E_{total}(t)$. The desired decay rates follow naturally from the definition of the time-weighted energy functional $E_{total}(t)$ and the equivalence of $\mathcal T+t$ and $1+t$ for the fixed $\mathcal T$.
	This completes the proof of Theorem~\ref{thm1}.
	
	\section*{Acknowledgment}
	The work of L.-A. Li was supported by the National Natural Science Foundation of China (No. 12501295), Beijing Natural Science Foundation (No. 1254045) and the Fundamental Research Funds for the Central Universities (No. 2243100008). J. Wu was partially supported by the National Science Foundation of the United States under Grants DMS-2104682 and DMS-2309748. X. Xu was partially supported by the National Natural Science Foundation of China (grants 12571244, 12171040) and the National Key R\&D Program of China (grant 2020YFA0712900).

	\bibliographystyle{abbrv}
	\bibliography{3DcompMHD.bib}
\end{document}